\documentclass[12pt]{article}

\usepackage[a4paper, textwidth = 16cm, textheight = 22.5cm, top = 3cm]{geometry}
\usepackage[colorlinks,citecolor=blue,urlcolor=blue,breaklinks]{hyperref}
\usepackage{amsmath}
\usepackage{xcolor,latexsym,amssymb,amsfonts,amsthm,bm,bbm,enumitem,xcolor,soul,graphicx, dsfont,enumitem,caption, comment,url,xspace}
\usepackage[normalem]{ulem}
\usepackage[ruled, lined, linesnumbered]{algorithm2e}
\usepackage{multirow}
\usepackage[round]{natbib}
\usepackage{booktabs}
\usepackage{subcaption}
\usepackage{microtype}
\usepackage{cleveref}

\usepackage[utf8]{inputenc}   % Or \usepackage{fontspec} for XeLaTeX
\usepackage[british]{babel}   % Babel, replace with your language
\usepackage{cmap} % correctly copyable characters in pdf
\usepackage[T1]{fontenc}      % Font encoding
\usepackage{lmodern}     % lmodern font
\usepackage{anyfontsize}

\SetKwInput{KwSet}{Set}

\newtheorem{thm}{Theorem}
\newtheorem{prop}[thm]{Proposition}
\newtheorem{lemma}[thm]{Lemma}
\newtheorem{cor}[thm]{Corollary}

\newtheorem{condition}{Condition}

\DeclareMathAlphabet{\mathbbb}{U}{bbold}{m}{n}
\DeclareMathOperator*{\argmin}{arg\,min}
\DeclareMathOperator*{\argmax}{arg\,max}
\DeclareMathOperator*{\soft}{soft}
\DeclareMathOperator*{\hard}{hard}

\DeclareMathOperator{\diag}{diag}
\DeclareMathOperator{\Beta}{Beta}

\newcommand{\teal}[1]{{\color{teal}#1}}

\newcommand{\charcoal}{\textsf{charcoal}\xspace}
\newcommand{\charcoalproj}{$\mathsf{charcoal}_{\mathsf{proj}}$\xspace}
\newcommand{\charcoallasso}{$\mathsf{charcoal}_{\mathsf{lasso}}$\xspace}

\title{Sparse change detection in high-dimensional linear regression}
\author{Fengnan Gao\footnote{Fudan University and Shanghai Center for Mathematical Sciences. Email: \url{fngao@fudan.edu.cn}} \; and\; Tengyao Wang\footnote{London School of Economics and Political Science. Email: \url{t.wang59@lse.ac.uk}}}
\date{(\today)}

\begin{document}
\maketitle
\begin{abstract}
  We introduce a new methodology `\charcoal' for estimating the location of sparse changes in high-dimensional linear regression coefficients, without assuming that those coefficients are individually sparse. The procedure works by constructing different sketches (projections) of the design matrix at each time point, where consecutive projection matrices differ in sign in exactly one column. The sequence of sketched design matrices is then compared against a single sketched response vector to form a sequence of test statistics whose behaviour shows a surprising link to the well-known CUSUM statistics of univariate changepoint analysis. The procedure is computationally attractive, and strong theoretical guarantees are derived for its estimation accuracy. Simulations confirm that our methods perform well in extensive settings, and a real-world application to a large single-cell RNA sequencing dataset showcases the practical relevance.
\end{abstract}

\begin{comment}
% text abstract ready for copying and pasting
  We introduce a new methodology 'charcoal' for estimating the location of sparse changes in high-dimensional linear regression coefficients, without assuming that those coefficients are individually sparse. The procedure works by constructing different sketches (projections) of the design matrix at each time point, where consecutive projection matrices differ in sign in exactly one column. The sequence of sketched design matrices is then compared against a single sketched response vector to form a sequence of test statistics whose behaviour shows a surprising link to the well-known CUSUM statistics of univariate changepoint analysis. The procedure is computationally attractive, and strong theoretical guarantees are derived for its estimation accuracy. Simulations confirm that our methods perform well in extensive settings, and a real-world application to a large single-cell RNA sequencing dataset showcases the practical relevance.
\end{comment}

\section{Introduction}
\label{Sec:Intro}

The past twenty years have witnessed rapid development of statistical
methodologies for high-dimensional data sets, where the number of variables of interest is often of comparable or even larger order of magnitude than the number of observations available. The most prominent example, perhaps, is the line of work on sparse linear regression, which started from the seminal work of \citet{tibshirani1996regression}, and was developed and generalized subsequently by many others (see, e.g. \citet{fan2010selective, buhlmann2011statistics} for a general overview of this area). In many of these works, the primary focus was on how to exploit the sparsity of the regression coefficients for their successful estimation, and to achieve this, a homogeneous data generating mechanism was often assumed for simplicity of analysis.

However, it is usually unrealistic in large, high-dimensional data sets to
assume that the data generating mechanism holds true throughout. In fact,
heterogeneity is the norm rather than exception in Big Data applications.
Several attempts have been made to handle data heterogeneity in
high-dimensional linear models.  For instance,  \citet{stadler2010l1,yin2018learning,krishnamurthy2019sample} considered the problem of learning mixture of sparse linear regression, where the regression coefficient vector is sampled from a small set of sparse regression coefficients.  When observations have a temporal structure, one common way to handle heterogeneity is to break the sequence of observations into shorter time segments on which the data are more homogeneous.  This line of thinking is the driving force behind the recent revival of interest in changepoint analysis, which dates back to the early work of \citet{Page1955}, but has seen growing interest recently in high-dimensional settings, see e.g., \citet{levy2009detection,Bai2010, Zhangetal2010, HorvathHuskova2012, Cho_2014, Jirak2015, Cho_2016, wang2018high, enikeeva2021change, liu2021minimax, follain2022high, chen2022high}. 

In our linear regression setting, such a changepoint setup means that the sequence of regression coefficient vectors has a piecewise constant structure. More precisely, for an unknown sequence of changepoints $0 < z_1 <\cdots < z_{\nu} < n$ (for notational convenience, we also define $z_0 := 0$ and $z_{\nu+1}:=n$) and an unknown sequence of a regression coefficients $(\beta^{(r)}: 1\leq r\leq \nu+1)$, we assume that the data $(x_t, y_t) \in \mathbb{R}^{p}\times \mathbb{R}$, $1\leq t\leq n$ are generated according to the following model 
\begin{equation}
	\label{eqn:y-x-z}
  y_{t} = x_{t}^\top \beta_{t}+ \epsilon_{t}, \quad \text{where $\beta_{t} = \beta^{(r)}$ for $z_{r-1} < {t} \leq z_{r}$, $1\leq r \leq \nu+1$},
\end{equation}
and $(\epsilon_t)_{1\le t \le n}$ are the observational errors distributed as $N(0, \sigma^2 I_n)$ conditionally independent of $(x_t)_{1 \le t \le n}$.
The goal is to locate the changepoints $z_1,\ldots,z_{\nu}$ upon observing the response vector $Y = (y_1, \dots, y_n)^\top$ and the design matrix $X= (x_1, \dots, x_n)^\top$.

Classically, when the dimension $p$ is far smaller than $n$, \citet{bai1997estimation, bai1998estimating, julious2001inference} showed that a least-square-based approach works well in the above changepoint problem, which is equivalent to maximum-likelihood estimations under Gaussianity assumptions. Specifically, for a given $\nu$, the maximum likelihood estimator finds the optimal partition of $\{1,\ldots,n\}$ into $\nu+1$ segments such that residual sum of squares from the least-square fit within each segment is minimized. The least-square (maximum-likelihood) fit from different choices of $\nu$ can then be compared using for instance the Bayesian Information Criterion (BIC) to choose the best $\nu$, which is often solved algorithmically via dynamic programming.

In the high-dimensional setting, the above maximum-likelihood/least-square approach no longer works.  Several works have appeared to analyse such regression changepoint problems in the high-dimensional context, see for instance \citet{RinaldoWangWen2021, wang2021optimal,lee2016lasso, kaul2019efficient} and references therein.
However, in addition to the modelling assumption in~\eqref{eqn:y-x-z}, these works also impose the additional assumption that all regression coefficients $(\beta^{(r)}:1\leq r\leq \nu+1)$ are individually sparse. Given a hypothesized set of changepoints, this additional assumption allows them to form estimators of $\beta^{(r)}$, $1\leq r\leq \nu+1$, which are in turn used to form goodness-of-fit statistics for the set of hypothesized changepoints. 
% \sout{they mostly impose strict assumptions on the regression coefficients being sparse on the entire time horizon, whence the problem could be in effect reduced to a low-dimensional one via dimension reduction techniques such as Lasso. 
% This is also the usual way out in the contemporary big data era, where the high-dimensionality is ubiquitous and statisticians (and data scientists) rely on such sparsity conditions to enforce a certain extent of estimability on the otherwise unsolvable regression coefficients. 
% Nevertheless, such sparsity assumption on all regression coefficients could be theoretically an overkill on one hand, and is often possibly violated in practice on the other. }

A major difference between this work and the aforementioned existing line of works is that we do not assume that the regression coefficients within different stationary segments are individually sparse. 
Instead, we make the less stringent assumption that the difference in the regression coefficient vectors before and after each change, i.e., $\theta^{(r)}:=(\beta^{(r+1)} - \beta^{(r)})/2$, are sparse in the sense that $\|\theta^{(r)}\|_0\leq k$, for $r=1,\ldots, \nu$.  We would argue that this is a more natural assumption, since it is the change in the regression coefficients, rather than the pre- and post-change coefficients themselves, that is the quantity of interest in this statistical problem. Practically, the assumption that all regression coefficients are sparse can be violated in applications.  For instance, \citet{kraft2009genetic} argued that in genetic studies,  ``many, rather than few, variant risk alleles are responsible for the majority of the inherited risk of each common disease'', leading to non-sparse regression coefficients. However, in such examples, the task of detecting sparse changes in these regression coefficients over time can still be of interest in, e.g., identifying different development stages in gene regulatory networks in species \citep{HATLEBERG202139}.  Furthermore, our `sparsity-in-change' assumption is also more in line with the assumptions made in the high-dimensional change-in-mean problem \citep[see, e.g.][]{Cho_2014, Jirak2015, wang2018high}, where the pre- and post-change mean vectors are regarded as nuisance parameters and sparsity assumptions only need to be placed on vectors of changes for successful detection and localization of the changepoints.

Allowing for dense pre- and post-change regression coefficients makes the changepoint estimation problem considerably more challenging. In particular, the general strategy employed by existing works on high-dimensional regression changepoints that relies on forming good estimators of $(\beta^{(r)}: r\in\{1,\ldots,\nu\})$ will unlikely be successful here.  Our first contribution in this paper is to propose a novel methodology, which we call \charcoal (\underline{cha}ngepoint in \underline{r}egression via a \underline{co}mplementary-sketching \underline{al}gorithm), and works by forming a projected response vector and a sequence of projected design matrices to eliminate the dense nuisance parameter.  For simplicity of exposition, we consider the single changepoint scenario, where $\beta_i = \beta^{(1)}\mathbbb{1}_{\{i\leq z_1\}} + \beta^{(2)}\mathbbb{1}_{\{i > z_1\}}$. 
Under the hypothesis that the true change takes place at time $t$, we have
\begin{equation}
  \label{eqn:one-cp-t}
  \begin{cases}
    Y_{(0,t]} = X_{(0,t]} \beta^{(1)} + \epsilon_{(0,t]},\\
    Y_{(t,n]} = X_{(t,n]} \beta^{(2)} + \epsilon_{(t,n]},
  \end{cases}
\end{equation}
where the subscript $(0,t]$ indicates the concatenation of relating quantities on $(1,\dots,t)$ and $(t,n]$ that of $(t+1,\dots,n)$. 
We may think of \eqref{eqn:one-cp-t} as a two-sample problem with different regression coefficients before and after $t$. 
We assume throughout the paper that $n > p$ for otherwise it is impossible to estimate the change when both pre- and post-change parameters are dense (see further discussion at the beginning of Section~\ref{Sec:Theory}).  By invoking the complementary sketching method of \citet{gao2021twosample}, we can find matrices $A_{(0,t]}\in\mathbb{R}^{t\times (n-p)}$ and $A_{(t,n]}\in\mathbb{R}^{(n-t)\times (n-p)}$ such that $(A_{(0,t]}^\top, A_{(t,n]}^\top)^\top$ has orthogonal columns spanning the orthogonal complement of the range of $X$. By forming the projected design matrix $W_t := A_{(0,t]}^\top X_{(0,t]} - A_{(t,n]}^\top X_{(t,n]}$ and the projected response $Z := A_{(0,t]}^\top Y_{(0,t]} + A_{(t,n]}^\top Y_{(t,n]}$, we can eliminate the possibly dense nuisance parameter $\zeta := (\beta^{(2)} + \beta^{(1)})/2$ and conduct tests on $\theta^{(1)} = (\beta^{(2)} - \beta^{(1)})/2$ is zero against that it is non-zero and sparse.
% \teal{\sout{The complementary sketching works by first finding a matrix $A$ whose columns are orthogonal and span the orthogonal complement of the column space of $(x_1,\dots,x_n)^\top$ in \(\mathbb{R}^n\), then forming two sketching matrices on the response and the design respectively, which differ only in sign in rows corresponding to the second sample. 
% After applying the sketching matrices, the nuisance disappears and the results are aggregated to form a statistic, whose size measures how different the two samples are vis-à-vis the regression coefficients. 
% The decision on whether to reject the null hypothesis that the two samples have the same regression coefficients is determined by comparing the test statistic to a certain explicit threshold. }}

In light of the true changepoint at time $z_1$, the hypothesized model \eqref{eqn:one-cp-t} is only correctly specified when $t = z_1$.  The further $t$ is away from $z_1$, the less different the two samples $(X_{(0,t]}, Y_{(0,t]})$ and $(X_{(t,n]}, Y_{(t,n]})$ are, since one of the samples will be further contaminated by the data points assigned to the wrong segment by the hypothesized changepoint.  Intuitively, we would expect the aforementioned two-sample test statistics to peak around $t=z_1$, which can thus be used to estimate the location of the single changepoint.  Unfortunately, while good for testing, these two-sample test statistics have variances too large for accurate changepoint localization.  Nevertheless, the general idea of using complementary sketching to eliminate nuisance parameters is valid.  We introduce in Section~\ref{Sec:Method} several alternative statistics based on the sketched design $W_t$ and response $Z$ that do lead to good changepoint estimation performance. In particular, we will show in Section~\ref{Sec:Theory} that a variant of the \charcoal procedure achieves a rate of convergence of order $\sqrt{k/(n\|\theta^{(1)}\|_2^2)}$, up to logarithmic factors. In the course of investigating the theoretical properties, we have developed new results in understanding the asymptotic behaviour of the sketched design matrices by generalizing existing matrix-variate Beta distribution to rank-deficient cases (Lemma~\ref{lemma:generalised-beta} and Corollary~\ref{cor:generalise-betap}) and extended sub-Gaussian bounds of Beta random variables to the matrix variate case (Lemma~\ref{lem:S0t-S0z}), both of which may be of independent interest. 

\subsection{Outline of the paper}
\label{sec:intro-outline}
We present the methodology in detail in Section~\ref{Sec:Method}, including several algorithms that all use the complementary sketching idea.
Section~\ref{Sec:Theory} provides theoretical performance guarantees to the slight variants of those proposed in Section~\ref{Sec:Method}.
In the first part of Section~\ref{Sec:Numerical}, we conduct numerical experiments on the \charcoal methodology over a comprehensive range of settings for both single and multiple changepoint estimation tasks and compare our methods with other changepoint localization methods in the high-dimensional linear regression context. 
In the second part of Section~\ref{Sec:Numerical}, we study a real data example to identify changes for each gene in terms of its interaction with other genes in the gene regulatory network across various development stages of T cells.
Section~\ref{Sec:Proof} collects the proofs of the main results while we gather the proofs of the ancillary results in Section~\ref{Sec:Ancillary}.

\subsection{Notation}
\label{sec:intro-notation}
For a positive integer $p$, $[p] = \{ 1, \dots, p\}$ consists of all positive integers not exceeding $p$.
For vector $v = (v_1,\dots, v_p)^\top$, \( \diag(v)\) is a $p\times p$ matrix such that $(\diag(v))_{i,j} = \mathbbb{1}_{\{i=j\}} v_i$ for $i,j \in[p]$. 
We follow the usual definitions of $\|v\|_0 = \sum_{i\in [p]} \mathbbb{1}_{\{v_i \neq 0\}}$, \( \|v\|_2 = (\sum_{i\in[p]} v_i^2)^{1/2}\), \( \|v\|_1 = \sum_{i\in[p]} |v_i|\) and \( \|v\|_\infty = \max_{i\in[p]} |v_i|\). 

Given a matrix $A \in \mathbb{R}^{n\times m}$, we make it a convention that $A = (A_{i,j})_{i \in [n], j \in [m]} = [ A_1 \mid \dots\mid A_m] = ( a_1, \dots, a_n)^\top$, where $a_i$ is the transpose of the $i$th row of $A$ and $A_j$ is the $j$th column of $A$.
Given any set $S\subseteq \mathbb{R}$, we write $A_S$ to be the submatrix of $A$ with row indices in $S$. For instance, given positive integers $s,t$ such that $1\leq s < t\leq n$, $A_{(s,t]} := (a_{s+1}, \dots, a_t)^\top$.
We define the usual norms for $A$ as follows \( \|A\|_{\mathrm{op}} := \sup_{v\in \mathbb{R}^m: \|v\|_2 = 1} \|A v\|_2 \) and \( \|A\|_{\max} := \max_{i\in [n], j \in [m]} |A_{i,j}|\). 
Assuming $n=m$ in $A$, $\diag(A)$ is an $n \times n$ matrix such that $(\diag(A))_{i,j} := \mathbbb{1}_{i=j} A_{i,j}$ for $i \in [n]$ and $\mathrm{tr}(A) := \sum_{i\in[n]} A_{i,i}$.

For $n \ge m$, $\mathbb{O}^{n\times m}: = \{O\in \mathbb{R}^{n\times m}: O^\top O = I_m \}$.   
We define \( \mathcal{S}^{p-1} : = \{ v\in\mathbb{R}^p: \|v\|_2 = 1 \}  \) and the $k$-sparse unit ball as $B_0(k):=\{v\in\mathbb{R}^p: \|v\|_2\leq 1, \|v\|_0\leq k\}$.

\section{Methodology}
\label{Sec:Method}
In this section, we describe in detail our \charcoal algorithm for identifying the changepoints in the problem setup of~\eqref{eqn:y-x-z}.  

\subsection{Single changepoint estimation}
We start by focusing on the setting of a single changepoint estimation, i.e., $\nu=1$, which captures the essence of the difficulty of this problem.  For simplicity, we denote $z := z_1$ for the location of the only changepoint and write $m:=n-p$.  The main idea is to use data-driven projections to sketch the design matrix and the response vector to eliminate the effect of the nuisance parameters. 

Recall the data generating model \eqref{eqn:y-x-z}.  At each time point $t\in[n-1]$, we perform a two-sample test for the equality of regression coefficients before and after $t$ using data points $(x_i, y_i)_{i=1}^t$ and $(x_i, y_i)_{i=t+1}^n$ respectively.  Motivated by \cite{gao2021twosample}, this can be achieved by constructing a matrix $A \in \mathbb{O}^{n\times m}$ whose columns span the orthogonal complement of the column space of $X$. We then define for any $t \in [n-1]$
\begin{align*}
	W_t & := A_{(0,t]}^\top X_{(0,t]} - A_{(t,n]}^\top X_{(t,n]} = 2A_{(0,t]}^\top X_{(0,t]} \in \mathbb{R}^{m\times p} \\
	Z   & := A_{(0,t]}^\top Y_{(0,t]} + A_{(t,n]}^\top Y_{(t,n]} = A^\top Y \in \mathbb{R}^m.
\end{align*}
We define $\theta = (\beta^{(1)}-\beta^{(2)})/2$, $\zeta = (\beta^{(1)}+\beta^{(2)})/2$ and $\xi = A^\top \epsilon \sim N_m(0, \sigma^2 I_m)$.
By the model construction, we have
\begin{equation}
	\begin{aligned}
		Z & =A^\top_{(0,z]} Y_{(0,z]} + A^\top_{(z,n]} Y_{(z,n]} = A^\top_{(0,z]} (X_{(0,z]} \beta^{(1)} + \epsilon_{(0,z]}) + A^\top_{(z,n]} (X_{(z,n]} \beta^{(2)} + \epsilon_{(z,n]}) \\
		  & = A_{(0,z]}^\top X_{(0,z]} (\theta + \zeta)  - A_{(z,n]}^\top X_{(z,n]} ( \theta - \zeta) + \xi =
		W_z \theta + \xi,
	\end{aligned}
	\label{Eqn:Z-Wz}
\end{equation}
whence we have eliminated the nuisance parameter $\zeta$, and obtain the sketched data in the form of $(Z, (W_t)_{t\in[n-1]})$.
By \eqref{Eqn:Z-Wz} and the sparsity assumption on $\theta$, $Z$ can be approximated by a sparse linear combination of the columns of $W_z$.
Therefore, the changepoint localization problem is reduced to finding $t$ such that $W_t$ forms a `best' sparse linear approximation to $Z$.

%We discuss in the following several ways to achieve this.
As mentioned in the introduction, a naive way to achieve this would be based on the two-sample test statistics introduced in \citet{gao2021twosample}. Specifically, let $Q = (Q_1,\ldots,Q_{n-1})^\top$ be defined such that
\[
	Q_t := \{\diag(W_t^\top W_t)\}^{-1/2} W_t^\top Z.
\]
We view $Q_t$ as the vector of the correlations between columns of $W_t$ and $Z$, where we naturally seek to find the time point $t$ such that such correlations are as large as possible. To take into account of possible observational errors, we first remove small entries of $Q_t$ via an entrywise hard-thresholding operation $\hard(Q_t, \lambda)$ for $\hard(v, \lambda): (v_i)_{i=1}^p \mapsto (v_i\mathbbb{1}_{\{|v_i|\geq \lambda\}})_{i=1}^p$, where the threshold level $\lambda$ is a tuning parameter. This allows us to estimate the location of the changepoint via $\hat z^{\hard} := \argmax_{t\in[n-1]} \|\hard(Q_t, \lambda)\|_2$. Note that $\|\hard(Q_t, \lambda)\|_2$ is the statistic from \citet{gao2021twosample} to test whether the two samples $(X_{(0,t]}, Y_{(0,t]})$ and $(X_{(t,n]}, Y_{(t,n]})$ have the same regression coefficient against the alternative that there is a sparse difference. 
% Conceptually, such a two-sample testing problem is only correctly specified when $t=z$, as otherwise one of the samples is generated by a mixture of both pre- and post-change coefficients.
As is argued before, we expect that two-sample testing statistics gives the strongest signal against the null of no change at $t = z$ --- the only point where the two-sample problem is correctly specified. 

However, the changepoint estimator $\hat z^{\hard}$ is less than ideal in practice, as the discontinuity of the hard-thresholding function creates large variabilities in the test statistics. Moreover, the theoretical guarantees given in \citet{gao2021twosample} becomes increasingly inapplicable for test statistics away from the true changepoint as one of the two samples contains a mixture of data both before and after the change. Coupled with the fact that $W_t$ has large variance when $t$ is close to the boundary, $Q_t$ may have a large number of entries above the hard-thresholding level $\lambda = 2\sqrt{\log p}$ as  recommended in \citet{gao2021twosample}.  Empirically, this is evidenced by high variance of the test statistics near the two endpoints of the interval for changepoint detection, as shown in Figure~\ref{Fig:0}. Quite often, this boundary effect may overwhelm the main signal near the true changepoint, leading to a spurious changepoint being estimated near the boundary. One way to alleviate the instability problem of $\hat z^{\hard}$ is to replace the hard-thresholding in $\hat{z}^{\hard}$ by a soft-thresholding operation on each entry of $Q$.  The changepoint is then estimated by $\hat{z}^{\mathrm{soft}} = \argmax_{t\in[n-1]} \| \soft(Q_t,\lambda)\|_2$, where $\mathrm{soft}(v, \lambda): (v_i)_{i=1}^p \mapsto (\mathrm{sign}{(v_i)} \max(|v_i| - \lambda,0 ))_{i=1}^p$ with a tuning parameter $\lambda$.
The continuity of the soft-thresholding function reduces the variance in the test statistics, and in the ensuing changepoint estimator.
However, as also shown in Figure~\ref{Fig:0}, the sequence of test statistics $(\| \soft(Q_t,\lambda)\|_2)_{t\in[n-1]}$ could still exhibit undesirably large, although less so than $(\| \hard(Q_t,\lambda)\|_2)_{t\in[n-1]}$, variations when $t$ is close to the boundary.  

\begin{figure}[htbp]
\includegraphics[width=0.95\textwidth]{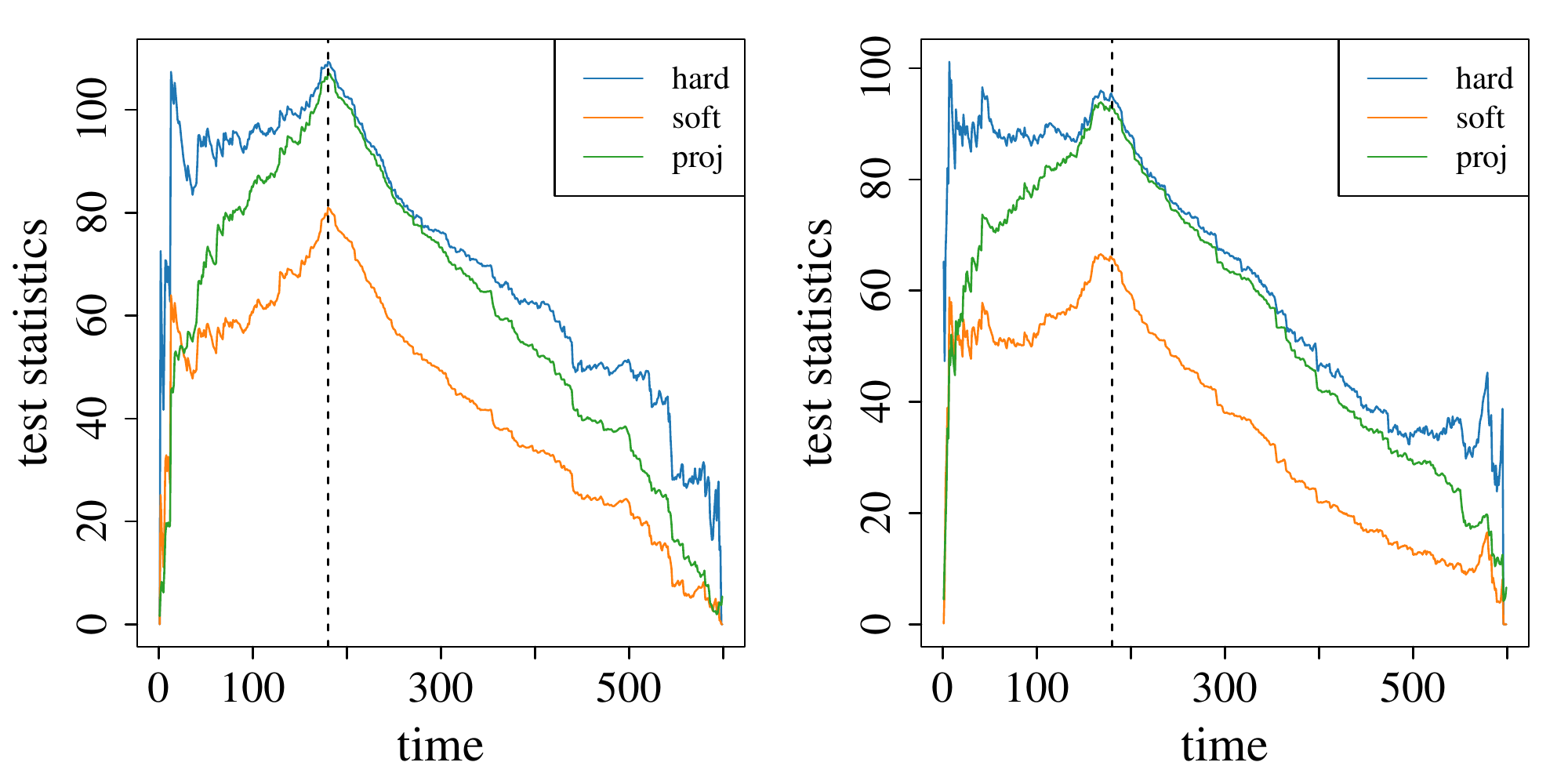}
\caption{\label{Fig:0} Visualization of different test statistics for changepoint localization.  We compare the hard-thresholded test statistics $\|\hard(Q_t, \lambda)\|_2$ used to estimate $\hat{z}^{\hard}$ (denoted by \texttt{hard}), the soft-thresholded test statistics $\|\mathrm{soft}(Q_t, \lambda)\|_2$ used to estimate $\hat{z}^{\mathrm{soft}}$ (denoted by \texttt{soft}) and the projected statistics $|\hat v^\top Q_t|$ in Algorithm~\ref{Algo:Estimation2} (denoted by \texttt{proj}) over two random realizations. Here, $n=600$, $p=200$, $\|\beta_1-\beta_2\|_0=10$, $\|\beta_1-\beta_2\|_2=8$, and the true change takes place at $z=180$, as indicated by the dashed lines.  In both panels, we observe that both $\|\hard(Q_t, \lambda)\|_2$ and $\|\mathrm{soft}(Q_t, \lambda)\|_2$ exhibit relatively strong boundary effect.}%  Though all three estimates are correct in the left panel, but it is possible that with slightly worse luck, the boundary effect would yield totally wrong estimates for $\hat{z}^{\hard}$ and $\hat{z}^{\mathrm{soft}}$, as shown in the right panel. }
\end{figure}

To avoid such boundary effect, we propose instead to aggregate the test statistics $(Q_t)_t$ via a projection-based approach. The key insight here is that, away from the boundary, the matrix $Q = (Q_1,\ldots,Q_{n-1})$ can be well-approximated by a rank-one matrix whose leading left singular vector is proportional to $\theta$. Hence, by first estimating $\theta/\|\theta\|_2$ via the leading left singular vector of $\mathrm{soft}(Q,\lambda)$, we can aggregate each vector of correlation $Q_t$ along the direction of $\hat v$ and estimate the changepoint by $\hat z:= \argmax_{t\in [\alpha n, (1-\alpha)n]} |\hat v^\top Q_t|$. This approach is summarized in Algorithm~\ref{Algo:Estimation2}. We allow Algorithm~\ref{Algo:Estimation2} to output both the changepoint estimator $\hat z$ and a test statistic $H_{\max}$, which can be used in our multiple changepoint algorithm to determine if an estimated changepoint is spurious.

% The aforementioned method is detailed in Algorithm~\ref{Algo:Estimation}.
To compute the sequence $(Q_t)_{t\in[n-1]}$ in Algorithm~\ref{Algo:Estimation2}, observe that the same $A$ and $Z$ can be used for all $t\in[n-1]$ and hence only need to be computed once.
It is worth noting that we exploit the structures of the sketched designs $(W_t)_t$ to greatly simplify their computations.
Recall that $a_t \in \mathbb{R}^{n-p}$ is the $t$th row vector of $A$, i.e., $A = (a_1, \dots, a_n)^\top$. $W_t$ are computed via the simple iterative scheme $W_0 = 0$ and $W_t = W_{t-1} + 2a_tx_t^\top$ for $t\in[n-1]$. As a common measure, we introduce the burn-in parameter ${\alpha}$ so that we forgo the possibilities of having changepoints in $(0, {\alpha} n) \cup ((1-{\alpha})n, n)$.

While the main focus of our current work is the changepoint localization problem, we remark that Algorithm~\ref{Algo:Estimation2} can be easily adapted to test the existence of a single changepoint in the sequence of regression coefficients. Specifically, we can construct the test
\begin{equation}
	\label{eq:test}
	\psi_{\alpha, \lambda, T} := \mathbbb{1}\Bigl\{\max_{t\in[\alpha n, (1-\alpha)n]} \|\soft(Q_t, \lambda)\| \geq T\Bigr\},
\end{equation}
where $T$ is some appropriate threshold. 

\begin{algorithm}[htbp]
\DontPrintSemicolon
	\KwIn{$X \in \mathbb{R}^{n \times p},  Y \in \mathbb{R}^{n}$ satisfying $n-p > 0$, a soft threshold level $\lambda \geq 0$, burn-in parameter ${\alpha}\geq 0$}
    Set $m \leftarrow n-p$\;
    Form $A\in\mathbb{O}^{n\times m}$ with columns orthogonal to the column space of $X$\;\label{Algo:Estimation1-2}
    Compute $Z\leftarrow A^\top Y$\;
    Set $W_0 = \mathbf{0}_{m\times p}$\; \label{Algo:Estimation1-4}
    \For{$1\le t \le n-1$}{
    	Compute $W_t \leftarrow W_{t-1} + 2 a_t x_t^\top$\;
    	Compute $Q_t = \{\diag(W_t^\top W_t)\}^{-1/2} W_t^\top Z$\;\label{Algo:Estimation1-7}
    	%Compute $H_t \leftarrow \|\mathrm{soft}(Q_t, \lambda)\|_2$\;
    }
	Form $Q := (Q_{\lfloor \alpha n\rfloor },\dots, Q_{\lceil (1-\alpha)n\rceil})^\top$\;
	% Compute $\hat M = \argmax_{M: \| M\|_{\mathrm{F}} \le 1} $\;
	Compute $\hat{v} \leftarrow$ the leading left singular vector of $\soft(Q, \lambda)$\; \label{Algo:Estimation2-3}
  \KwOut{$\hat{z} := \argmax_{\alpha n \le  t \le  (1-\alpha)n} |\hat v^\top Q_t| $ and $H_{\max} := \max_{\alpha n \le t \le (1-\alpha)n} \| \soft(Q_t, \lambda)\|$.} % $H_{\max} := | \hat{v}^\top Q_{\hat z}|$.}
	\caption{Pseudocode for changepoint estimation}
	\label{Algo:Estimation2}
\end{algorithm}

Finally, we mention that another natural approach to find the $W_t$ whose columns form the best sparse linear approximation of $Z = W_z\theta + \xi$ is to fit a sparse linear model by regressing $Z$ against $W_t$ and compare the goodness-of-fit across $t$ via the Bayesian Information Criterion (BIC).
We choose the BIC for the model selection purpose, though it is conceivably straightforward to apply any other model selection criteria. The pseudocode for this procedure is given in Algorithm~\ref{Algo:lassobic}. 
Specifically, for appropriately chosen $(\lambda_t)_{t\in[n-1]}$, we compute first the Lasso solutions in Step~\ref{alg:step:glmnet} and then the corresponding BICs in Step~\ref{alg:step:bic}.
In practice, the sequence of regularizing parameters $(\lambda_t)_t$ may be chosen via cross-validation for each $t$.

\begin{algorithm}[htbp]
\DontPrintSemicolon
	\KwIn{$X \in \mathbb{R}^{n \times p},  Y \in \mathbb{R}^{n}$ satisfying $n > p$,  ${\alpha}>0$ and a sequence $(\lambda_t)_{t\in[n-1]}$ }
	Follow Algorithm~\ref{Algo:Estimation2} until line \ref{Algo:Estimation1-4}\;
	\For{$1\le t \le n-1$}{
		Compute $W_t \leftarrow W_{t-1} + 2 a_t x_t^\top$\;
        Compute the Lasso estimator
        $
        \hat\theta_t  \leftarrow \argmin_{v \in \mathbb{R}^{p}}\bigl\{ \frac{1}{2m}\|Z-W_t v\|_2^2 + \lambda_t\|v\|_1\bigr\}
        $\;
		\label{alg:step:glmnet}
    Compute $H_t \leftarrow -(\|Z - W_t \hat\theta_t\|_2^2  + \|\hat\theta_t\|_0 \log m )$\;\label{alg:step:bic}
	}
\KwOut{$\hat{z} := \argmax_{{\alpha} n \le t \le (1-{\alpha})n} H_t $} % and $H_{\max}:=H_{\hat z}$.}
	\caption{Pseudocode for changepoint estimation with Lasso with BIC}
	\label{Algo:lassobic}
\end{algorithm}

Algorithm~\ref{Algo:Estimation2} has a computational complexity of $O(n^2 p)$, with the most computationally intensive step being its Step~\ref{Algo:Estimation1-2} to form the sketching matrix $A$ (e.g.\ via a QR decomposition). For Algorithm~\ref{Algo:lassobic}, each Lasso step has a computational cost of $O(k^2n)$ \citep{efron2004least}, leading to an overall computational complexity of $O(n^2(p+k^2))$. It is remarkable that for sparse signals ($k = O(\sqrt{p})$), the changepoint algorithms we proposed here has essentially the same computational complexity as the complementary-sketching-based two-sample test \citep{gao2021twosample} for any hypothesized changepoint location $t$. 

\subsection{Multiple changepoint estimation}

The single changepoint estimation procedure described above can be combined with a generic top-down multiple changepoint localization method, such as binary segmentation \citep{vostrikova1981detecting}, wild binary segmentation \citep{fryzlewicz2014wild} and its variants \citep[e.g.,][]{BaranowskiEtAl2019, kovacs2020seeded, fryzlewicz2020detecting}, to iteratively identify multiple changepoints. For concreteness, we describe an approach combining Algorithm~\ref{Algo:Estimation2} with the narrowest-over-threshold method of \citet{BaranowskiEtAl2019}. Algorithm~\ref{Algo:Multiple} is a slight generalisation of \citet[Algorithm~1]{BaranowskiEtAl2019}. It takes as input a single changepoint estimation procedure $\hat z$ and a testing procedure $\psi$. When the data $D_1,\ldots, D_n$ are the covariate-response pair $(X_i, Y_i)_{i\in[n]}$, we may apply Algorithm~\ref{Algo:Estimation2} or~\ref{Algo:lassobic} to obtain $\hat z$ and define $\psi(X,Y) := \mathbbb{1}_{\{H_{\max} > T\}}$ for some $T$ using the output $H_{\max}$ of Algorithm~\ref{Algo:Estimation2}. However, note that both Algorithms~\ref{Algo:Estimation2} and~\ref{Algo:lassobic} require the number of observations to be larger than the dimension for the complementary sketching to work. If this is not satisfied, we simply define $\hat z(X,Y) := 0$ and $\psi(X,Y) := 0$.

Essentially, in Algorithm~\ref{Algo:Multiple}, we generate multiple intervals and run the single changepoint algorithm on each interval to obtain candidate changepoint estimates and test results. We choose the candidate changepoint associated with the narrowest interval for which the test rejects the null, and add that to the set of estimated changepoints. We then segment the data at this estimated changepoint, and repeat the above process recursively on the data to the left and right segments, using only intervals lying completely within each segment. The process terminates when none of the tests reject the null. Furthermore, for practical reasons, we recommend combining Algorithm~\ref{Algo:Multiple} with some second-stage refinements, for which we discuss in more details in Section~\ref{sec:simu-multi}.

\begin{algorithm}[htbp]
\DontPrintSemicolon
\SetKwFunction{NOT}{NOT}\SetKwProg{Fn}{Function}{}{end}
	\KwIn{Data $D_1,\ldots,D_n$, number of intervals $M$, burn-in parameter $\varpi>0$, single changepoint estimation procedure $\hat z$ and a single changepoint testing procedure $\psi$}
    Set $\hat Z \leftarrow \emptyset$ and generate a set of $M$ intervals $\mathcal{M} := \{(s_1, e_1],\ldots, (s_M, e_M]\}$ independently and uniformly from $\{(a,b]: 0\leq a < b \leq n\}$. \;
    Run $\NOT(0, n)$ where $\texttt{NOT}$ is defined below.\;
    Let $\hat \nu \leftarrow |\hat Z|$ and sort elements of $\hat Z$ in increasing order to yield $\hat z_1 < \cdots  < \hat z_{\hat\nu}$.\;
    \vskip 0.5ex
    \KwOut{$\hat{z}_1,\ldots,\hat{z}_{\hat\nu}$}
    \vskip 0.5ex
    \Fn{\NOT{$s$, $e$}}{
%        Set $\mathcal{M}_{s,e} \leftarrow \{m: (s_m, e_m] \subseteq (s, e]\}$\;
%        \For{$m \in \mathcal{M}_{s,e}$}{
%                Run Algorithm~\ref{Algo:Estimation2} with input $X_{(s_m, e_m]}$, $Y_{(s_m, e_m]}$, $\lambda$ and $\alpha$, and let $\hat{z}^{(m)}$ and $H_{\max}^{(m)}$ be the output.\;\label{Algo:MultipleStep7}
%            }
        \label{Line5}Set $\mathcal{R}^{(s,e]}\leftarrow\{m: (s_m, e_m] \subseteq (s, e],\, \psi(D_{(s_m+ n\varpi, e_m- n\varpi]}) = 1\}$\;
        \If{$\mathcal{R}^{(s,e]}\neq \emptyset$}{
            \label{Line7}$m_0 \leftarrow \argmin_{m\in\mathcal{R}^{(s,e]}} (e_m - s_m)$\;
            $b\leftarrow  s_{m_0} + \hat z(D_{(s_m, e_m]})$\;
            $\hat Z \leftarrow \hat Z \cup \{b\}$\;
            \NOT{$s$, $b$}\;
            \NOT{$b$, $e$}\; } 
          }
	\caption{Pseudocode for multiple changepoint estimation}
	\label{Algo:Multiple}
\end{algorithm}

% \newpage
\section{Theoretical guarantees} % for a single changepoint}
\label{Sec:Theory}
In this section, we establish theoretical guarantees for the changepoint procedures proposed in Section~\ref{Sec:Method}. We start by focusing on the single changepoint estimation problem. For simplicity of analysis, we will assume that the noise variance $\sigma^2$ is known in this section, which by scale invariance can be further assumed to be equal to 1. We discuss practical aspects of estimating $\sigma^2$ in Section~\ref{sec:tuning-variants}.  We first present two conditions, which we will need to establish the results in this section.
\begin{condition}
	\label{cond:design}
	All entries of $X$ are independent standard normals. % $x_i \sim N_p(0, I_p)$ for $i \in [n]$.
\end{condition}
\begin{condition}
	\label{cond:regime}
	$n, z, p$ satisfy that $n>p$, $z/n \to \tau \in (0,1)$ and $(n-p)/n \to \eta \in (0,1)$ as $\min(z,n,p) \to \infty$.
\end{condition}
The design Condition~\ref{cond:design} requires that the rows of the design matrix $X = (x_1, \dots, x_n)^\top$ follow the isotropic Gaussian distribution.  Condition~\ref{cond:regime} specifies the asymptotic regime we work in. Note that the assumption $n > p$ is necessary, since otherwise, even if $z$ is known, it is impossible to test if $\|\theta\|_2 = 0$ against a sparse alternative (see the discussion of condition (C2) in \citet{gao2021twosample}).
The key ingredient of our theoretical analysis is the following proposition, which shows that $W_t^\top W_z$ is close to a multiple of identity in terms of their actions on sparse vectors.  
We impose both conditions~\ref{cond:design} and~\ref{cond:regime} only to enable the application of the existing random matrix theory on the limiting spectral measure of matrix-variate Beta distributions in the proof of Proposition~\ref{Prop:WtWz}.  In principle, even if the above conditions are violated, the theoretical results in the rest of the section hold for any data $(X,Y)$ such that \eqref{Eq:sup-WtWz} is satisfied.  In particular, we remark that the empirical study in Section~\ref{Sec:Model-Mis} has demonstrated that our methodology exhibits good finite-sample performance even when the above two conditions do not hold.

\begin{prop}
	\label{Prop:WtWz}
	Suppose that Conditions~\ref{cond:design} and~\ref{cond:regime} are satisfied and define
	\[
		g(t; z) := \begin{cases}
			4t(n-z) (n-p)/n^2 & \text{if }  1\le t \le z,  \\
			4z(n-t) (n-p)/n^2 & \text{if }  z < t \le n-1.
		\end{cases}
	\]
	There exists a constant $C_{\tau,\eta}>0$, depending only on $\tau$ and $\eta$ such that with probability 1, for any fixed $v\in\mathcal{S}^{p-1}$ and $\ell \in [p]$, we have for all but finitely many $p$'s that
	\begin{equation}
		\sup_{t\in[n-1]} \sup_{u\in B_0(\ell)} u^\top \bigl\{W_t^\top W_z - g(t; z) I_p\bigr\}v \leq C_{\tau,\eta} \sqrt{\ell n \log p}.
		\label{Eq:sup-WtWz}
	\end{equation}
	\end{prop}
Note that we suppress the dependence on $n$ and $t$ in the notation of $g(t;z)$.  Taking $\ell=1$ in the above proposition, we would expect $\diag(W_t^\top W_t)$ to concentrate around  $g(t;t) I_p = 4t(n-t)(n-p)n^{-2}I_p$ for each $t$. This would allow us to approximate the test statistics $Q_t = \{\mathrm{diag}(W_t^\top W_t)\}^{-1/2}(W_t^\top W_z\theta + W_t^\top \xi)$ for a fixed $t$.  However, due to a lack of non-asymptotic probabilistic bounds in random matrix theory on the convergence of the spectral measures of matrix-variate Beta random matrices, we are unable to establish the said convergence of \(\diag(W_t^\top W_t) \) uniformly over $t\in[n-1]$.
As such, we instead show the theoretical results for a slightly modified variant of Algorithm~\ref{Algo:Estimation2}, where we replace the definition of \(Q_t\) by
\begin{equation*}
	% \label{Eqn:Qt-alternative}
	Q_t := \sqrt{\frac{n}{t(n-t)}} W_t^\top Z.
\end{equation*}
We will henceforth refer to the above variant of Algorithm~\ref{Algo:Estimation2} as Algorithm~\ref{Algo:Estimation2}$'$.
It is worth noting that the latter is merely a proof device, and in practice we always recommend applying Algorithm~\ref{Algo:Estimation2}.
Empirically, the primed variant has a slightly worse but comparable estimation accuracies than Algorithm~\ref{Algo:Estimation2}, which can be seen in Table~\ref{Tab:PrimeComparison} in Section~\ref{sec:tuning-variants}. 

With this alternative choice of $Q_t$, after removing the perturbation introduced by observational errors $\xi$'s with an appropriate soft-thresholding tuning parameter $\lambda$, we expect
\[
	H_t := \|\mathrm{soft}(Q_t,\lambda)\|_2 \approx h_t := \|\theta\|_2\gamma_t ,
	% \begin{cases}
	% 	\frac{4(n-p)}{n} \sqrt{\frac{t}{n(n-t)}} (n-z) \| \theta\|_2 & \text{if } 1\le t \le z,  \\[6pt]
	% 	\frac{4(n-p)}{n} \sqrt{\frac{n-t}{nt}} z \| \theta\|_2       & \text{if } z < t \le n-1.
	% \end{cases}
\]
where $\gamma_t$ is defined by
% gamma_t is almost the same as h_t, h_t is only used in the proofs as a local variable. 
\begin{equation}
	\label{Eq:gamma_t}
  \gamma_t := g(t;t) \sqrt{\frac{n}{t(n-t)}}=  
	\begin{cases}
		\frac{4(n-p)}{n} \sqrt{\frac{t}{n(n-t)}} (n-z) & \text{if } 1\le t \le z,  \\[6pt]
		\frac{4(n-p)}{n} \sqrt{\frac{n-t}{nt}} z       & \text{if } z < t \le n-1.
	\end{cases}
% 	\begin{cases} \frac{4t^{1/2}(n-z)(n-p)}{n^{3/2}(n-t)^{1/2}}, & t\leq z \\[5pt]
%               \frac{4(n-t)^{1/2}z(n-p)}{n^{3/2}t^{1/2}},     & t>z.
% 	\end{cases}
\end{equation}
Interestingly, $h_t$ (or $\gamma_t$) is proportional to the CUSUM statistic in the univariate change-in-mean problem, whence $h_t$ attains its maximum at $t=z$ \citep[cf.][Equation~(10)]{wang2018high}.
% More precise analysis leads to Theorem~\ref{Thm:LocalisationRate2}.
By exploiting the above observation, we establish in Theorem~\ref{thm:test} that the testing procedure~\eqref{eq:test} is capable of determining whether a (single) changepoint is present in the regression data, as mentioned in Section~\ref{Sec:Method}. 
% The consistency of this test follows as a corollary of the changepoint location estimation consistency result of Theorem~\ref{Thm:LocalisationRate}.

Recall $\zeta := (\beta^{(1)} + \beta^{(2)})/2$ and $\theta := (\beta^{(1)} - \beta^{(2)})/2$, where we regard $\zeta$ as a possibly dense nuisance parameter and wish to localize the changepoint only assuming the difference parameter $\theta$ is sparse, i.e., \( \|\theta\|_0 \le k\) for some unknown but fixed $k$ typically much smaller than $p$.

% \teal{[Remove Theorem 2, and make Corollary 3 the main theorem for the test based procedure]}

\begin{comment}
\begin{thm}
	\label{Thm:LocalisationRate}
	Assume Conditions~\ref{cond:design} and~\ref{cond:regime} and that data $(X, Y)$ are generated according to~\eqref{eqn:y-x-z} with $\nu=1$.
	Suppose that  $\| \theta\|_0 \le k\leq p/2$ and that $\min(\tau,1-\tau)\geq {\alpha}$ for some known ${\alpha}$. There exists $c_{\tau,\eta,{\alpha}}, C_{\tau,\eta,{\alpha}} > 0$, depending only on $\tau,\eta,{\alpha}$, such that if  $\lambda > c_{\tau,\eta,{\alpha}}\max(1,\|\theta\|_2)\log p$, then the output $\hat z$ from Algorithm old 1, with inputs $X, Y, \lambda$ and ${\alpha}$, satisfies that with probability 1, for all but finitely many $p$'s,
	\[
		\frac{|\hat z - z|}{n} \leq C_{\tau,\eta,{\alpha}} \frac{\lambda\sqrt{k}}{\sqrt{n}\|\theta\|_2}.
	\]
\end{thm}
\end{comment}

\begin{thm}
	\label{thm:test}
	Assume Conditions~\ref{cond:design} and~\ref{cond:regime} and that data $(X, Y)$ are generated according to~\eqref{eqn:y-x-z} with $\nu=1$.
	Suppose that  $\| \theta\|_0 \le k$ satisfies $(k\log p)/n\to 0$ and that $\min(\tau,1-\tau)\geq {\alpha}$ for some known ${\alpha}$. There exists  $c_{\tau,\eta,\alpha}, C_{\tau,\eta,\alpha}, c'_{\tau,\eta,\alpha}> 0$, depending only on $\tau,\eta,\alpha$, such that for $\lambda = c_{\tau,\eta,\alpha}\log p$, $T = C_{\tau,\eta,\alpha}\sqrt{k}\log p$, the following holds.
	\begin{enumerate}
		\item If $\theta = 0$, then $\psi_{\alpha,\lambda,T}(X,Y) \xrightarrow{\mathrm{a.s.}} 0$.
		\item If $\|\theta\|_2 > c'_{\tau,\eta,\alpha}\frac{\sqrt{k}\log p}{\sqrt{n}}$, then $\psi_{\alpha,\lambda,T}(X,Y) \xrightarrow{\mathrm{a.s.}} 1$.
	\end{enumerate}
\end{thm}

% last teal point
We now turn our attention to the estimation in Algorithm~\ref{Algo:Estimation2}. The key of understanding the performance of Algorithm~\ref{Algo:Estimation2}$'$ lies in an analysis of the estimated projection vector $\hat{v}$ in Step~\ref{Algo:Estimation2-3} of the algorithm. By Proposition~\ref{Prop:WtWz}, we expect $Q = (Q_1,\ldots,Q_{n-1})$ to be well-approximated by the rank-one matrix $\theta \gamma^{\top}$, where $\gamma = (\gamma_t)_{t\in[n-1]}$ is defined in \eqref{Eq:gamma_t}.
Thus, the oracle projection direction to aggregate $Q$ is along $\theta/\|\theta\|_2$. The following proposition shows that the estimated projection direction $\hat{v}$ is well-aligned with this oracle direction.
\begin{prop}
	\label{Prop:ProjAngle}
	Assume Conditions~\ref{cond:design} and~\ref{cond:regime} and that data $(X, Y)$ are generated according to~\eqref{eqn:y-x-z} with $\nu=1$. Suppose that $k\leq p/2$ and that $\min(\tau,1-\tau)\geq {\alpha}$ for some known ${\alpha}$. There exists $c_{\tau,\eta,{\alpha}}, C_{\tau,\eta,{\alpha}} > 0$, depending only on $\tau,\eta,{\alpha}$, such that if  $\lambda > c_{\tau,\eta,{\alpha}}\max(1,\|\theta\|_2)\log p$, then the projection direction estimator $\hat v$ in Algorithm~\ref{Algo:Estimation2}$'$ satisfies with probability 1 for all but finitely many $p$'s that
	\[
		\sin\angle(\hat v, \theta) \leq C_{\tau,\eta,{\alpha}} \frac{\lambda\sqrt{k}}{\sqrt{n}\|\theta\|_2}.
	\]
\end{prop}
Equipped with Proposition~\ref{Prop:ProjAngle}, we are now in a position to state the convergence rate of the changepoint estimator from Algorithm~\ref{Algo:Estimation2}$'$.
\begin{thm}
	\label{Thm:LocalisationRate2}
	Assume Conditions~\ref{cond:design} and~\ref{cond:regime} and that data $(X, Y)$ are generated according to~\eqref{eqn:y-x-z} with $\nu=1$. Suppose that $\|\theta\|_2\leq 1$, $k\leq p/2$ and that $\min(\tau,1-\tau)\geq {\alpha} > 0$ for some known ${\alpha}$. There exists $c_{\tau,\eta,{\alpha}} > 0$, depending only on $\tau,\eta,{\alpha}$, such that if  $\lambda > c_{\tau,\eta,{\alpha}}\log p$, then the output $\hat{z}$ of Algorithm~\ref{Algo:Estimation2}$'$ with input $(X,Y)$, $\lambda$ and ${\alpha}$ satisfies with probability 1 for all but finitely many $p$'s that
	\[
		\frac{|\hat z - z|}{n} \lesssim_{\tau,\eta,{\alpha}} \frac{\lambda^2\sqrt{k}}{\sqrt{n}\|\theta\|_2^2}.
	\]
\end{thm}

Theorem~\ref{Thm:LocalisationRate2} shows that with a tuning parameter choice of order $\log p$, and when $\|\theta\|_2$ is bounded (which is the more difficult regime for estimation), Algorithm~\ref{Algo:Estimation2}$'$ produces a consistent changepoint estimator with a rate of convergence of order $k^{1/2}n^{-1/2}\|\theta\|_2^{-2}$ up to logarithmic factors. 
However, in light of the testing viewpoint of Theorem~\ref{thm:test}, in which it is possible to test apart the null of no change against a sparse alternative if $\sqrt{n}\|\theta\|_2/\sqrt{k}$, up to logarithmic factors, is sufficiently large, the rate in the above theorem appears to have an extra factor of $\|\theta\|_2^{-1}$.
This additional factor is likely to arise from the technical difficulty of controlling the weak, though complex, dependence between the estimated projection direction $\hat v$ and the sketched Gaussian noises $(W_t^\top \xi: t\in[n-1])$. 
Indeed, the following theorem shows that if $\hat{v}$ is estimated from an independent sample, then the estimator from Algorithm~\ref{Algo:Estimation2} has a rate of convergence that agrees with what is prescribed in Theorem~\ref{thm:test}, up to logarithmic factors.

\begin{thm}
	\label{Thm:LocalisationRate3}
	Assume the same conditions as in Theorem~\ref{Thm:LocalisationRate2}. Let $(\tilde X, \tilde Y)$ be an independent copy of $(X, Y)$. Let $Q$ be the matrix constructed in Step 8 of Algorithm~\ref{Algo:Estimation2}$'$ with input $(X,Y)$, $\lambda$ and ${\alpha}$. Suppose $\hat v$ is computed in Step 9 of Algorithm~\ref{Algo:Estimation2}$'$ with input $(\tilde X,\tilde Y)$, $\lambda$ and ${\alpha}$. Then $\hat z:=\argmax_{t\in[n-1]}|\hat v^\top Q_t|$ satisfies with probability 1 for all but finitely many $p$'s that
	\[
		\frac{|\hat z - z|}{n} \lesssim_{\tau,\eta,{\alpha}} \frac{\lambda \sqrt{k}\log p}{\sqrt{n}\|\theta\|_2}.
	\]
\end{thm}
This additional independent sample $(\tilde X, \tilde Y)$ may be obtained in reality via a sample-splitting scheme. For example, we may take all odd time points to construct the $Q$ matrix, and then use the even time points to estimate the projection direction $\hat{v}$. However, such sample splitting is necessary only from a technical viewpoint, and the algorithm typically performs better without sample splitting in practice.

We remark that the rate in the above Theorem~\ref{Thm:LocalisationRate3} is slower compared to the usual results from change-in-mean problems, where rates of order $n^{-1}\|\theta\|_2^{-2}$ are achievable under appropriate conditions \citep[see e.g.][]{wang2018high, verzelen2020optimal}. Our slower rate arises essentially from the approximation step in Proposition~\ref{Prop:WtWz}, which would not be needed in a change-in-mean problem. It remains to be seen if the estimation rate can be improved via alternative and possibly more refined analysis routes.

We now turn our attention to theoretical guarantees in the multiple changepoint setting.  The following theorem shows that provided that we have a good single changepoint estimation and testing procedure in any changepoint problem, combining the narrowest-over-threshold with the single change procedures yields a multiple changepoint estimation procedure of similar accuracy with theoretical guarantees.

\begin{thm}
\label{Thm:Multiple}
Let $D_1,\ldots,D_n$ be a data sequence with changepoints $0 = z_0 < z_1 < \cdots < z_\nu < z_{\nu+1} = n$. We assume that $z_i - z_{i-1} \geq n\Delta_\tau$ for all $i\in[\nu+1]$. Let $\mathcal{M}$ be defined as in Algorithm~\ref{Algo:Multiple}. Write  $\mathcal{I}_0:=\{ (s,e] \in \mathcal{M}: (s+n\varpi,e-n\varpi]\cap \{z_1,\ldots,z_\nu\} = \emptyset\}$ and for $i\in[\nu]$, define $\mathcal{I}_i:= \{(s,e]\in\mathcal{M}: s\in [z_i-n\Delta_\tau/2, z_i-n\Delta_\tau/3], e\in[z_i+n\Delta_\tau/3,z_i+n\Delta_\tau/2]\}$ and $\tilde{\mathcal{I}}_i:= \{(s,e]\in\mathcal{M}: \min\{z_i- s, e-z_i\}\geq n\Delta_\tau/6 \text{ and } e-s\leq n\Delta_\tau\}$. Let $\hat z$ and $\psi$ be the single changepoint estimation and testing procedure used in Algorithm~\ref{Algo:Multiple}. Define the events
\begin{align*}
  \Omega_0 & : = \{ \forall\, i \in [\nu], \exists\, m \in [M], \text{ s.t. } (s_m, e_m] \in \mathcal{I}_i \},\\
  \Omega_1 & := \{\text{$\psi(D_{(s+n\varpi,e-n\varpi]}) = 0$ for all $(s,e]\in\mathcal{I}_0$}\},\\
  \Omega_2 & := \bigcap_{i\in[\nu]}\bigl\{\text{$\psi( D_{(s+n\varpi,e -n\varpi]}) = 1$ for all $(s,e]\in\mathcal{I}_i$}\bigr\}
\end{align*}
and for some $\phi_1,\ldots,\phi_{\nu} > 0$,
\[
  \Omega_3:= \bigcap_{i\in[\nu]}\bigl\{\text{$|\hat z(D_{(s,e]})  - (z_i- s)| \leq n\phi_i$ for all $(s,e]\in\tilde{\mathcal{I}}_i$}\bigr\}.
\]
Let $\hat z_1,\ldots,\hat z_{\hat\nu}$ be the output of Algorithm~\ref{Algo:Multiple} with inputs $D_1,\ldots, D_n$, $M > 0$, $\varpi = \Delta_\tau/6$, $\hat z$ and $\psi$. Assume further $\phi := \max_{i\in[\nu]}\phi_i < \varpi$.  We have on $\Omega_0\cap\Omega_1\cap\Omega_2\cap\Omega_3$ that 
\[
  \hat\nu = \nu \text{ and }  n^{-1}|\hat z_i - z_i| \leq \phi_i \text{ for all $i\in[\nu]$}.
\]
In particular, we have
\[
\mathbb{P}\bigl(\hat\nu = \nu \text{ and }  n^{-1}|\hat z_i - z_i| \leq \phi_i \text{ for all $i\in[\nu]$}\bigr) \geq 1 - \mathbb{P}(\Omega_1^c) - \mathbb{P}(\Omega_2^c) -\mathbb{P}(\Omega_3^c) - \nu e^{-\Delta_\tau^2 M / 36}.
\]
\end{thm}

Note that the theorem is valid for any generic multiple changepoint estimation that combines valid single changepoint estimation and testing procedures and the top-down narrowest-over-threshold multiple changepoint estimation paradigm. Hence, it can be applied in contexts other than the linear regression setting here.  The statement of Theorem~\ref{Thm:Multiple} is slightly stronger than the usual results on narrowest-over-threshold procedures, where $\phi_i$ are taken identical.

Applying the above theorem to our specific problem, we extend the single changepoint estimation result in Theorem~\ref{Thm:LocalisationRate3} to the multiple changepoint setting and establish the estimation accuracy of Algorithm~\ref{Algo:Multiple}.
We first give the following condition, which is the equivalent of Condition~\ref{cond:regime} in the multiple changepoint setting.
\begin{condition}
  \label{cond:regime-M}
  $n,  p$ satisfy $n >p$ and that $(n-p)/n \rightarrow \eta$ as $\min(n,p) \rightarrow \infty$.  Assume further the changepoints satisfy $0 = z_0 < z_1 < \dots < z_\nu < z_{\nu+1} = n$ and $z_i - z_{i-1} > n\Delta_\tau$ for all $i \in [\nu+1]$ and $z_i/n \rightarrow \tau_i$ for $0 \le i \le \nu+1$.
\end{condition}
Due to the asymptotic nature of our theoretical results in the above, we consider a sequence of Algorithm~\ref{Algo:Multiple}. To facilitate proof, we study a specific coupling of the random intervals $\{(s_m,e_m]:m\in[M]\}$ generated across this sequence as follows:
\begin{equation}
\label{Eq:IntervalCoupling}
\begin{aligned}
  &(\tilde s_m, \tilde e_m] \stackrel{\mathrm{iid}}{\sim} \mathrm{Unif}\bigl(\{(\tilde a, \tilde b]: 0\leq a < b\leq 1\}\bigr),\text{ } \forall\; m \in [M],\\
&\text{$s_m = \lfloor n\tilde s_m\rfloor $, $e_m = \lceil n\tilde e_m\rceil$ for $m\in[M]$ and $n\in\mathbb{N}$}.
\end{aligned}
\end{equation}
Note that the intervals generated by \eqref{Eq:IntervalCoupling} have the same law as those generated in Algorithm~\ref{Algo:Multiple}.

\begin{cor}
  \label{Cor:Multiple}
  Let $X$ and $Y$ be generated by \eqref{eqn:y-x-z} and write $D_i := (X_i, Y_i)$ for $i\in[n]$.  Assume Conditions~\ref{cond:design} and~\ref{cond:regime-M} hold.  There exist $c, C, c', C'>0$, depending only on $\alpha, \Delta_\tau, \eta$, such that the following holds.  For $\alpha < 1/6$, $\lambda=c\log p$, $T=C\sqrt{k}\log p$, let $\hat z = \hat{z}_{\alpha, \lambda}$ be the sample-splitted version of the single changepoint estimator defined in Algorithm~\ref{Algo:Estimation2}$'$ and $\psi = \psi_{\alpha, \lambda, T}$ be the testing procedure defined in \eqref{eq:test}. If $\Delta_\tau>3(1-\eta)$, $\frac{c'\sqrt{k}\log p}{\sqrt{n}} \le \|\theta^{(i)}\|_2\le 1$ and $\| \theta^{(i)}\|_0 \le k$ such that $\frac{k\log p}{n} \rightarrow 0$, then the output $\hat z_1,\ldots,\hat z_{\hat \nu}$ of Algorithm~\ref{Algo:Multiple} with intervals $\{(s_m,e_m]:m\in[M]\}$ generated according to~\eqref{Eq:IntervalCoupling}, inputs $(D_i)_{i\in[n]}$, $M>0$, $\varpi=\tau/6$, $\hat z_{\alpha,\lambda}$ and $\psi_{\alpha, \lambda, T}$ satisfies with probability $1 - \nu e^{-\Delta_\tau^2 M /36}$ that for all but finitely many $p$'s, 
  \[
    \hat\nu = \nu \text{ and }\frac{|\hat z_i - z_i|}{n} \leq \frac{C'\lambda\sqrt{k}\log p}{\sqrt{n}\|\theta^{(i)}\|_2} \text{ for all $i\in[\nu]$}.
  \]
\end{cor}

\section{Numerical study}
\label{Sec:Numerical}
The implementation of our single- and multiple-changepoint algorithms are both available in our GitHub repository.\footnote{\url{https://github.com/gaofengnan/charcoal}} 

\subsection{Tuning parameter choice and comparison of variants}
\label{sec:tuning-variants}
Theoretical analysis in Section~\ref{Sec:Theory} have assumed that the noise variance $\sigma^2$ is known. In practice, we may obtain an upward-biased estimator $\tilde\sigma$ as the median absolute deviation of entries of the $Q$ matrix. We note
\begin{align*}
\mathrm{Var}(Q_{t}) &= \mathbb{E}[\mathrm{Var}\{\tilde W_t^\top (W_z\theta + \xi)\mid X\}] + \mathrm{Var}[\mathbb{E}\{\tilde W_t^\top (W_z\theta + \xi)\mid X\}]\\
&=\sigma^2\mathbb{E}(\tilde W_t^\top \tilde W_t) + \mathrm{Var}(\tilde W_t^\top W_z\theta).
\end{align*}
Since $\mathbb{E}(\tilde W_t^\top \tilde W_t)$ has all diagonal entries equal to 1, every entry of $Q$ has a marginal variance of at least $\sigma^2$.

\begin{figure}[htbp]
  \centering
  \begin{subfigure}[b]{0.48\textwidth}
    \includegraphics[width=\textwidth]{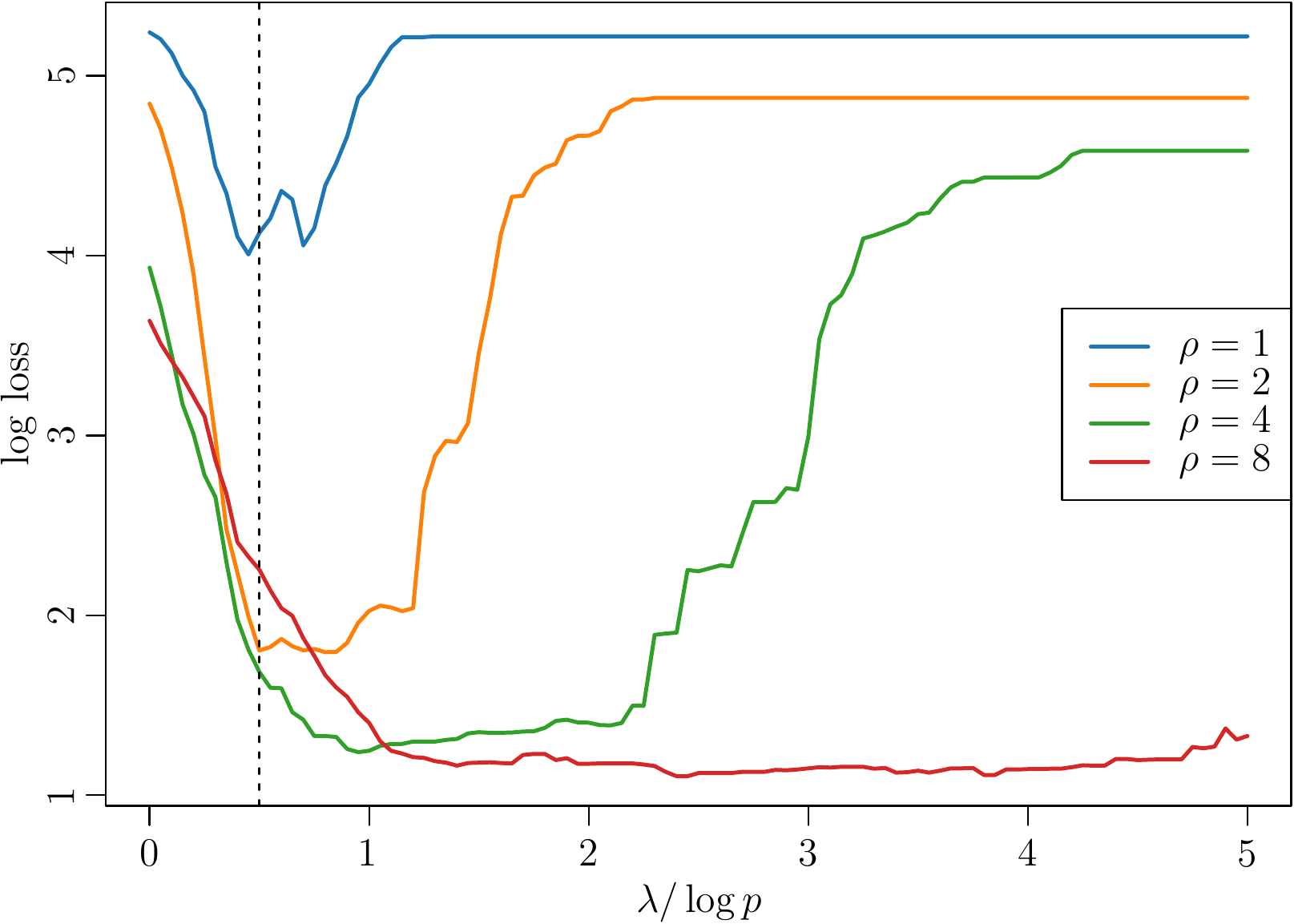}
    \caption{Varying $\rho$}
  \end{subfigure}
  ~ 
  \begin{subfigure}[b]{0.48\textwidth}
    \includegraphics[width=\textwidth]{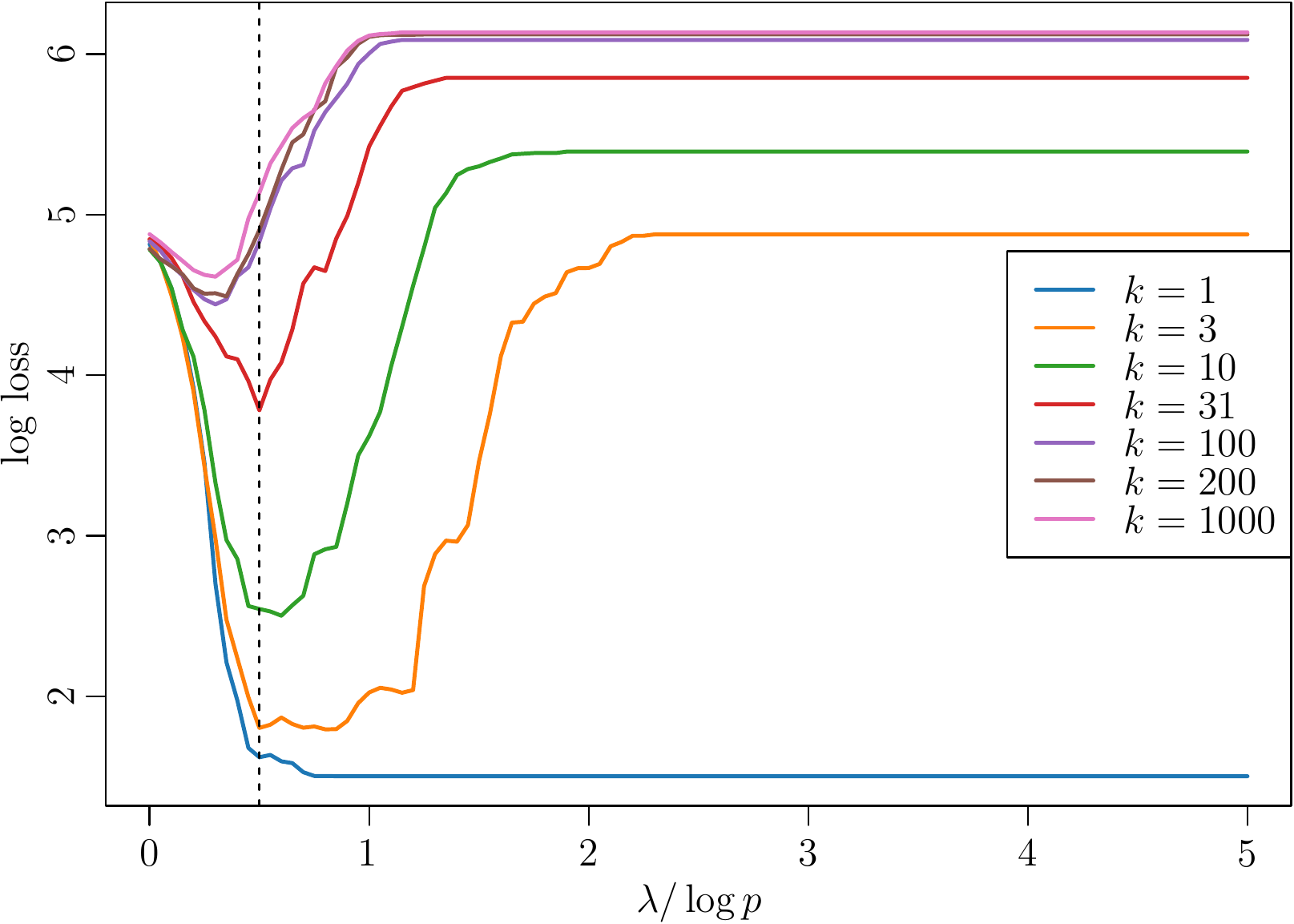}
    \caption{Varying $k$}
  \end{subfigure}
  \\ 
  \begin{subfigure}[b]{0.48\textwidth}
    \includegraphics[width=\textwidth]{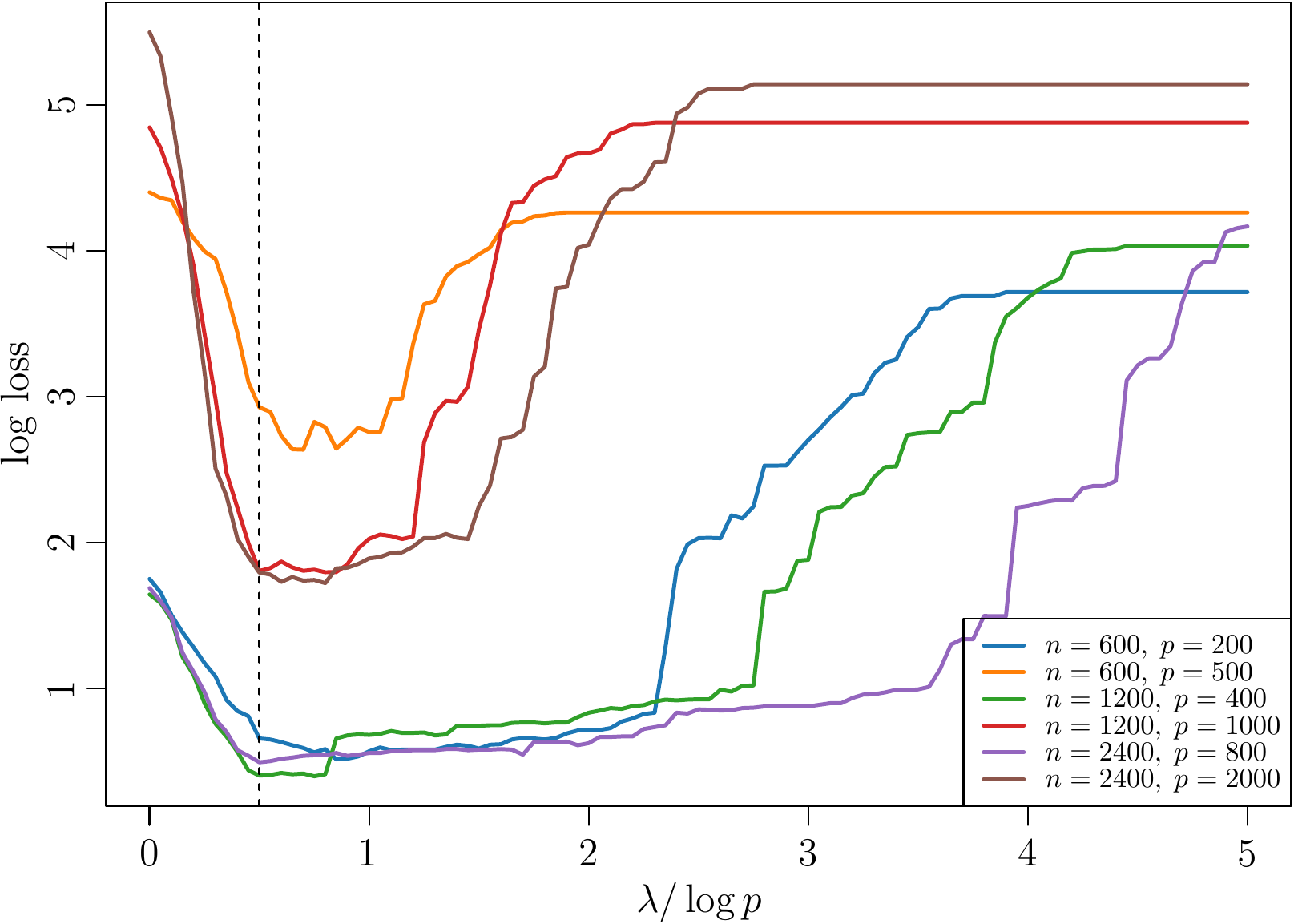}
    \caption{Varying dimensions}
  \end{subfigure}
  ~
  \begin{subfigure}[b]{0.48\textwidth}
    \includegraphics[width=\textwidth]{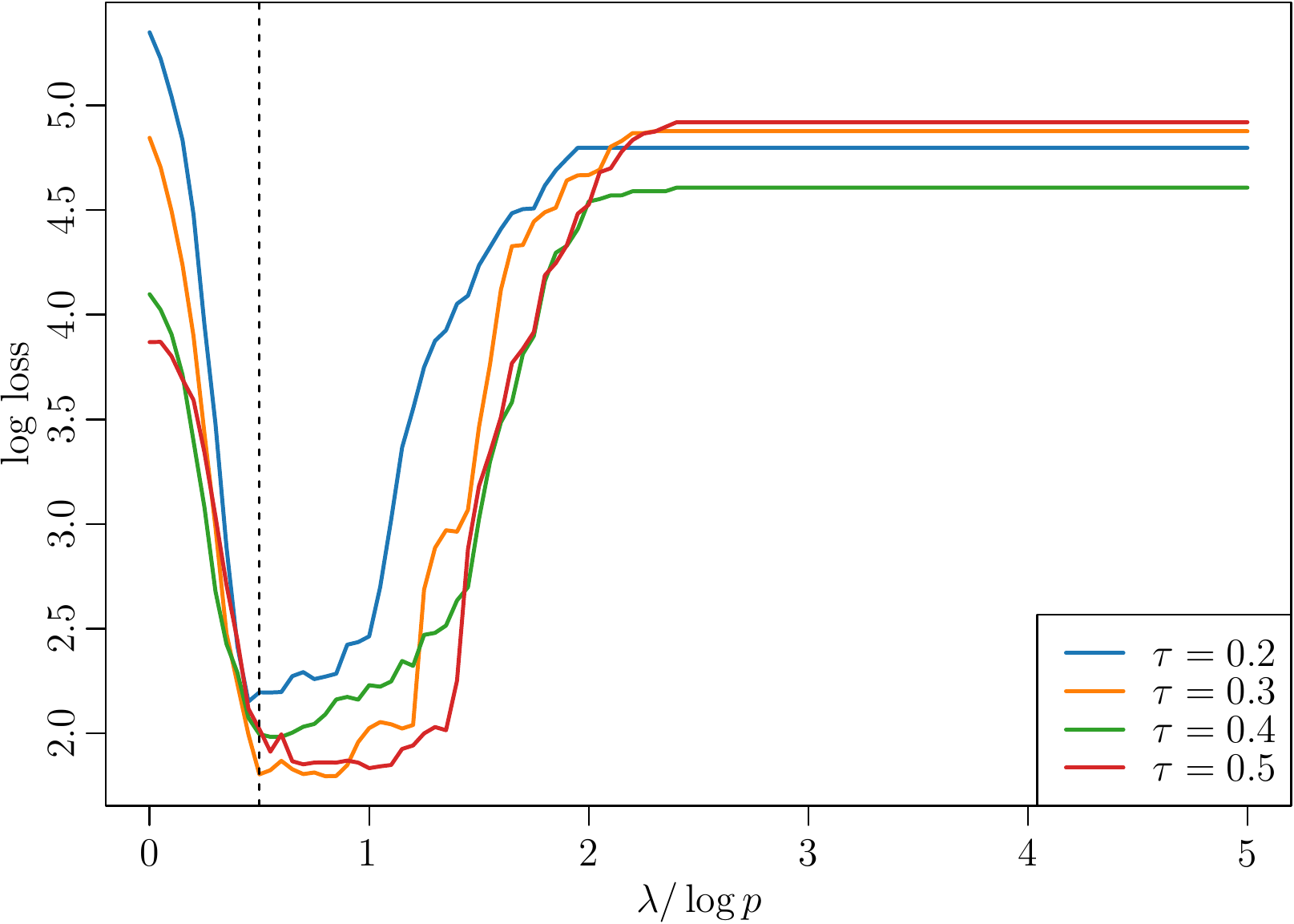}
    \caption{Varying change locations}
  \end{subfigure}
% \begin{tabular}{cc}
% \includegraphics[width=0.48\textwidth]{fig/fig1a.pdf} & 
% \includegraphics[width=0.48\textwidth]{fig/fig1b.pdf}\\
% \includegraphics[width=0.48\textwidth]{fig/fig1d.pdf}
% \end{tabular}
  \caption{\label{fig:tuning} The effect of choosing different $\lambda$ in a series of studies.  Our recommendation of $\lambda := 0.5\log p$ is marked in dashed vertical line in each panel. Unless specified otherwise in each one, the panels share the parameters $n=1200, p=1000, \tau=0.3, k = 3, \rho=2$.}
\end{figure}

% \teal{[Remove old algorithm 1 in the discuss below]}

Algorithm~\ref{Algo:Estimation2} requires a soft-thresholding tuning parameter $\lambda > 0$ as an input. The theoretical results in Section~\ref{Sec:Theory} suggests using $\lambda =c\sigma\log p$ for some $c>0$. We investigate here the performance of our algorithm at different soft-thresholding levels $\lambda$. Specifically, we computed the logarithmic average loss $|\hat z - z|$ of Algorithm~\ref{Algo:Estimation2} over 100 Monte Carlo repetitions for parameter settings of $n\in \{600,1200,2400\}$, $p\in\{n/3, 5n/6\}$, $\tau\in\{0.2,0.3,0.4,0.5\}$, $k\in\{1,3,10,\lfloor \sqrt{p}\rfloor, \lfloor 0.1p\rfloor, \lfloor 0.2p\rfloor, p\}$, $\rho:=\|\theta\|_2\in\{1,2,4,8\}$, $\sigma=1$ and various choices of $c \in [0.1, 5]$. In all our simulations here and below, we sample the vector of change in the regression coefficients $\theta$ uniformly from the set $\{v: \|v\|_0 = k, \|v\|_2=\rho\}$, and generate dense pre-change vector from $N_p(0, \max\{1,\rho^2\}I_p)$. Figure~\ref{fig:tuning} illustrates part of the simulation results where we vary one aspect of the parameters at a time.  From the figure, we see that a choice of $c = 0.5$ provides good statistical performance across the parameter settings considered, and we will henceforth adopt this choice of $\lambda = 0.5\tilde\sigma\log(p)$ in our subsequent numerical studies. 

We now compare the statistical performance of various versions of complementary-sketching-based approaches proposed in the paper, including Algorithms~\ref{Algo:Estimation2} and~\ref{Algo:lassobic} from Section~\ref{Sec:Method} and the slight variant Algorithm \ref{Algo:Estimation2}$'$ mentioned in Section~\ref{Sec:Theory} to facilitate theoretical analysis.  
For a demonstrative purpose, we have also included the naive hard- and soft-thresholded changepoint estimators $\hat{z}^{\hard}$ and $\hat{z}^{\soft}$ mentioned just above Algorithm~\ref{Algo:Estimation2}. 
We use the $\lambda$ choice suggested in the previous paragraph for $\hat{z}^{\hard}$ and $\hat{z}^{\soft}$, Algorithm~\ref{Algo:Estimation2} and its variant, and choose $\lambda_t$ in Algorithm~\ref{Algo:lassobic} via a five-fold cross-validation for each $t\in[n-1]$. Empirical observations suggest that Algorithms~\ref{Algo:Estimation2}, \ref{Algo:Estimation2}$'$ and \ref{Algo:lassobic} work well without any burn-in (i.e.\ $\alpha=0$). 
However, both $\hat{z}^{\soft}$ and $\hat{z}^{\hard}$ do suffer from more serious boundary effects, as seen in the large root mean squared errors in Table~\ref{Tab:PrimeComparison}. % and a burn-in parameter of $\alpha = 0.1$ is used for these two algorithms. 
In addition, Algorithm \ref{Algo:Estimation2} has roughly the same but slightly better estimation accuracy compared to its primed variant.  This justifies our recommendation of Algorithm~\ref{Algo:Estimation2} over its primed variant, and the similarity in performance further consolidates the relevance of our theoretical analysis on the primed variant as a proof device.

\begin{table}[htbp]
\centering

\begin{tabular}{rrrrrrrrrr}
\toprule
$n$ & $p$ & $k$ & $\rho$ & $\hat{z}^{\soft}$ & $\hat{z}^{\hard}$ & Alg$1$  & Alg$1'$ & Alg$2$\\
\midrule
$600$ & $200$ & $3$ & $1$ & $50.1$ & $178.35$ & $12.6$ & $40.79$ & $22.34$\\
&  &  & $2$ & $3.98$ & $88.76$ & $3.1$ & $5.01$ & $4.58$\\
&  &  & $4$ & $25.37$ & $92.3$ & $2.14$ & $4.33$ & $2.29$\\
\addlinespace
&  & $14$ & $1$ & $31.7$ & $122.11$ & $37.36$ & $83.59$ & $129.71$\\
&  &  & $2$ & $6.4$ & $88.04$ & $6.14$ & $7.32$ & $9.83$\\
&  &  & $4$ & $3.88$ & $58.92$ & $3.82$ & $4.67$ & $2.95$\\
\addlinespace
$1200$ & $400$ & $3$ & $1$ & $11.24$ & $196.81$ & $10.95$ & $13.34$ & $19.54$\\
&  &  & $2$ & $3.15$ & $136.02$ & $3.69$ & $4.58$ & $4.96$\\
&  &  & $4$ & $35.65$ & $101.17$ & $1.67$ & $4.36$ & $2.09$\\
\addlinespace
&  & $20$ & $1$ & $39.17$ & $146.98$ & $18.14$ & $37.63$ & $168.56$\\
&  &  & $2$ & $5.06$ & $144.5$ & $7.29$ & $7.59$ & $15.51$\\
&  &  & $4$ & $4.72$ & $106.84$ & $2.72$ & $2.93$ & $4.47$\\
\bottomrule
\end{tabular}
\caption{\label{Tab:PrimeComparison}Comparisons of performances of $\hat{z}^{\soft}$, $\hat{z}^{\hard}$, Algorithms~\ref{Algo:Estimation2}, \ref{Algo:Estimation2}$'$ and \ref{Algo:lassobic} in various settings in terms of root mean squared error.  No burn-in is applied anywhere ($\alpha=0$).  The true change takes place at $0.3n$.}
\end{table}

\subsection{Comparisons with other methods}
From the discussion above, we recommend using Algorithms~\ref{Algo:Estimation2} and~\ref{Algo:lassobic} for their robustness against the choice of the burn-in parameter $\alpha$, and for the more accurate estimation of Algorithm~\ref{Algo:Estimation2} over its primed variant.
We will henceforth focus on Algorithms~\ref{Algo:Estimation2} and~\ref{Algo:lassobic}, which we call \charcoalproj and \charcoallasso, respectively.  In this section, we compare the performance of \charcoalproj and \charcoallasso with existing approaches in the literature. Specifically, we will compare against the VPBS algorithm of \citet{RinaldoWangWen2021}, two-sided Lasso-based approaches of \citet{lee2016lasso} (LSS) and \citet{Leonardi2016Computationally} (LB), and a two-stage refinement approach of \citet{kaul2019efficient} (KJF). We have used the authors' own implementation for VPBS and KJF, and \citet{kaul2019efficient}'s implementation for LSS. We have implemented LB ourselves using the recommended tuning parameter choices as in \citet{Leonardi2016Computationally}. It is worth noting that none of the four existing methods in the literature were designed to estimate changes in the regression coefficients when both the pre- and post-change coefficients are dense. 

We compare the performance in terms of mean absolute loss of various methods in a single changepoint estimation task for $n\in\{600,1200\}$, $p\in\{n/3, 5n/6\}$, $\tau\in\{0.1,0.3,0.5\}$, $k\in\{3,\lfloor\sqrt{p}\rfloor, p\}$, $\rho\in\{1,2,4,8\}$. Table~\ref{Tab:Comparison} shows a representative subset of these simulation results. We see that none of VPBS, LB, KJF and LSS show any sign of consistent estimation as their average loss do not decrease as the signal strength increases. On the other hand, both \charcoalproj and \charcoallasso have shown highly promising performance in various settings. It is surprising that \charcoalproj and \charcoallasso also seem to work even when the vector of change is dense. We notice that \charcoallasso shows better estimation accuracy when either the signal strength $\rho$ is high or the vector of change $\theta$ is dense.

\begin{table}[htbp]
	\begin{center}
    \begin{tabular}{rrrrrrrrr}
  \toprule
$p$ & $k$ & $\rho$ & \charcoalproj & \charcoallasso & VPBS & LB & KJF & LSS\\
\midrule
$400$ & $3$ & $1$ & $\textbf{7.2}$ & $13.2$ & $452.4$ & $556.1$ & $238.8$ & $472.2$\\
&  & $2$ & $\textbf{2.2}$ & $3.5$ & $476.3$ & $569.2$ & $239.3$ & $364.1$\\
&  & $4$ & $\textbf{1.1}$ & $1.5$ & $434.2$ & $532.8$ & $239.1$ & $272.1$\\
&  & $8$ & $\textbf{0.7}$ & $0.8$ & $326.3$ & $496.8$ & $239.1$ & $310.8$\\
\addlinespace
 & $20$ & $1$ & $\textbf{12.4}$ & $85.4$ & $422.7$ & $528.8$ & $238.9$ & $479.5$\\
&  & $2$ & $\textbf{3.0}$ & $9.2$ & $494.9$ & $546.8$ & $238.9$ & $284.5$\\
&  & $4$ & $\textbf{2.0}$ & $2.6$ & $431.9$ & $553.1$ & $239.1$ & $268.5$\\
&  & $8$ & $1.9$ & $\textbf{0.8}$ & $356.2$ & $513.3$ & $239.3$ & $261.5$\\
\addlinespace
 & $400$ & $1$ & $\textbf{162.2}$ & $344.2$ & $477.8$ & $569.8$ & $238.8$ & $429.9$\\
&  & $2$ & $\textbf{46.3}$ & $338.4$ & $504.0$ & $583.2$ & $238.8$ & $252.4$\\
&  & $4$ & $25.3$ & $\textbf{13.3}$ & $446.3$ & $554.1$ & $238.9$ & $285.6$\\
&  & $8$ & $20.7$ & $\textbf{3.0}$ & $355.6$ & $487.6$ & $239.1$ & $250.1$\\
\addlinespace
$1000$ & $3$ & $1$ & $\textbf{60.7}$ & $113.3$ & $241.6$ & $429.5$ & $237.2$ & $227.3$\\
&  & $2$ & $\textbf{8.3}$ & $11.8$ & $243.4$ & $441.4$ & $239.0$ & $228.2$\\
&  & $4$ & $\textbf{2.9}$ & $4.0$ & $239.5$ & $366.9$ & $243.9$ & $230.6$\\
&  & $8$ & $2.4$ & $\textbf{1.4}$ & $235.1$ & $245.1$ & $262.2$ & $230.7$\\
\addlinespace
 & $31$ & $1$ & $300.3$ & $364.9$ & $233.4$ & $440.1$ & $238.8$ & $\textbf{227.4}$\\
&  & $2$ & $\textbf{71.7}$ & $140.9$ & $242.5$ & $469.5$ & $238.9$ & $228.3$\\
&  & $4$ & $16.0$ & $\textbf{12.5}$ & $251.3$ & $358.4$ & $238.9$ & $224.5$\\
&  & $8$ & $13.7$ & $\textbf{4.6}$ & $244.5$ & $249.0$ & $238.2$ & $230.1$\\
\addlinespace
 & $1000$ & $1$ & $275.5$ & $359.8$ & $232.6$ & $483.0$ & $239.3$ & $\textbf{231.8}$\\
&  & $2$ & $256.9$ & $320.8$ & $238.4$ & $447.4$ & $238.9$ & $\textbf{229.2}$\\
&  & $4$ & $224.1$ & $\textbf{91.0}$ & $242.7$ & $378.2$ & $239.1$ & $228.0$\\
&  & $8$ & $194.5$ & $\textbf{39.6}$ & $246.4$ & $253.5$ & $242.4$ & $226.7$\\
\bottomrule
\end{tabular}
		\caption{\label{Tab:Comparison}Average loss of various changepoint methods under different settings. Other parameters: $n=1200$, $z=360$. The method with the least average loss in each line is marked in bold. }
	\end{center}
\end{table}

\subsection{Model misspecification}
\label{Sec:Model-Mis}
While we have focused on the Gaussian Orthogonal Ensemble (GOE) design (i.e.\ $X$ has independent $N(0,1)$ entries) and Gaussian noise in the theoretical analysis, our methodology can be applied in more general settings. In this subsection, we investigate the robustness of the estimation accuracy of our method to deviations from this Gaussian distributional assumptions. Specifically for $n= 1,200$, $p = 400$, $\tau=0.3$, $k=20$ and $\rho \in \{1.5^0, \dots, 1.5^8\}$, we varied the design matrix $X$ to have either $N_p(0,\Sigma)$ rows, where $\Sigma =  (0.7^{|i-j|})_{i,j\in[p]}$ has an autoregressive Toeplitz structure, or independent Rademacher entries. We also vary the noise distribution to take $t_4$, $t_6$, centred $\mathrm{Exp}(1)$ or Rademacher distributions. Overall, we see from Figure~\ref{Fig:Robustness} that the performance of \charcoallasso is robust to both non-GOE design matrices and discrete, heavy-tailed or skewed noise distributions. Similar results hold for the \charcoalproj method. 

\begin{figure}[htbp]
	\centering
	\begin{subfigure}[b]{0.45\textwidth}
		\includegraphics[width=\textwidth]{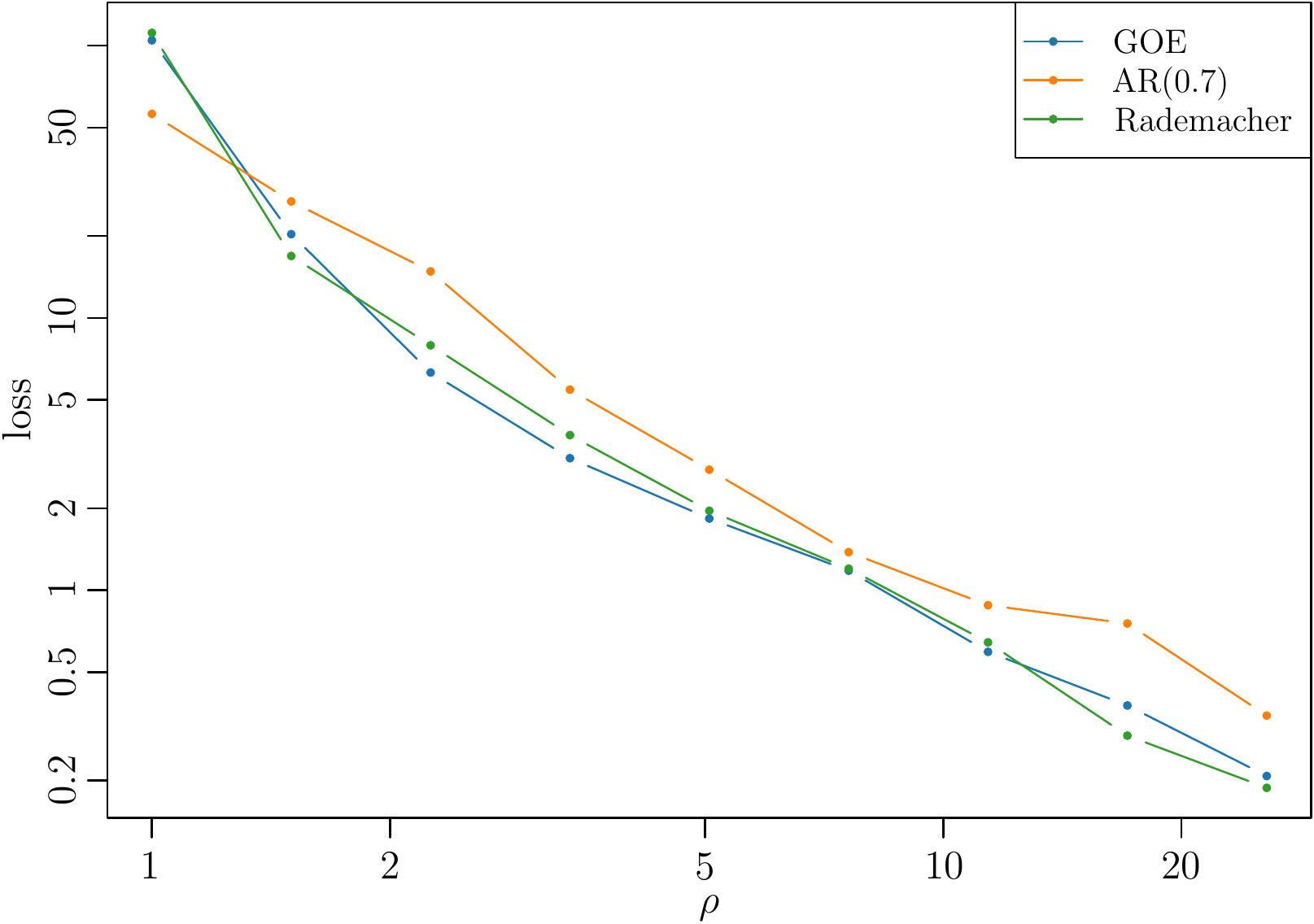}
		\caption{Varying design matrices}
	\end{subfigure}
	\begin{subfigure}[b]{0.45\textwidth}
		\includegraphics[width=\textwidth]{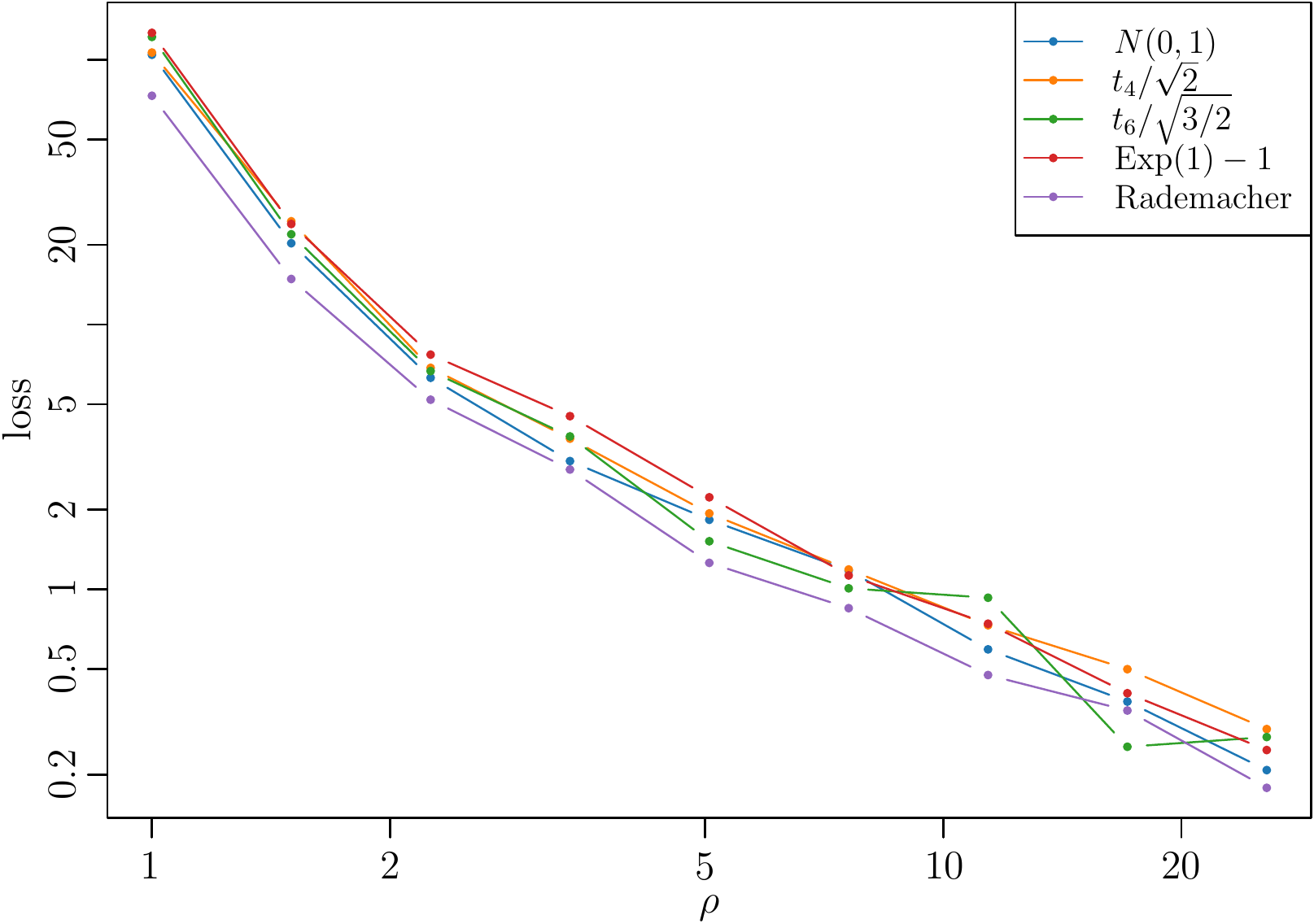}
		\caption{Varying noise distributions}
	\end{subfigure}
	\caption{\label{Fig:Robustness} Robustness to varying design matrices and noise distributions. }

\end{figure}

\subsection{Multiple changepoints}
\label{sec:simu-multi}
As mentioned in Section~\ref{Sec:Method}, our \charcoal algorithms can be easily combined with generic multiple changepoint methods to deal with multiple changepoints, and we proposed a specific version in Algorithm~\ref{Algo:Multiple} of such a multiple changepoint localization procedure. We run Algorithm~\ref{Algo:Multiple} with $\varpi = 0$ and $M = 200$. For the single changepoint estimator input $\hat z$ in Algorithm~\ref{Algo:Multiple}, we employ Algorithm~\ref{Algo:Estimation2} with the recommended value of $\lambda$ in Section~\ref{sec:tuning-variants} and the burn-in parameter $\alpha = 0.05$. For the testing procedure input $\psi$ of Algorithm~\ref{Algo:Multiple}, we run Algorithm~\ref{Algo:Estimation2} to obtain output $H_{\max}$ and define $\psi(X,Y) = \mathbbb{1}_{\{H_{\max} > T\}}$, where the testing threshold $T$ is chosen by a Monte Carlo simulation as follows. We generate $B=1000$ pairs of $(X,Y)$ under model~\eqref{eqn:y-x-z} with $\nu=0$ (i.e.\ no changepoint), and run Algorithm~\ref{Algo:Estimation2} with these synthetic $(X,Y)$ pairs and the same $\lambda$ and $\alpha$ choices as above. This would return $B$ test statistics $(H_{\max}^{b})_{b\in[B]}$, which are used to estimate an upper $0.01/M$ quantile by fitting a generalized extreme value distribution \citep{smith1985maximum}. 

% There are several practical issues that come with the generic multiple changepoint localizations in Algorithm~\ref{Algo:Multiple}.  
% Despite our best effort to determine an appropriate threshold \(T\) in Algorithm~\ref{Algo:Multiple} (which is a difficult problem in its own right), it more than often reports too many changepoints.  
% Moreover, since the reported changepoints are obtained in the narrowest interval, their individual estimation quality is in question. 
% These issues motivate us to take the two-stage approach as follows.

While Algorithm~\ref{Algo:Multiple} already produces good estimators of multiple changepoints, its performance can be further improved by the following post-processing procedures. Such post-processing has previously been described in e.g., \citet[Section~2 of online supplement]{fryzlewicz2018tail}.

Specifically, after we obtain an initial candidate set of changepoints using Algorithm~\ref{Algo:Multiple}, we iteratively run the test prescribed by \eqref{eq:test} on the largest time interval containing each candidate changepoint as the only estimated changepoint, and remove that candidate changepoint if the test is non-significant. For the remaining candidate changepoints $\hat{z}_1,\ldots,\hat{z}_{\hat\nu}$, we refine their estimated locations in two steps. We first perform a `midpoint' refinement, where we use (for instance) Algorithm~\ref{Algo:lassobic} to output a refined estimator $\tilde{z}_i$ based on data $\{(x_t, y_t):t\in ( (\hat{z}_{i-1} + \hat{z}_i)/2, (\hat{z}_{i} + \hat{z}_{i+1})/2 ]\}$ for each $i\in [\hat\nu]$. Here, we use the convention that $\hat{z}_0 = 0$ and $\hat{z}_{\hat\nu+1} = n$. Using midpoints between successive estimated changepoints ensures that each $( (\hat{z}_{i-1} + \hat{z}_i)/2, (\hat{z}_{i} + \hat{z}_{i+1})/2 ]$ contains with high probability at most one true changepoint. However, it does not use the full data available around each true changepoint. As such, we perform a second refinement step after this, where we use Algorithm~\ref{Algo:lassobic} to output a further refined estimator $\hat{z}_i^{\mathrm{refined}}$ based on data $\{(x_t, y_t):t\in ( \tilde{z}_{i-1} + \alpha n, \tilde{z}_{i+1} - \alpha n]\}$ for each $i\in [\hat\nu]$ with $\alpha$ being the burn-in parameter as in Algorithm~\ref{Algo:Multiple}. Again, we use the convention that $\tilde{z}_0 = 0$ and $\tilde{z}_{\hat\nu+1} = n$. For both refinement steps, we may also use  Algorithm~\ref{Algo:Estimation2} in place of Algorithm~\ref{Algo:lassobic}, and they have very similar performances in our numerical experiment.  For definiteness and simpler presentation, we employ Algorithm~\ref{Algo:lassobic} for both refinement steps in the following numerical experiments.

\begin{table}[htbp]
\centering
\begin{tabular}{rrrrrrrrrrr}
  \toprule
  $n$ &  $p$ &  $k$ &  $\rho_{\min}$ & \multicolumn{5}{c}{$\hat\nu - \nu$ value counts} &  Haus &  ARI\\
  \cmidrule(lr){5-9} 
  & & & &  $-3$ & $-2$ & $-1$ & $0$ & $1$ & &\\
\midrule
$1200$ & $200$ & $3$ & 0.8 & $0$ & $0$ & $96$ & $4$ & $0$ & $292.8$ & $0.742$\\
 &  &  & 1.2 & $0$ & $0$ & $22$ & $78$ & $0$ & $75.4$ & $0.918$\\
 &  &  & 1.6 & $0$ & $0$ & $0$ & $98$ & $2$ & $8.8$ & $0.978$\\
 \addlinespace
 &  & $10$ & 0.8 & $0$ & $2$ & $97$ & $1$ & $0$ & $304.9$ & $0.710$\\
 &  &  & 1.2 & $0$ & $0$ & $42$ & $55$ & $3$ & $141.1$ & $0.856$\\
 &  &  & 1.6 & $0$ & $0$ & $1$ & $96$ & $3$ & $18.0$ & $0.960$\\
 \addlinespace
 &  & $100$ & 0.8 & $3$ & $67$ & $30$ & $0$ & $0$ & $591.7$ & $0.303$\\
 &  &  & 1.2 & $0$ & $4$ & $88$ & $8$ & $0$ & $319.3$ & $0.611$\\
 &  &  & 1.6 & $0$ & $0$ & $52$ & $46$ & $2$ & $217.1$ & $0.759$\\
 \addlinespace
$2400$ & $400$ & $3$ & 0.8 & $0$ & $0$ & $25$ & $75$ & $0$ & $155.3$ & $0.881$\\
 &  &  & 1.2 & $0$ & $0$ & $0$ & $100$ & $0$ & $14.3$ & $0.975$\\
 &  &  & 1.6 & $0$ & $0$ & $0$ & $100$ & $0$ & $10.1$ & $0.983$\\
 \addlinespace
 &  & $10$ & 0.8 & $0$ & $15$ & $53$ & $32$ & $0$ & $376.9$ & $0.720$\\
 &  &  & 1.2 & $0$ & $0$ & $2$ & $98$ & $0$ & $37.3$ & $0.945$\\
 &  &  & 1.6 & $0$ & $0$ & $1$ & $99$ & $0$ & $21.0$ & $0.970$\\
 \addlinespace
 &  & $100$ & 0.8 & $42$ & $57$ & $1$ & $0$ & $0$ & $1154.9$ & $0.184$\\
 &  &  & 1.2 & $0$ & $32$ & $54$ & $14$ & $0$ & $647.0$ & $0.457$\\
 &  &  & 1.6 & $0$ & $0$ & $14$ & $84$ & $2$ & $376.9$ & $0.658$\\
\bottomrule
\end{tabular}
\caption{\label{Tab:Multiple}Summary of results of multiple changepoint estimations under (M1) and (M2) described in Section~\ref{sec:simu-multi} with $\rho\in\{0.8,1.2,1.6\}$ and $k\in\{3,10,100\}$.  The first nine rows of the table corresponds to setting (M1) and the last nine rows corresponds to (M2).
The columns of $\hat{\nu} - \nu$ tabulates the difference in number of estimated and true changepoints over 100 Monte Carlo repetitions. The `Haus' and `ARI' columns measure the average Hausdorff distance and the average adjusted rand index between the discovered partition and the true partition over 100 repetitions.  
}
\end{table}

We assume that the regression noise level $\sigma$ is known and consider the following two multiple changepoint specifications in our simulations: (M1) $n=1200$, $p=200$, $\nu=3$ and three changepoints are located at $z = (z_1,z_2,z_3) = (240, 540, 900)$ with signal sizes $(\|\theta^{(1)}\|_2,\|\theta^{(2)}\|_2,\|\theta^{(3)}\|_2) = \rho_{\min}\times  (1, 1.5, 2)$ and sparsity $\|\theta^{(1)}\|_0=\|\theta^{(2)}\|_0=\|\theta^{(3)}\|_0=k$ respectively for various $\rho_{\min}$ and $k$; (M2) $n=2400$, $p=400$, $\nu=4$ and four changepoints are located at $z=(z_1,z_2,z_3,z_4) = (720,1320,1800,2160)$ with signal sizes $(\|\theta^{(1)}\|_2,\|\theta^{(2)}\|_2,\|\theta^{(3)}\|_2,\|\theta^{(4)}\|_2) = \rho_{\min} \times (1,1.15,1.45,2.18)$ and sparsity $\|\theta^{(1)}\|_0=\|\theta^{(2)}\|_0=\|\theta^{(3)}\|_0 = \|\theta^{(4)}\|_0=k$ respectively for various $\rho_{\min}$ and $k$.
% \begin{enumerate}[label={(M\arabic*)}]
% \item $n=1200$, $p=200$, $\nu=3$ and three changepoints are located at $z = (z_1,z_2,z_3) = (240, 540, 900)$ with signal sizes $(\|\theta^{(1)}\|_2,\|\theta^{(2)}\|_2,\|\theta^{(3)}\|_2) = \rho_{\min}\times  (1, 1.5, 2)$ and sparsity $\|\theta^{(1)}\|_0=\|\theta^{(2)}\|_0=\|\theta^{(3)}\|_0=k$ respectively for various $\rho_{\min}$ and $k$.
% \item $n=2400$, $p=400$, $\nu=4$ and four changepoints are located at $z=(z_1,z_2,z_3,z_4) = (720,1320,1800,2160)$ with signal sizes $(\|\theta^{(1)}\|_2,\|\theta^{(2)}\|_2,\|\theta^{(3)}\|_2,\|\theta^{(4)}\|_2) = \rho_{\min} \times (1,1.15,1.45,2.18)$ and sparsity $\|\theta^{(1)}\|_0=\|\theta^{(2)}\|_0=\|\theta^{(3)}\|_0 = \|\theta^{(4)}\|_0=k$ respectively for various $\rho_{\min}$ and $k$.
% \end{enumerate}

Note that for (M2), the signal sizes are chosen such that 
\[
\|\theta^{(i)}\|_2^2\frac{(z_i - z_{i-1})(z_{i+1}-z_i)(z_{i+1}-z_{i-1}-p)}{(z_{i+1}-z_{i-1})^2}
\]
is approximately constant for each $i\in[\nu]$, which according to \citet{gao2021twosample} means that the effective signal-to-noise ratio of testing for each changepoint $z_i$ within the interval $(z_{i-1}, z_{i+1}]$ is almost constant. Table~\ref{Tab:Multiple} reports the multiple changepoint estimation performances for both (M1) and (M2) with $\rho_{\min}\in\{0.8,1.2,1.6\}$ and $k\in\{3,10,100\}$. The multiple changepoint estimation accuracy is measured in terms of the difference between the number of estimated and true changepoints, the average Hausdorff distance between the sets $\{z_i:i\in[n]\}$ and $\{\hat z^{\mathrm{refined}}_i:i\in[\hat n]\}$ and finally the average adjusted Rand index (ARI) \citep{rand1971objective} of the estimated segments against the truth, over 100 Monte Carlo repetitions. We see from Table~\ref{Tab:Multiple} that the promising single changepoint estimation performance of our methodology carries over to the multiple changepoint settings. 

Figure~\ref{Fig:Multi} visualizes the simulation results by showing the histograms of estimated changepoints in four of the parameter settings shown in Table~\ref{Tab:Multiple}. It is worth noting that in the bottom two panels of the figure, where the effective signal-to-noise ratios are chosen to be approximately constant for all the four changepoints, we indeed see a similar number of times in identifications of each changepoint.

%We then conduct the numerical experiment of applying the above two-stage multiple changepoint estimation to $18$ different settings, for each of which we repeat the experiment $100$ times.  The results are aggregated and presented in Table~\ref{Tab:Multiple} for an overview.  
%We also lay out the histograms of the estimated changepoint locations in Figure~\ref{Fig:Multi} for an easy visualization of the estimation performances in different settings. 
%Note that Figures~\ref{fig:3a}, \ref{fig:3b}, \ref{fig:3c} and \ref{fig:3d} corresponds to the first, the third, the 13th and the 15th lines of Table~\ref{Tab:Multiple}.
%
%For each change, we may calculate the signal-to-noise ratio for each change by $n \kappa_1$ with $\kappa_1 = z(n-z)(n-p)/n^3$ and $n$ being the length of the interval where the changepoint in question is the only one, see the results and discussions about the `effective' sample size in \cite{gao2021twosample}. 
%For instance, for the settings $(n,p) = (1200,400)$ with true changes at fractions $(0.2,0.45,0.75)$ as prescribed in the first 9 lines in Table~\ref{Tab:Multiple}, the first changepoint at $z = 240$ is the only change in the interval $[1,540]$, resulting in the signal-to-noise ratio of $240*300*(540-400)/540^3 = $
%The signal-to-noise ratio for the three-changepoint settings (the above 9 lines in Table~\ref{Tab:Multiple}) are $(1,1.748,2.331)$. 
%The signal sizes for the four-changepoint settings (the 9 lines below in Table~\ref{Tab:Multiple}) are chosen such that the signal-to-noise ratios are about the same for each change. 

\begin{figure}[htbp]
  \centering
  \begin{subfigure}[b]{0.450\textwidth}
    \includegraphics[width=\textwidth]{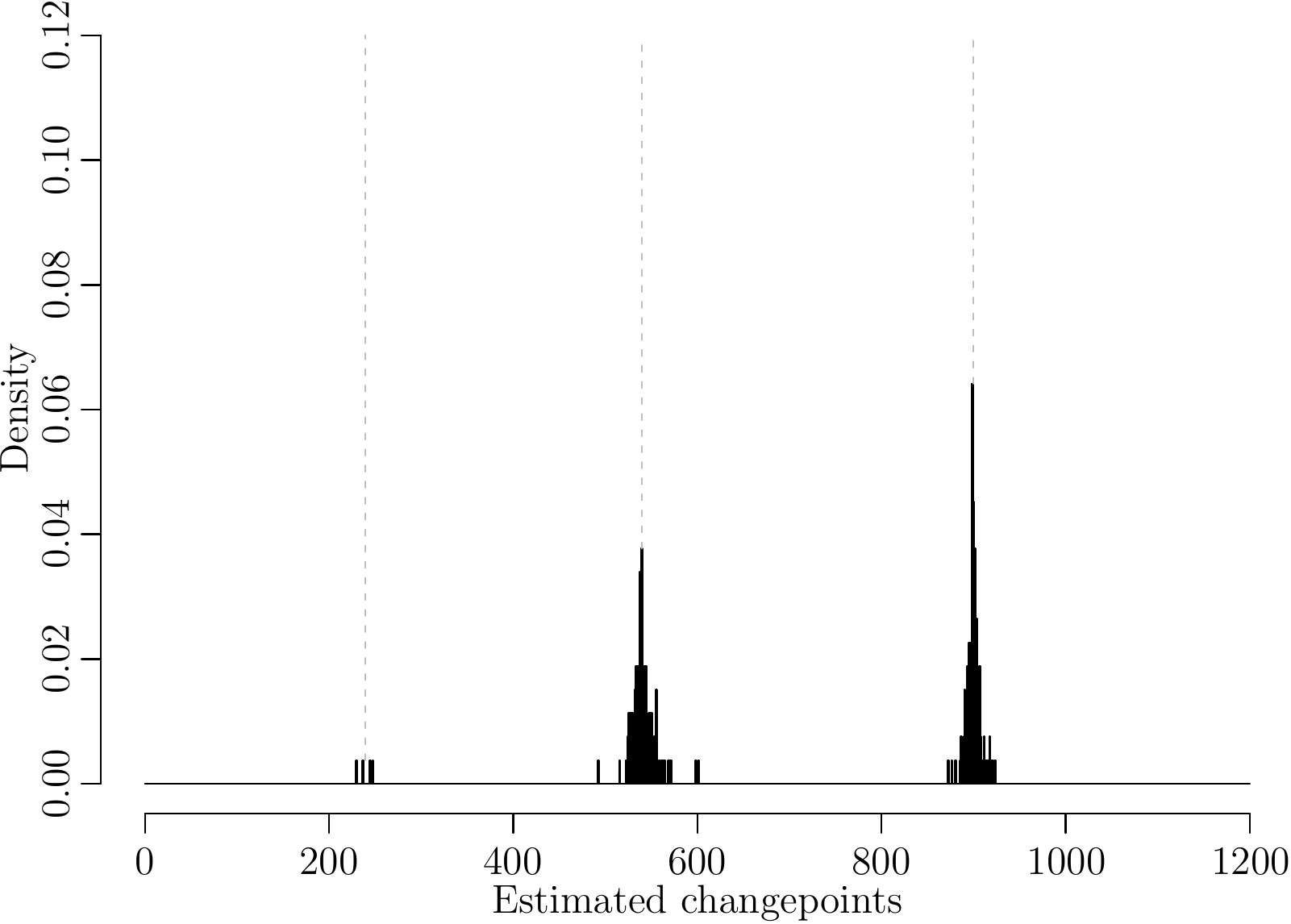}
    \caption{\label{fig:3a}$n=1200$, $p = 200$, $k=3$, $\rho_{\min} = 0.8$}
  \end{subfigure}
  ~ 
  \begin{subfigure}[b]{0.450\textwidth}
    \includegraphics[width=\textwidth]{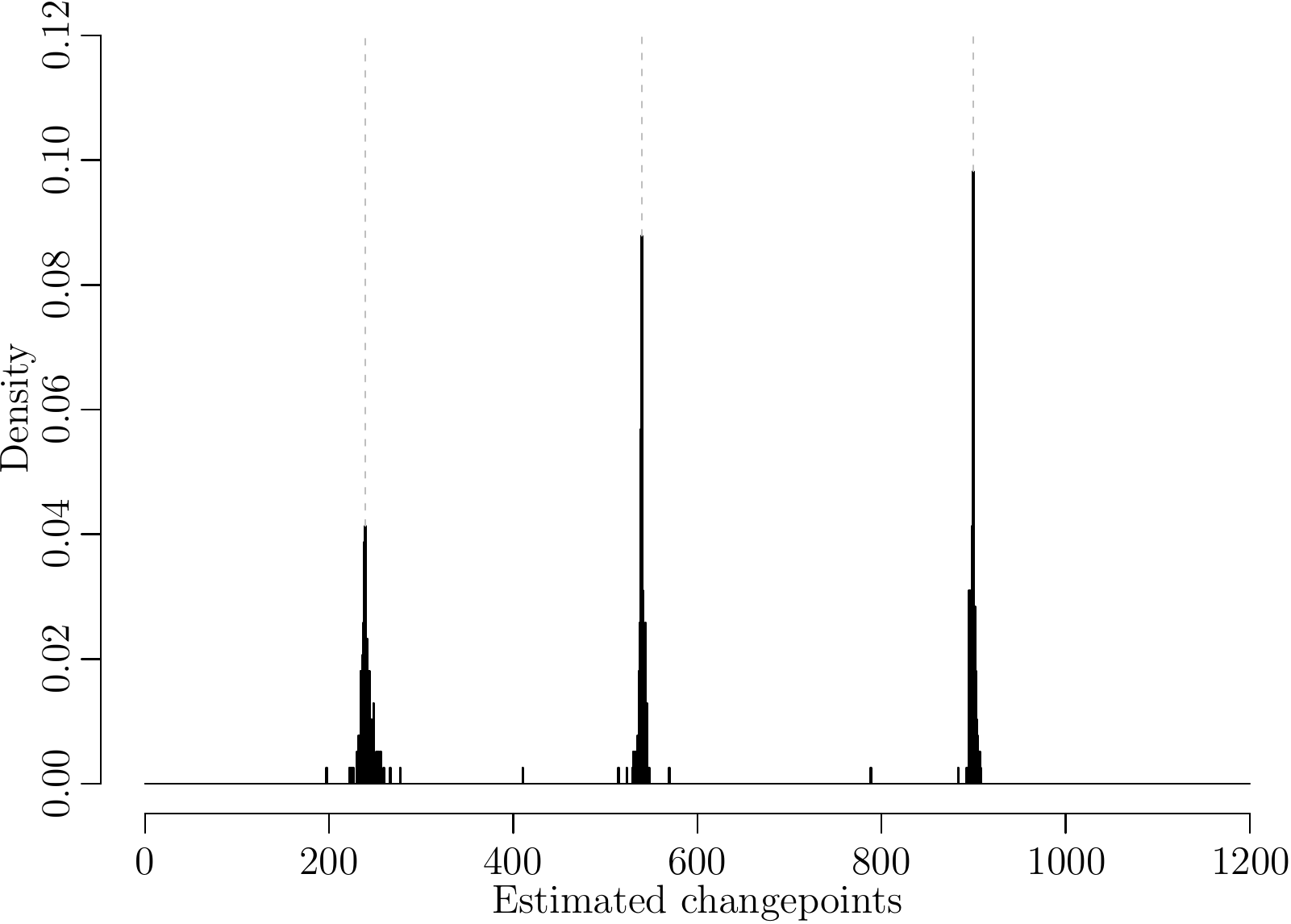}
    \caption{\label{fig:3b}$n=1200$, $p = 200$, $k=3$, $\rho_{\min} = 1.6$}
  \end{subfigure}
  \\ 
  \begin{subfigure}[b]{0.450\textwidth}
    \includegraphics[width=\textwidth]{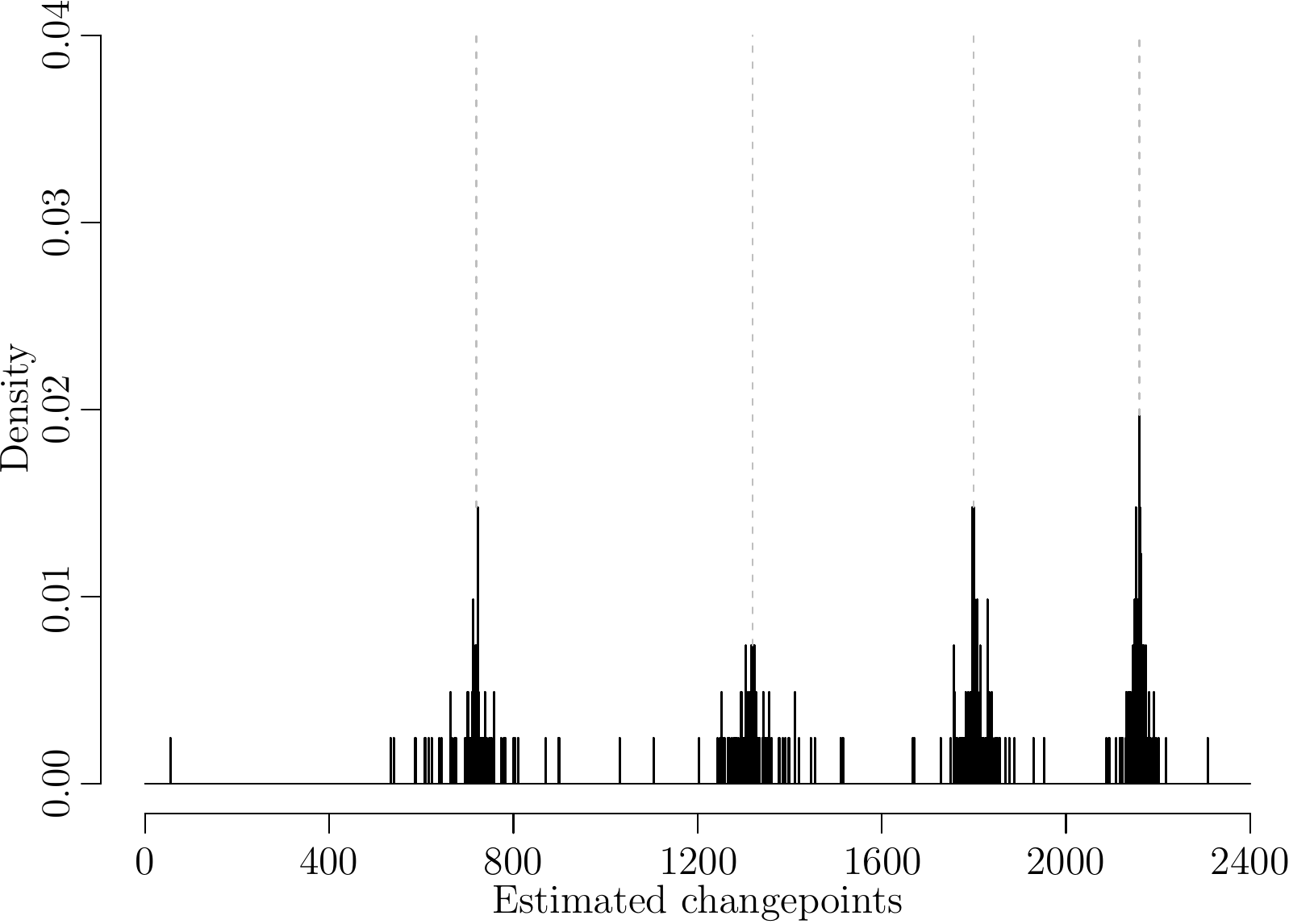}
    \caption{\label{fig:3c}$n=2400$, $p = 400$, $k=10$, $\rho_{\min} = 0.8$}
  \end{subfigure}
  ~
  \begin{subfigure}[b]{0.450\textwidth}
    \includegraphics[width=\textwidth]{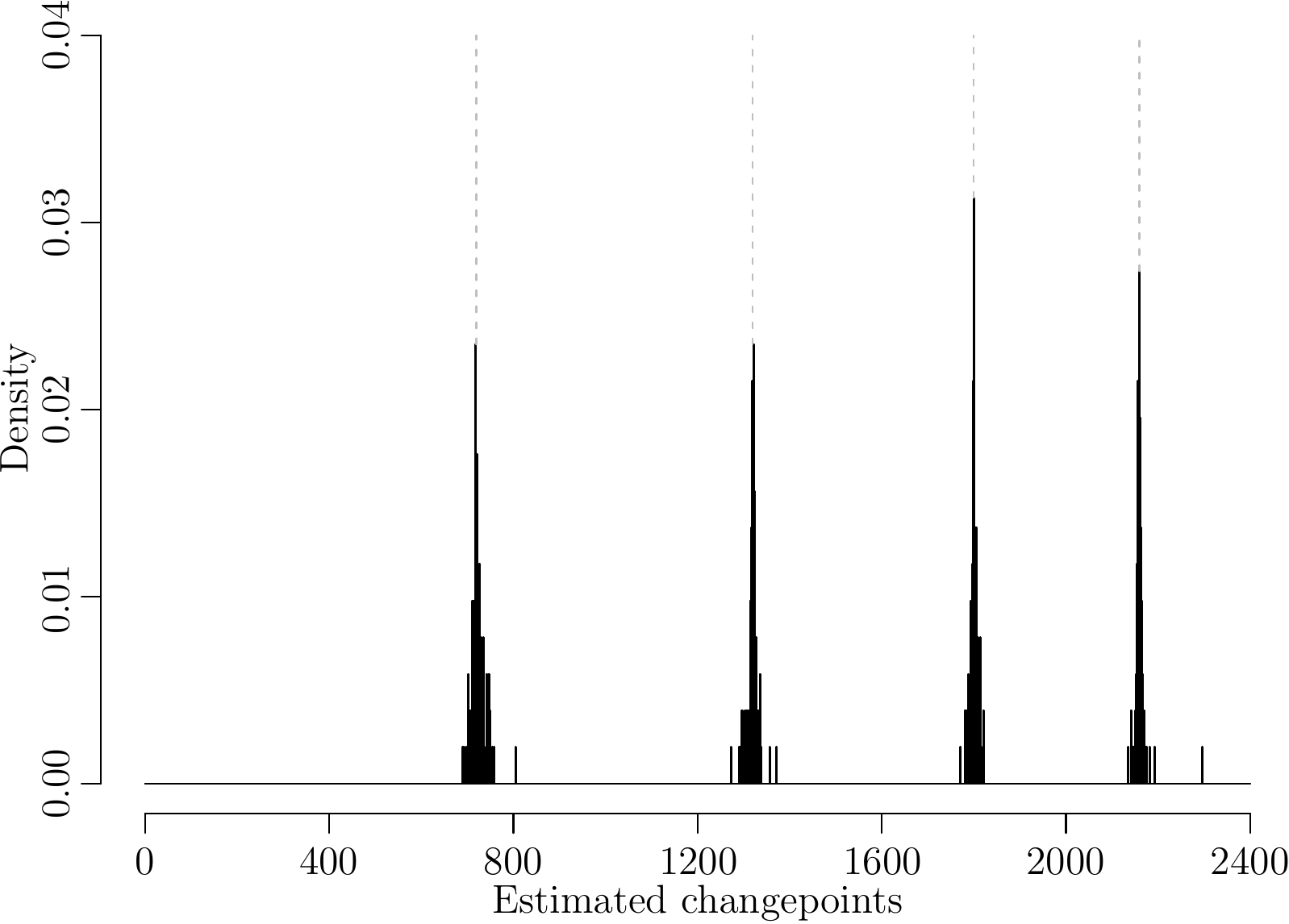}
    \caption{\label{fig:3d}$n=2400$, $p = 400$, $k=10$, $\rho_{\min} = 1.6$}
  \end{subfigure}
\caption{\label{Fig:Multi}Histogram of estimated changepoint locations in four settings.  The true changes take place at $z=(240,540,900)$ for the $(n,p)=(1200,200)$ specifications in the above two panels with the signal strengths at respective changes being $\rho_{\min} \times (1,1.5,2)$.  For the two panels below with $(n,p) = (2400,400)$, $z=(720, 1320, 1800,2160)$ with $\rho_{\min}\times (1,1.15,1.45,2.175)$.  The locations of true changes are marked in lightly-coloured dashed vertical lines in each plot.}
\end{figure}
\subsection{Real data example}
\label{sec:realdata}
In this subsection, we showcase how the \charcoal algorithm can be applied to a single-cell gene expression dataset from \citet{suo2023dandelion}. The original data consists of the logarithmic normalized gene expression levels of 3211 genes measured in 11853 cells along the developmental trajectory from proliferating double positive cells (DP(P) cells) to quiescent double positive cells (DP(Q) cells), $\alpha\beta$T entry cells (ABT cells) and finally to CD4+ T cells. These cells have been ordered in pseudotime according to their development stage in \citet{suo2023dandelion}, which we use as our timeline (see Figure~\ref{Fig:Pseudotime}). 
We are interested in understanding the change in the gene regulatory networks along this time trajectory. We can estimate the changepoints by modelling the logarithmic normalized expressions using Gaussian graphical models and seek changes in the nodewise regression coefficients of each gene against the remaining genes. To speed up the computation, we preprocess the data by subsampling $1/3$ of the original cells and only using genes that have non-zero expression in at least $5\%$ of the cells. Our preprocessed data is available on the GitHub repository. The changepoints are estimated using Algorithm~\ref{Algo:Estimation2} with tuning parameters chosen as suggested in Section~\ref{sec:tuning-variants}. In Table~\ref{Tab:ListGenes}, we list the genes that reported most significant test statistics in their nodewise regression coefficients along this pseudotime trajectory. From Figure~\ref{Fig:Pseudotime}, we see that most of the changes are identified immediately before the boundary between the DP(P) and DP(Q) boundary, and most of the associated genes TK1, CKAP2L, TTK, ARHGEF39, DEPDC1, SPC25, GTSE1, HMMR, CENPA are well-known regulators for cell proliferation in biology \citep[][see, e.g.]{bitter2020thymidine, mills1992expression, zhou2018arhgef39, zhang2019depdc1, guo2016silencing}. The change in nodewise regression coefficient of the RAG2 gene occurred immediately before the DP(Q) and ABT boundary, which agrees with the existing literature that RAG2 is a regulator for T cell development \citep{kalman2004mutations}.

\begin{figure}[htbp]
\begin{center}
\includegraphics[width=0.6\textwidth]{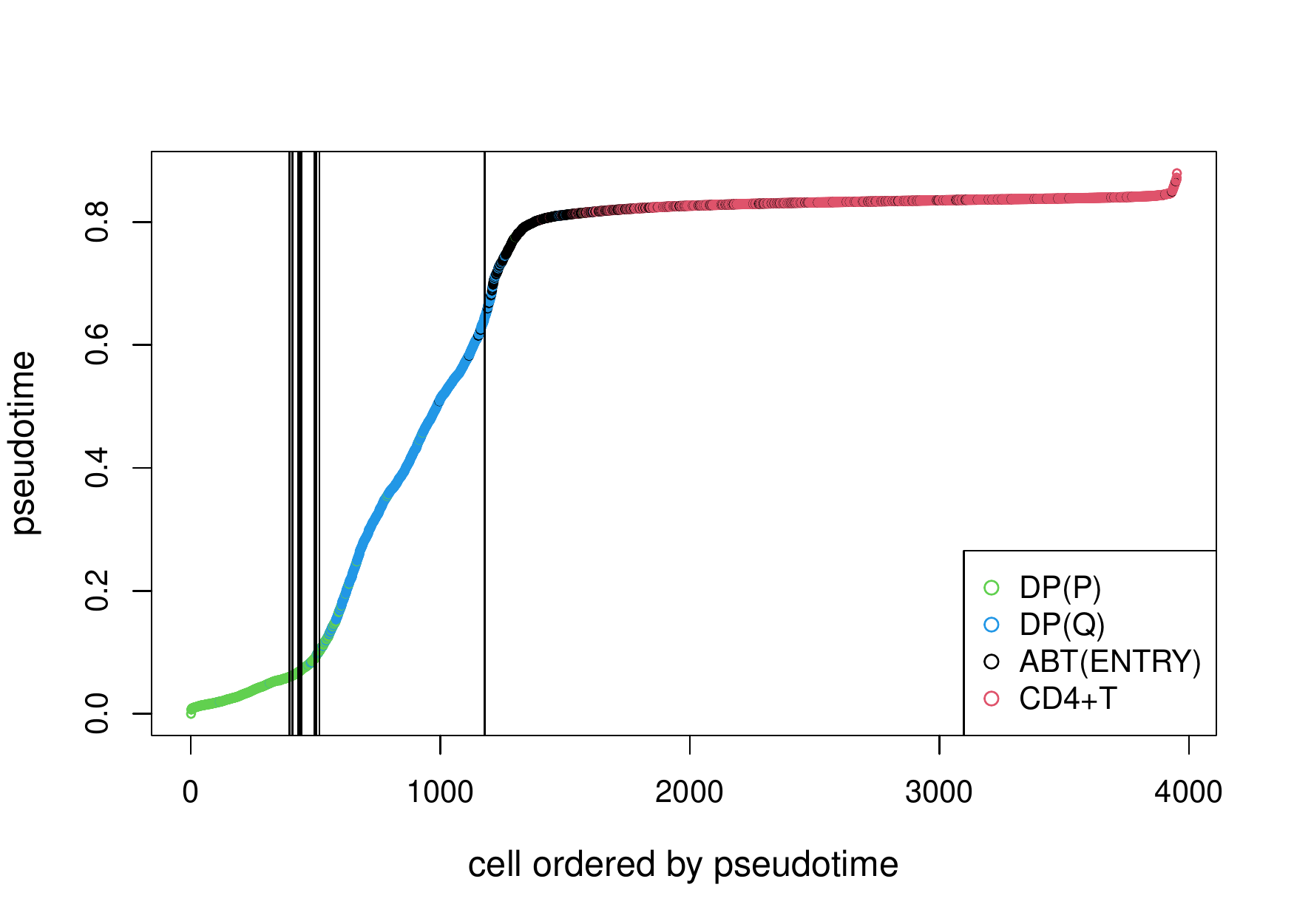}
\end{center}
\caption{\label{Fig:Pseudotime}Ordered pseudotime of cells in the real data example of Section~\ref{sec:realdata}. Each plot point represents a cell, coloured by its annotated cell type. Vertical lines corresponds to estimated changepoint locations of the most significant changes in the nodewise regression coefficients as described in Section~\ref{sec:realdata}.}
\end{figure}
\begin{table}
\begin{center}
{\small
\begin{tabular}{lcl}
\toprule
Gene & Changepoint & Top interacting partners\\
\midrule
TK1 & 495 & 
PTP4A1, KIF20B, CENPF, SPRY1, ZWINT, \\
& & CALCRL, RHPN1, LYPD3, COMMD3, LINC00672\\[3pt]
CKAP2L & 430 & 
TOGARAM1, DEPDC1B, AP001816.1, TUBA1C, \\
& & FANCI, IGF2BP2, REC8, TOB1, ZMAT3, STK11IP\\[3pt]
RAG2 & 1178 & 
SMPD3, AL365440.2, LZTFL1, AEBP1, HIST1H2BJ, \\
& & MTSS1, CD1C, CSNK2B, CASC15, SLC29A1\\[3pt]
TTK & 407 & 
UBE2S, KNL1, CDC20, TRAV19, DDIT3, AC023157.3,\\
& &  AC012360.3, IL23A, DNTT, USP53\\[3pt]
ARHGEF39 & 396 & 
HJURP, CD1A, SLC25A25, CCDC152, MBTD1, PON1, \\
& & H1F0, RNF125, APH1B, DDX3Y\\[3pt]
DEPDC1 & 444 & 
AL138899.1, USPL1, RIPK4, SERPINF1, EPHB6,\\
& &   MTSS1, SLC8A1-AS1, SLC5A3, HDAC4, SGK1\\[3pt]
SPC25 & 442 & 
ATF3, ITGAE, CDC42EP3, AC136475.5, EPS8, NINJ2,\\
& &   NDC80, ZNF280D, L3MBTL3, FBLN5\\[3pt]
GTSE1 & 437 & 
MID1IP1, HIST1H2AG, GADD45G, PSRC1, FBLN5,\\
& &  HIST1H2BN, CCDC171, ARMH1, SERTAD2, DDX3Y\\[3pt]
HMMR & 516 &
TAX1BP3, LAIR1, SERP2, LANCL2, MANEA-DT,\\
& &   TIMP1, CSRNP1, TSGA10, CKAP5, RGS16\\[3pt]
CENPA & 503 &  
TRBV7-3, SOCS1, FRMD4B, CDKN1A, FXYD2, \\
& & PTPN12, NLGN4X, NINL, KLRG1, SWT1\\
\bottomrule
\end{tabular}
}
\end{center}
\caption{\label{Tab:ListGenes}List of genes with most significant changes in the nodewise regression coefficients, together with their changepoint locations and top 10 interacting partners.}
\end{table}
\newpage
% \appendix

\section{Proof of main results}
\label{Sec:Proof}

\begin{proof}[Proof of Proposition~\ref{Prop:WtWz}]
	Define $\kappa_1 = \kappa_1(n,z,p) := z(n-z)(n-p)/n^3$, which under Condition~\ref{cond:regime} is $O(1)$.
	We decompose
	\begin{align*}
		W_t^\top W_z  - \frac{4tn\kappa_1}{z}I_p & = W_t^\top W_z - \frac{t}{z} W_z^\top W_z + \frac{t}{z} W_z^\top W_z - \frac{4tn\kappa_1}{z} I_p                                                                              \\
		% & = W_t^\top W_z - \frac{n-z}{n-t} W_t^\top W_t + \frac{n-z}{n-t} W_t^\top W_t -  \frac{n-z}{n-t}\mathbb{E} W_t^\top W_t. \\
		                                         & = {4t}  \biggl\{  \biggl( \frac{S_{0,t}}{t} -  \frac{S_{0,z}}{z}\biggr) S_{0,n}^{-1} S_{z,n}   + \frac{n}{z} \biggl( \frac{W_z^\top W_z}{4n} -  \kappa_1 I_p\biggr) \biggr\}.
	\end{align*}
	% where we employ Lemma~\ref{Lemma:WtWz} in the second equality. 

	Define $\Omega_\nu := \{ \|S_{0,n}^{-1} S_{z,n}\|_{\mathrm{op}} \| S_{0,z}/z\|_{\mathrm{op}} \le  \nu\}$.
	By Lemma~\ref{Lem:Net}, we have a $(1/2)$-net $\mathcal{N}$ of $B_0(\ell)$ with cardinality at most $(5ep/\ell)^\ell$ such that we have
	\begin{align*} % place holder psi
		% \mathbb{P} ( \sup_{v \in B_0(\ell)} v^\top (S_{0,t}/t - S_{0,z}/z)S_{0,n}^{-1} v \ge t4\kappa_1) \le 
		\mathbb{P}( & \sup_{u \in B_0(\ell)} |u^\top(S_{0,t}/t - S_{0,z}/z)S_{0,n}^{-1}S_{z,n}  v | \ge x)                                                                                                             \\
		            & \le \mathbb{E}[\mathbb{P}( 2 \sup_{u \in \mathcal{N} } |u^\top(S_{0,t}/t - S_{0,z}/z)S_{0,n}^{-1} v | \ge x \mid S_{0,z}, S_{z,n}) \mathbbb{1}_{\Omega_\nu}] + \mathbb{P}(\Omega_\nu^{\mathrm{c}}) \\
		            & \le 52 \biggl( \frac{5ep}{\ell}\biggr)^{\ell} \exp\{ - t^2 x^2 /(32 z \nu^2\})   + \mathbb{P}(\Omega_\nu^{\mathrm{c}}),                                                                          % \\
		% & \le \exp\{- t^2x^2/(4z\nu^2)\} + \exp\biggl( - \frac{n \bigl( 1- \sqrt{p/n} \bigr)^2}{8} \biggr) + 2 \exp(-p/8). 
	\end{align*}
	where the first inequality holds by that $\Omega_\nu$ is measurable with respect to the $\sigma$-algebra generated by $(S_{0,z}, S_{z,n})$ and the second by Lemma~\ref{lem:S0t-S0z} and a union bound.
	Define $\lambda_{\min}(A)$ and $\lambda_{\max}(A)$ for any generic symmetric matrix $A$ to be the smallest and largest eigenvalues of $A$, respectively.
	By \citet[Theorem~6.1]{Wainwright2019}, we have
	\begin{gather*}
		\mathbb{P}\biggl(\lambda_{\max}\biggl(\frac{S_{z,n}}{n-z}\biggr) \ge 2\bigl(1 + \sqrt{p/(n-z)}\bigr)^2\biggr) \le \exp\biggl( -\frac{(n-z) \bigl(1+\sqrt{p/(n-z)}\bigr)^2}{8} \biggr),  \\
		\mathbb{P}\biggl(\lambda_{\max}\biggl(\frac{S_{0,z}}{z}\biggr) \ge 2\bigl(1 + \sqrt{p/z}\bigr)^2\biggr) \le \exp\biggl( -\frac{z \bigl(1+\sqrt{p/z}\bigr)^2}{8} \biggr),  \\
		\mathbb{P}\biggl(\lambda_{\min}\biggl(\frac{S_{0,n}}{n}\biggl)  \le \frac{(1-  \sqrt{{p}/{n}})^2}{4} \biggr) \le \exp\biggl( -\frac{n ( 1- \sqrt{p/n})^2}{8}\biggr).
	\end{gather*}
	We define $\nu:= 16[(n-z)/n]  ( 1 + \sqrt{{p}/{z}} )^2 ( 1 + \sqrt{{p}/(n-z)} )^2  ( 1 - \sqrt{{p}/{n}} )^{-2}$.
	By a union bound, we arrive at
	\begin{equation*}
		\mathbb{P}(\Omega_\nu^{\mathrm{c}} ) =
		\mathbb{P}( \| S_{0,n}^{-1} S_{z,n} \|_{\mathrm{op}} \| S_{0,z}/z\|_{\mathrm{op}}  \le \nu )  \le \exp\biggl( - \frac{n \bigl( 1- \sqrt{p/n} \bigr)^2}{8} \biggr) + 2 \exp(-p/8).
	\end{equation*}
	Combining the above displays and setting $x := 16\nu  \sqrt{ \ell z \log (5ep/\ell)}/t$, we have by a union bound that %\teal{[check $p^{-6}$]}
	\begin{equation*}
		\begin{aligned}
			\mathbb{P}\biggl(  \sup_{u \in B_0(\ell)}  |u^\top(S_{0,t}/t - S_{0,z}/z)S_{0,n}^{-1}  v | \ge x \biggr)  \le p^{-7} + e^{- {n ( 1- \sqrt{p/n} )^2}/{8} } + 2 e^{-p/8}.
		\end{aligned}
	\end{equation*}
	Taking another union bound over $t\in[z]$, and by the Borel--Cantelli lemma, we have with probability 1, for all but finitely many $p$'s that
	\begin{equation}
		\label{Eq:Result1}
		\sup_{t\in[z]} \sup_{u\in B_0(\ell)} u^\top \bigl( W_t^\top W_z - (t/z) W_z^\top W_z \bigr) v \leq C'_{\tau,\eta}\sqrt{\ell p \log(ep/\ell)},
	\end{equation}
	for some constant $C'_{\tau,\eta} > 0$ that depends only on $\tau$ and $\eta$.
	By \citet[(22)]{gao2021twosample}, we have for all but finitely many $p$'s that
	\[
		\sup_{u\in B_0(\ell)} u^\top \biggl(\frac{W_z^\top W_z}{4n} - \kappa_1 I_p\biggr) u \le (4 + o(1)) \sqrt{\frac{(\ell+4) \log(10ep/\ell)}{n}} \bigl\{(\kappa_1 + \kappa_2)\sqrt{n/p} + \kappa_1\bigr\},
	\]
	where $\kappa_2>0$ is again a constant depending only on $\tau$ and $\eta$. This, together with the first claim of Lemma~\ref{Lem:kOpNorm}, implies that with probability 1, for all but finitely many $p$'s, we have
	\begin{equation}
		\label{Eq:Result2}
		\sup_{t\in[z]}\sup_{u\in B_0(\ell)} \frac{4tn}{z} u^\top \biggl(\frac{W_z^\top W_z}{4n} - \kappa_1 I_p\biggr)v \leq C''_{\tau,\eta}\sqrt{\ell p\log(ep/\ell)}.
	\end{equation}
	The conclusion follows by combining~\eqref{Eq:Result1} and~\eqref{Eq:Result2}, and the corresponding inequality for $t \in [n-1]\setminus [z]$.
\end{proof}

% \begin{proof}[Proof of Theorem~\ref{Thm:LocalisationRate}]
\begin{proof}[Proof of Theorem~\ref{thm:test}]
	% By symmetry, it suffices to consider the case that $\hat z \leq z$.  
  % We first show that $H_t$ can be viewed as a perturbation of $h_t$ by some small-order term. 
  Applying Proposition~\ref{Prop:WtWz}, we have with probability 1 that for all but finitely many $p$'s that
	\begin{equation}
		\sup_{t\in[z]} \sup_{u\in B_0(\ell)} u^\top \biggl(W_t^\top W_z - \frac{4t(n-z)(n-p)}{n^2} I_p\biggr) \frac{\theta}{\|\theta\|_2} \lesssim_{\tau,\eta} \sqrt{p\ell\log p}.
		\label{Eq:tmp1}
	\end{equation}
	Let $S := \mathrm{supp}(\theta)$.  Taking $\ell=k$ in~\eqref{Eq:tmp1}, we have
	\begin{align}
		\sup_{{\alpha}n\leq t\leq z} & \,\biggl\|\sqrt\frac{n}{t(n-t)}\biggl(\biggl(W_t^\top W_z - \frac{4t(n-z)(n-p)}{n^2} I_p\biggr)\theta\biggr)_S\biggr\|_2  \nonumber                                               \\
		                             & \leq \sup_{{\alpha} n\leq t \leq z} \sup_{u\in B_0(k)} u^\top \biggl(\sqrt\frac{n}{t(n-t)}W_t^\top W_z - \frac{4t^{1/2}(n-z)(n-p)}{n^{3/2}(n-t)^{1/2}} I_p\biggr) \theta\nonumber \\
		                             & \lesssim_{\tau,\eta,{\alpha}} \sqrt{k\log p}\|\theta\|_2.\label{Eq:SignalCoordinates}
	\end{align}

	By~\eqref{Eq:tmp1} and the second claim of Lemma~\ref{Lem:kOpNorm}, with probability 1 that for all but finitely many $p$'s,
	\begin{align*}
		\sup_{t\in[z]} \|(W_t^\top W_z \theta)_{S^c}\|_2 & = \sup_{t\in[z]}\, \biggl\|\biggl(\biggl\{W_t^\top W_z - \frac{4t(n-z)(n-p)}{n^2}I_p\biggr\} \theta\biggr)_{S^c}\biggr\|_2 \\
		                                                 & \lesssim_{\tau,\eta} p\sqrt{\log p}\|\theta\|_2.
	\end{align*}
	For any $Q\in\mathbb{O}^{p\times p}$, we have $X \stackrel{\mathrm{d}}{=} X Q^\top =: \tilde{X}$, and the latter has the corresponding sketching matrix $\tilde{A} = A Q^\top$ because $\tilde{A}^\top \tilde{A} = I_{n-p}$ and $\tilde{A}^\top \tilde{X} = Q A^\top X Q^\top =\mathbf{0}_{(n-p)\times p}$.
	As such, for any $Q\in\mathbb{O}^{p\times p}$ such that $Q\theta = \theta$, $Q(W_t^\top W_z\theta) = (QW_tQ^\top)^\top (QW_zQ^\top)   \theta \stackrel{\mathrm{d}}= W_t^\top W_z\theta$. In particular, $(W_t^\top W_z\theta)_{S^c}$ is spherically symmetric on $\mathbb{R}^{p-k}$. Hence, by Lemma~\ref{Lem:Sphere} (with a choice of $\delta=2p^{-4}$), with probability 1 we have for all but finitely many $p$'s that
	\begin{align}
		\sup_{{\alpha} n\leq t\leq z}\, \biggl\|\sqrt\frac{n}{t(n-t)}(W_t^\top W_z \theta)_{S^c}\biggr\|_\infty & \leq  \sup_{{\alpha} n\leq t\leq z} \|(W_t^\top W_z \theta)_{S^c}\|_2\sqrt\frac{n}{t(n-t)}\frac{4\sqrt{\log p}}{\sqrt{p-k}}\nonumber \\
		                                                                                                        & \lesssim_{\tau,\eta,{\alpha}} \|\theta\|_2 \log p\label{Eq:NoiseCoordinates}
	\end{align}
	Let $X = QT$ be the QR decomposition of $X$ and define $B_t:=Q_{(0,t]}^\top Q_{(0,t]}$. By Equation (16) in the proof of \citet[Proposition~8]{gao2021twosample}, there exists $C_{\eta}>0$, depending only on $\eta$ that for any fixed $p$ and $t$ and $j\in[p]$, with probability $1-p^{-4}$, we have
	\[
		\frac{1}{n}(W_t^\top W_t)_{j,j} \leq \frac{4\mathrm{tr}(B_t(I_p-B_t))}{p} + C_{\eta}\sqrt\frac{\log p}{n} \leq 1 + C_{\eta} \sqrt\frac{\log p}{n},
	\]
	where the final inequality follows from the fact that $\|B_t\|_{\mathrm{op}}\leq 1$. Taking union bounds over $j\in[p]$ and $t\in[z]$, and applying the Borel--Cantelli lemma, we have with probability 1 that for all but finitely many $p$'s,
	\[
		\sup_{\alpha n\leq t\leq z}\sup_{j\in[p]} \, \frac{n}{t(n-t)}(W_t^\top W_t)_{j,j}  \lesssim_{\tau,{\alpha}}1.
	\]
	Furthermore, applying the Gaussian tail bound followed by a union bound, we have with probability 1 for all but finitely many $p$'s that
	\begin{equation}
		\label{Eq:NoiseContribution}
		\sup_{\alpha n\leq t\leq z}\biggl\|\sqrt\frac{n}{t(n-t)}W_t^\top \xi\biggr\|_\infty \leq 4\sqrt{\log p} \sup_{\alpha n\leq t\leq z}\sup_{j \in[p]} \, \frac{n}{t(n-t)}(W_t^\top W_t)_{j,j}\lesssim_{\tau, {\alpha}}\sqrt{\log p}.
	\end{equation}

	We now work on the probability 1 event $\Omega$, such that~\eqref{Eq:SignalCoordinates}, \eqref{Eq:NoiseCoordinates}, \eqref{Eq:NoiseContribution} all hold for all but finitely many $p$'s.

	For sufficiently large $c_{\tau,\eta,{\alpha}}$, we have the right-hand side of~\eqref{Eq:NoiseCoordinates} and~\eqref{Eq:NoiseContribution} are both dominated by $\lambda/2$. Hence, on $\Omega$, we have for all $t\in[\alpha n,z]$ that
	\begin{align*}
		H_t & = \biggl\|\soft\biggl(\sqrt\frac{n}{t(n-t)}W_t^\top (W_z  \theta +\xi), \lambda\biggr)\biggr\|_2                                          \\
		    & =\biggl\|\soft\biggl(\sqrt\frac{n}{t(n-t)}(W_t^\top W_z)_{S,S} \theta_S + \sqrt\frac{n}{t(n-t)}(W_t^\top \xi)_S, \lambda\biggr)\biggr\|_2
	\end{align*}
	Writing $\tilde H_t := \bigl\|\sqrt\frac{n}{t(n-t)}(W_t^\top W_z)_{S,S} \theta_S\bigr\|_2$, we have by the triangle inequality and~\eqref{Eq:NoiseContribution} that on $\Omega$,
	\begin{equation}
		\label{Eq:HtHttilde}
		\sup_{t\in[\alpha n,z]} \bigl|H_t - \tilde H_t \bigr|\lesssim_{\tau,{\alpha}} \sqrt{k}\lambda + \sqrt{k\log p}.
	\end{equation}
	Recall the definition of $\gamma_t$ in \eqref{Eq:gamma_t} and write
	\[
		h_t:= \gamma_t \|\theta\|_2 =  \frac{4t^{1/2}(n-z)(n-p)}{n^{3/2}(n-t)^{1/2}}\|\theta\|_2= \frac{4(n-p)}{n}\sqrt\frac{t}{n(n-t)}(n-z)\|\theta\|_2.
	\]
	Then by~\eqref{Eq:SignalCoordinates}, we also have on $\Omega$ that
	\begin{equation}
		\label{Eq:Httildeht}
		\sup_{t\in[{\alpha}n, z]} |\tilde H_t - h_t | \lesssim_{\tau,\eta,{\alpha}} \sqrt{k\log p}\|\theta\|_2.
	\end{equation}
  \begin{comment}
	Since $\hat{z} = \argmax_{t \in [\alpha n, (1-\alpha)n]} H_t$ $H_{\hat z} \geq H_{z}$ and $h_z \geq h_{\hat z}$, we have
	\begin{align}
		h_z - h_{\hat z} & = (h_z- H_z) + (H_{\hat z}- h_{\hat z}) + (H_z - H_{\hat z}) \nonumber                                                       \\
		                 & \leq |H_z-\tilde H_z| + |\tilde H_z - h_z| + |H_{\hat z} - \tilde H_{\hat z}| + |\tilde H_{\hat z} - h_{\hat z}|\nonumber    \\
		                 & \lesssim_{\tau,\eta,{\alpha}} \sqrt{k\log p}(1+\|\theta\|_2) + \sqrt{k}\lambda \lesssim_{\tau,\eta,{\alpha}}\sqrt{k}\lambda.
		\label{Eq:hDiff}
	\end{align}
	By \citet[Lemma~7]{wang2018high}, we have for all $t \in [z - \min(z,n-z)/2, z]$ that
	\begin{equation}
		\label{Eq:hGrad}
		h_z - h_t \geq \frac{8}{3\sqrt{6}}\frac{n-p}{p}\frac{|z-t|}{\sqrt{\min(z,n-z)}}\|\theta\|_2.
	\end{equation}
	Thus, combining~\eqref{Eq:hDiff} and~\eqref{Eq:hGrad}, together with the fact that $t\mapsto h_t$ is increasing on $(0,z]$, we have on $\Omega$ that
	\[
		\frac{|\hat z - z|}{n} \lesssim_{\tau,\eta,{\alpha}} \frac{ \sqrt{k}\lambda\sqrt{\min(z,n-z)}}{n\|\theta\|_2}\frac{p}{n-p} \lesssim_{\eta} \frac{\sqrt{k}\lambda}{\sqrt{n}\|\theta\|_2},
	\]
	completing the proof.
\end{comment}
% \end{proof}

	% Let $(h_t)_{t\in[n-1]}$ be defined as in the proof of Theorem~\ref{Thm:LocalisationRate}. 
By a symmetric argument, both~\eqref{Eq:HtHttilde} and~\eqref{Eq:Httildeht} hold for $t \in [z, (1-\alpha)n]$ and consequently for $t \in [\alpha n , (1-\alpha)n]$ (with perhaps a slightly different constant).  Consequently, we have that with probability 1 for all but finitely many $p$,
	\[
		\sup_{t\in[\alpha n, (1-\alpha)n]} |H_t - h_t| \leq C_1\bigl( \sqrt{k}\lambda + \sqrt{k\log p}\|\theta\|_2\bigr),
	\]
	where $C_1$ depends only on $\tau$, $\eta$ and $\alpha$.

  If $\theta = 0$, then $h_t = 0$ for all $t$, and hence for sufficiently large $C_{\tau,\eta,\alpha}$, we have with probability $1$ for all but finitely many $p$'s that $\sup_{t\in[\alpha n, (1-\alpha)n]} |H_t| \leq T$ and thus the first conclusion holds.

	For the second conclusion, we have for some $C_2$ depending only on $\tau,\eta$ and $\alpha$ that
	\begin{align*}
		|H_z| \geq |h_z| - \sup_{t\in[\alpha n, (1-\alpha)n]} |H_t - h_t| & \geq C_2\sqrt{n}\|\theta\|_2 - C_1 (\sqrt{k}\lambda + \sqrt{k\log p}\|\theta\|_2) \\
		                                                                  & \geq C_2(1-o(1))\sqrt{n}\|\theta\|_2  - C_1c_{\tau,\eta,\alpha}\sqrt{k}\log p.
	\end{align*}
	The signal size condition on $\|\theta\|_2$ then ensures that $|H_z| \geq (C_2 c'_{\tau,\eta,\alpha}/2 - C_1 c_{\tau,\eta,\alpha}) \sqrt{k}\log p$. Hence, for sufficiently large $c'_{\tau,\eta,\alpha}$, we can ensure that with probability 1 for all but finitely many $p$'s, we have $\max_{t\in[\alpha n, (1-\alpha)n]} |H_t| \geq |H_z|\geq T$, completing the proof.
\end{proof}

\begin{proof}[Proof of Proposition~\ref{Prop:ProjAngle}]
	Define for each $t\in [\alpha n, (1-\alpha)n]$ a vector $\tilde Q_t\in\mathbb{R}^p$ such that $(\tilde Q_t)_{S^{\mathrm{c}}} := 0$ and
	\[
		(\tilde Q_t)_S := \sqrt\frac{n}{t(n-t)}(W_t^\top W_z)_{S,S}\theta_S.
	\]
	By~\eqref{Eq:NoiseCoordinates} and~\eqref{Eq:NoiseContribution} and their symmetric results for $t\in[z, (1-\alpha)n]$, we have with probability 1 for all but finitely many $p$'s that
	\begin{equation}
		\label{Eq:QtQtildet}
		\sup_{\alpha n \leq t\leq (1-\alpha)n} \|Q_t - \tilde Q_t\|_\infty \lesssim_{\tau,\eta,{\alpha}} \max(1,\|\theta\|_2)\log p
	\end{equation}
	Recall the definition of $\gamma = (\gamma_t)_{t\in[n-1]}$ in~\eqref{Eq:gamma_t}. Applying~\eqref{Eq:tmp1} and its symmetric result for $t\geq z$ with $\ell=1$, we have with probability $1$ for all but finitely many $p$'s that
	\begin{equation}
		\label{Eq:Qtildetqt}
		\sup_{t\in [n-1]} \|\tilde Q_t -  \theta\gamma_t \|_\infty \lesssim_{\tau,\eta,{\alpha}} \|\theta\|_2 \sqrt{\log p}.
	\end{equation}
	By symmetry, both~\eqref{Eq:QtQtildet} and~\eqref{Eq:Qtildetqt} are still valid when we replace the supremum over $t\in[z, (1-{\alpha})n]$ instead. Define $\tilde Q := (\tilde Q_{\lceil \alpha n\rceil },\dots, \tilde Q_{\lfloor (1-\alpha)n\rfloor})^\top$ and $\gamma := (\gamma_t)_{t\in [\alpha n, (1-\alpha)n]}$. We then have
	\begin{equation}
		\label{Eq:Qtqt}
		\|Q - \theta\gamma^\top\|_\infty \leq \|Q - \tilde Q\|_\infty+\|\tilde Q - \theta\gamma^\top\|_\infty  \lesssim_{\tau,\eta,{\alpha}}\max(1,\|\theta\|_2)\log p.
	\end{equation}
	By \citet[Propositions~2 and~4 in the online supplement]{wang2018high}, for $c_{\tau,\eta,{\alpha}}$ large enough such that $\lambda \geq \|Q - \theta\gamma^\top\|_\infty$, we have
	\[
		\sin\angle(\hat v, \theta) \lesssim_{\tau,\eta,{\alpha}} \frac{\lambda\sqrt{kn}}{\|\theta\|_2\|\gamma\|_2}, %\asymp_{\tau,\eta,{\alpha}} \frac{\lambda\sqrt{k}}{\sqrt{n}\|\theta\|_2},
	\]
	whence the desired form follows by noting $\| \gamma\|_2 \asymp_{\tau, \eta, {\alpha}} n$ by Condition~\ref{cond:regime}.
\end{proof}

\begin{proof}[Proof of Theorem~\ref{Thm:LocalisationRate2}]
	To simplify exposition, all statements should be interpreted as valid with probability 1 for all but finitely many $p$'s.
  % Note that $t$ is always the temporal index such that $t\in[n-1]$. 
	Write $v:= \theta/\|\theta\|_2$ for simplicity, and note that $\|v\|_0 = \|\theta\|_0\leq k$.  Since the estimator $\hat z$ is unchanged if we replace $\hat v$ by $-\hat v$ in Algorithm~\ref{Algo:Estimation2}, we may assume without loss of generality that $\rho:=\hat v^\top v \geq 0$. 
  Our strategy is to view $(\hat v^\top Q_t : t\in[\alpha n, (1-\alpha)n]) $ as a perturbation of a multiple of $(\gamma_t: t\in[\alpha n, (1-\alpha)n])$, which is maximized at $z$. By \eqref{Eq:Qtqt}, we may choose $c_{\tau,\eta,\alpha}$ large enough such that $\|Q-\theta\gamma^\top\|_\infty \leq \lambda$. By Proposition~\ref{Prop:ProjAngle}, we then have
	\begin{align*}
    \max_{t\in[\alpha n, (1-\alpha)n]} \bigl|\hat v^\top (Q_t - \theta\gamma_t)\bigr| & \leq  \|\hat v\|_1\| Q - \theta\gamma^\top\|_{\max} \leq (\|v\|_1 + \|\hat v - v\|_1)\lambda \\
		                                                                                     & \lesssim_{\tau,\eta,{\alpha}} (\sqrt{k} + \sqrt{p}\|\hat v - v\|_2) \lambda.
	\end{align*}
	From Proposition~\ref{Prop:ProjAngle}, there exists $C_{\tau,\eta,{\alpha}}'>0$, depending only on $\tau,\eta,{\alpha}$, such that
  \begin{equation}
    \label{Eq:Thm11tmp1.5a}
		\|\hat v - v\|_2 \leq 2 \sin\angle (\hat v, \theta) \leq C_{\tau,\eta,{\alpha}}'\frac{\lambda\sqrt{k}}{ \sqrt{n}\|\theta\|_2},
  \end{equation}
	which implies that
	\begin{equation}
		\max_{t\in[\alpha n, (1-\alpha)n]} \bigl|\hat v^\top (Q_t - \theta\gamma_t)\bigr| \lesssim_{\tau,\eta,\alpha} \frac{\lambda^2\sqrt{k}}{\|\theta\|_2}. \label{Eq:Thm11tmp1}
	\end{equation}
	We may further assume that
	\begin{equation}
		\label{Eq:Thm11tmp1.5}
    \frac{\lambda\sqrt{k}}{\sqrt{n}\|\theta\|_2}\leq \frac{\lambda^2\sqrt{k}}{\sqrt{n}\|\theta\|_2^2}\leq \frac{1}{ C'_{\tau,\eta,{\alpha}}}
	\end{equation}
  for all $p$'s, since for $p$ where this is not satisfied the result is trivially true. Then, $\sin\angle(\hat v, \theta)\leq 1/2$ and thus $\rho = \{1-\sin^2\angle(\hat v, \theta)\}^{1/2} \geq 1/2$. Consequently, from~\eqref{Eq:Thm11tmp1} and \eqref{Eq:Thm11tmp1.5}, increasing $C'_{\tau,\eta,\alpha}$ if necessary, we have
	\[
		\hat v^\top \theta \gamma_z = \rho\|\theta\|_2\frac{4z^{1/2}(n-z)^{1/2}(n-p)}{n^{3/2}} \geq 2 \max_{t\in[\alpha n, (1-\alpha)n]} \bigl|\hat v^\top (Q_t - \theta\gamma_t)\bigr|,
	\]
	which implies in particular that $\hat v^\top Q_{\hat z} > 0$.  Now, since $z = \argmax_{t\in[n-1]} \gamma_t$ and $\hat z = \argmax_{t\in[\alpha n, (1-\alpha)n]} \hat v^\top Q_t$, we have from~\eqref{Eq:Thm11tmp1} that
	\begin{equation}
		\label{Eq:Thm11tmp3}
		\hat v^\top \theta\gamma_z - \hat v^\top \theta\gamma_{\hat z} \leq  \hat v^\top Q_z - \hat v^\top Q_{\hat z} + 2 \max_{t\in[\alpha n, (1-\alpha)n]} \bigl|\hat v^\top (Q_t - \theta\gamma_t)\bigr| \lesssim_{\tau,\eta,{\alpha}} \frac{\lambda^2\sqrt{k}}{\|\theta\|_2}.
	\end{equation}
	On the other hand, by~\citet[Lemma~7]{wang2018high}, we have
	\begin{equation}
		\label{Eq:Thm11tmp2}
		\inf_{t\in[z-\min\{z,n-z\}/2,z+\min\{z,n-z\}/2]} \frac{\hat v^\top \theta\gamma_z - \hat v^\top \theta\gamma_{t}}{|z-t|}\gtrsim_{\tau,\eta} \|\theta\|_2\sqrt{n}.
	\end{equation}
	We arrive at the conclusion by combining~\eqref{Eq:Thm11tmp3} and~\eqref{Eq:Thm11tmp2}.
	% \[
	% |\hat z - z| \lesssim_{\tau,\eta,{\alpha}} \frac{\lambda^2\sqrt{k}}{\sqrt{n}\|\theta\|_2^2},
	% \]
	% as desired.
\end{proof}

\begin{proof}[Proof of Theorem~\ref{Thm:LocalisationRate3}]
	As in the proof of Theorem~\ref{Thm:LocalisationRate2}, all statements are valid with probability 1 for all but finitely many $p$'s, and we may assume without loss of generality that $\hat v^\top v \geq 0$.  
  Let $\tilde{Q}_t$ be as in the proof of Proposition~\ref{Prop:ProjAngle}. 
  The main difference to the proof of Theorem~\ref{Thm:LocalisationRate2} will be an improvement of~\eqref{Eq:Thm11tmp1} using the independence between $\hat v$ and $Q_t- \theta\gamma_t$.  Specifically, since
	\begin{align*}
    % \sqrt{\frac{t(n-t)}{n}}
		\{t(n-t)/n\}^{1/2}\hat v^\top(Q_t-\tilde Q_t) & = \hat v^\top W_t^\top (W_z\theta + \xi) - \hat v_S^\top (W_t^\top W_z)_{S,S}\theta_S \\
		                                              & =\hat v_{S^c}^\top (W_t^\top W_z\theta)_{S^c} +\hat v^\top W_t^\top \xi,
	\end{align*}
	we have that
	\begin{align}
		\max_{t\in[\alpha n,(1-\alpha)n]} \bigl|\hat v^\top (Q_t-\theta\gamma_t)\bigr| & \leq \max_{t\in[\alpha n,(1-\alpha)n]} \biggl|\hat v^\top (\tilde Q_t - \theta\gamma_t) + \sqrt\frac{n}{t(n-t)}\hat v^\top(W_t^\top W_z\theta)_{S^c}\biggr| \nonumber \\
		                                                                                  & \qquad + \max_{t\in[\alpha n,(1-\alpha)n]}\biggl| \sqrt\frac{n}{t(n-t)} \hat v^\top W_t^\top \xi\biggr|\label{Eq:Thm12tmp0}
	\end{align}
	We control the two terms on the right-hand side of~\eqref{Eq:Thm12tmp0} separately. 
  By~\eqref{Eq:Qtildetqt}, ~\eqref{Eq:NoiseCoordinates}, the Cauchy--Schwarz inequality and finally~\eqref{Eq:Thm11tmp1.5a}, the first term from the above display~\eqref{Eq:Thm12tmp0} is bounded by 
	\begin{align}
    \|\hat v\|_1 & \biggl(\|\tilde Q-\theta\gamma^\top\|_{\max} + \max_{t\in[\alpha n,(1-\alpha)n]}\biggl\|\sqrt\frac{n}{t(n-t)}\| \hat v \|_2 \|(W_t^\top W_z\theta)_{S^c}\|_2 \biggr\|_\infty\biggr)\nonumber \\
    & \lesssim_{\tau,\eta,{\alpha}}(\sqrt{k}\|v\|_2 + \sqrt{p}\|\hat v - v\|_2) \|\theta\|_2 \log p \lesssim_{\tau,\eta,{\alpha}} \lambda\sqrt{k}\log p.       
    \label{Eq:Thm12tmp1}
	\end{align}
	On the other hand, since $\hat v$, $W_t$ and $\xi$ are mutually independent, we have $\hat v^\top W_t^\top \xi \mid (\hat v, W_t) \sim N(0, \|W_t \hat v^\top \|_2^2)$. By \citet[(22)]{gao2021twosample}, we have
	\[
		\sup_{t\in[\alpha n, (1-\alpha)n]} \biggl\|\frac{1}{n}W_t^\top W_t\biggr\|_{\mathrm{op}} \lesssim_{\tau,\eta,{\alpha}} 1.
	\]
	Hence, by Gaussian tail bounds followed by a union bound, we have
	\begin{equation}
		\label{Eq:Thm12tmp2}
		\max_{t\in[\alpha n,(1-\alpha)n]}\biggl| \sqrt\frac{n}{t(n-t)} \hat v^\top W_t^\top \xi\biggr| \lesssim_{\tau,\eta,{\alpha}} \sqrt{\log p}.
	\end{equation}
	Substituting~\eqref{Eq:Thm12tmp1} and~\eqref{Eq:Thm12tmp2} into~\eqref{Eq:Thm12tmp0}, we have
	\begin{equation}
		\label{Eq:Thm12tmp3}
		\max_{t\in[\alpha n,(1-\alpha)n]} \bigl|\hat v^\top (Q_t-\theta\gamma_t)\bigr|\lesssim_{\tau,\eta,{\alpha}} \lambda\sqrt{k}\log p.
	\end{equation}
	Following the same argument as in the proof of Theorem~\ref{Thm:LocalisationRate2}, with~\eqref{Eq:Thm12tmp3} replacing~\eqref{Eq:Thm11tmp1}, we arrive at the following counterpart to~\eqref{Eq:Thm11tmp3}:
	\begin{equation}
		\label{Eq:Thm12tmp4}
		\hat v^\top \theta\gamma_z - \hat v^\top \theta\gamma_{\hat z}  \leq \hat v^\top Q_z - \hat v^\top Q_{\hat z} + 2 \max_{t\in[\alpha n, (1-\alpha)n]} \bigl|\hat v^\top (Q_t - \theta\gamma_t)\bigr| \lesssim_{\tau,\eta,{\alpha}} \lambda\sqrt{k}\log p.
	\end{equation}
	Combining~\eqref{Eq:Thm12tmp4} with~\eqref{Eq:Thm11tmp2}, the proof is complete.
	% we see that 
	% \[
	% |\hat z - z|\lesssim_{\tau,\eta,{\alpha}} \frac{\lambda\sqrt{k}\log p}{\sqrt{n}\|\theta\|_2},
	% \]
	% as desired.
\end{proof}

\begin{proof}[Proof of Theorem~\ref{Thm:Multiple}]
First observe that since
\begin{equation*}
  \mathbb{P}(\Omega_0^c) \le \sum_{i=1}^\nu \prod_{m=1}^M \Bigl( 1  - \mathbb{P}\bigl( (s_m, e_m] \in \mathcal{I}_i \bigr)\Bigr)  \le \nu \bigl( 1 - \Delta_\tau^2/36\bigr)^M \le \nu \exp(-\Delta_\tau^2 M/36),
\end{equation*}
the second conclusion follows immediately from the first one. To establish the first conclusion, we henceforth work on the event $\Omega_0 \cap \Omega_1 \cap \Omega_2\cap\Omega_3$. 

For $0 \le s < e \le n$, we define the following sets
\begin{align*}
  \mathcal{M}^{(s,e]} & := \{ m\in[M]: s\le s_m < e_m \le e\}, \\
  \mathcal{R}^{(s,e]} & := \{ m\in \mathcal{M}^{(s,e]}: \psi(D_{(s_m + n\varpi,e_m - n\varpi]}) =1  \},\\
  \mathcal{Z}^{(s,e]} & : = \{ i \in [\nu]: z_i \in (s, e]\},\\
  \mathcal{Z}_{\mathrm{good}}^{(s,e]} & : = \{ i \in [\nu]: z_i \in (s, e], \min\{ z_i - s, e-z_i\} \ge n\Delta_\tau /2\}, \\
  \mathcal{Z}_{\mathrm{bad}}^{(s,e]} & : = \{ i \in [\nu]: z_i \in (s, e], \min\{ z_i - s, e-z_i\} < n\phi_i \}.
\end{align*}
Note that on the event $\Omega_0$, we can associate each true changepoint $z_i$ with an $m_i\in [M]$ such that $(s_{m_i}, e_{m_i}] \in \mathcal{I}_i$.
On $\Omega_2$, we have 
\begin{equation}
  \bigl\{ m_i : i\in \mathcal{Z}_{\mathrm{good}}^{(s,e]}\bigr\} \subseteq \mathcal{R}^{(s,e]}.
  \label{Eq:mi-Rse}
\end{equation}
Recall the assumption $\phi < \varpi$. For any $(s_0, e_0] \subset (s, e]$ such that $(s_0, e_0] \cap \{ z_i: i \in [\nu], z_i \in (s,e]\} \subseteq \mathcal{Z}_{\mathrm{bad}}^{(s,e]}$, we have  $(s_0, e_0] \in \mathcal{I}_0$ and hence $\psi(D_{(s_0+n\varpi,e_0-n\varpi]})=0$ on $\Omega_1$. 
% \fn{Hole here.  And $\mathcal{I}_0$ should be redefined to make sense of $(s_0+n\varpi, e_0 - n \varpi]\in \mathcal{I}_0$.}

For any set of changepoints $\hat Z$, we can partition the original timeline $(0,n]$ into $|\hat Z| +1$ segments, which we call segments induced by $\hat Z$. We now prove by induction that as we update $\hat Z$ throughout the recursion of Algorithm~\ref{Algo:Multiple}, for any $(s,e]$ induced by $\hat Z$, we have $\mathcal{Z}^{(s,e]} = \mathcal{Z}^{(s,e]}_{\mathrm{good}} \cup \mathcal{Z}^{(s,e]}_{\mathrm{bad}}$. The base case is trivially true as at the beginning of the algorithm, $\hat Z = \emptyset$, so the only segment induced is $(0,n]$ so $\mathcal{Z}^{(s,e]} = \mathcal{Z}^{(s,e]}_{\mathrm{good}}$ and $ \mathcal{Z}^{(s,e]}_{\mathrm{bad}}=\emptyset$ by our assumption that $z_i-z_{i-1}\geq n\Delta_\tau$ for all $i\in[\nu+1]$. 

Now assuming that the inductive hypothesis holds at some iteration of the recursion in Algorithm~\ref{Algo:Multiple}. We show that the claimed statement still holds if a new changepoint is estimated. Let $\hat Z$ be the set of changepoints identified before this new changepoint, say $\hat z_*$, is added to it. We necessarily have $\hat z_* = \hat z(D_{(s',e']}) + s'$ for some $(s',e'] \in \mathcal{Z}^{(s,e]}$ where $(s,e]$ is induced by $\hat Z$. From the inductive hypothesis, we have $\mathcal{Z}^{(s,e]} = \mathcal{Z}^{(s,e]}_{\mathrm{good}} \cup \mathcal{Z}^{(s,e]}_{\mathrm{bad}}$ and $\mathcal{Z}^{(s,e]}_{\mathrm{good}}$ is necessarily non-empty, for otherwise all changepoints in $(s,e]$ are within a distance of  $n\phi$ to the boundary of the interval, which together with the fact that $\phi < \varpi$, implies that  $\mathcal{M}^{(s,e]}\cap \mathcal{R}^{(s,e]} = \emptyset$, contradicting the fact that a new changepoint is identified.   Thus, there exists some $i' \in \mathcal{Z}^{(s,e]}_{\mathrm{good}}$, which by~\eqref{Eq:mi-Rse} means that $m_{i'}\in \mathcal{R}^{(s,e]}$. By the definition of $m_0$ in Line~\ref{Line7} of Algorithm~\ref{Algo:Multiple}, we have $e_{m_0}-s_{m_0} \leq e_{m_{i'}} - s_{m_{i'}} \leq n\Delta_\tau$. Thus, from the condition of the theorem, we have that $(s_{m_0}, e_{m_0}]$ contains at most one changepoint. If $(s_{m_0}, e_{m_0}]\cap \{z_i:i\in\mathcal{Z}^{(s,e]}\} = \emptyset$, then on $\Omega_1$, $\psi(D_{(s_{m_0}+n\varpi, e_{m_0}-n\varpi]})=0$, contradicting $m_0\in \mathcal{R}^{(s,e]}$. If $(s_{m_0}, e_{m_0}]$ contains a single changepoint $z_i$ for $i\in\mathcal{Z}^{(s,e]}_{\mathrm{bad}}$, then since $\phi < \varpi$, we again have on $\Omega_1$ that $\psi(D_{(s_{m_0}+n\varpi, e_{m_0}-n\varpi]})=0$, a contradiction. By the inductive hypothesis, this implies that $(s_{m_0}, e_{m_0}]$ contains exactly one true change-point $z_{i_0}$ for some $i_0 \in \mathcal{Z}^{(s,e]}_{\mathrm{good}}$ and that $\min\{e_{m_0}-z_{i_0}, z_{i_0} - s_{m_0}\} \geq n\varpi$. Hence, $(s_{m_0}, e_{m_0}] \in \tilde{\mathcal{I}}_{i_0}$, and thus on $\Omega_3$, we have $|\hat z_* - z_{i_0}| \leq n\phi_{i_0}$. 

We finally check that the two new segments induced by $\hat Z \cup \{\hat z_*\}$, say $(\hat z_{\mathrm{left}}, \hat z_*]$ and $(\hat z_*, \hat z_{\mathrm{right}}]$ for $\hat z_{\mathrm{left}} < \hat z_* < \hat z_{\mathrm{right}}$, still satisfy the inductive hypothesis. By symmetry, we may assume without loss of generality that $\hat z_* \leq z_{i_0}$. Since $|z_{i_0} - \hat z_*| \leq n\phi$, we have $i_0\in\mathcal{Z}^{(\hat z_{\mathrm{left}}, \hat z_*]}_{\mathrm{bad}}$. For any $i\in \mathcal{Z}^{(\hat z_{\mathrm{left}}, \hat z_*]}$ such that $i < i_0$, we have $\hat z_* - z_i \geq z_{i_0} - z_i \geq n\Delta_\tau$, and thus $z_i \in \mathcal{Z}^{(\hat z_{\mathrm{left}}, \hat z_*]}_{\mathrm{good}}\cup \mathcal{Z}^{(\hat z_{\mathrm{left}}, \hat z_*]}_{\mathrm{bad}}$ by the inductive hypothesis, consequently, $\mathcal{Z}^{(\hat z_{\mathrm{left}}, \hat z_*]} = \mathcal{Z}^{(\hat z_{\mathrm{left}}, \hat z_*]}_{\mathrm{good}}\cup \mathcal{Z}^{(\hat z_{\mathrm{left}}, \hat z_*]}_{\mathrm{bad}}$. Similarly, for $i\in \mathcal{Z}^{(\hat z_*, \hat z_{\mathrm{right}}]}$, we have $i > i_0$ and $z_i - \hat z_* \geq z_i - z_{i_0} - (\hat z_* - z_{i_0}) \geq n\Delta_\tau - \phi \geq n\Delta_\tau/2$. Again by the inductive hypothesis, $\mathcal{Z}^{(\hat z_*, \hat z_{\mathrm{right}}]} = \mathcal{Z}^{(\hat z_*, \hat z_{\mathrm{right}}]}_{\mathrm{good}}\cup \mathcal{Z}^{(\hat z_*, \hat z_{\mathrm{right}}]}_{\mathrm{bad}}$. This completes the induction. 

As a consequence of the above inductive argument, we have shown that a new changepoint will be identified in Algorithm~\ref{Algo:Multiple} if and only if $(s,e] \cap \{z_i:i\in\mathcal{Z}^{(s,e]}_{\mathrm{good}}\} \neq \emptyset$. Thus, from the inductive claim, at the end of the recursion, each changepoint, say $z_i$, must be less than $n\phi_i$ away from one of the end points of the segments induced by $\hat{Z}$.   This, as well as the assumption that $z_i-z_{i-1}\geq n\Delta_\tau$ for all $i\in[\nu+1]$, means that $|\hat Z| = \nu$ and that $|\hat z_i - z_i|\leq n\phi_i$ as desired. 
\end{proof}
\begin{proof}[Proof of Corollary~\ref{Cor:Multiple}]
  First we write $\phi_i = \frac{C' \lambda \sqrt{k} \log p}{\sqrt{n} \|\theta^{(i)}\|_2}$ and define $\Omega_0$, $\Omega_1$, $\Omega_2$ and $\Omega_3$ as in Theorem~\ref{Thm:Multiple}. Observe that $\Omega_1$, $\Omega_2$ and $\Omega_3$ has implicit dependence on $p$, whereas in the specific coupling~\eqref{Eq:IntervalCoupling} considered in this theorem, $\Omega_0$ does not vary with $p$. We have from the proof of Theorem~\ref{Thm:Multiple} that $\mathbb{P}(\Omega_0) \geq 1-\nu e^{-\Delta_\tau^2M/36}$. Hence, it suffices show that on $\Omega_0$, we have for all but finitely many $p$'s that $\Omega_1$, $\Omega_2$ and $\Omega_3$ hold simultaneously.  We keep in mind that $M$ is fixed and finite, and for the rest of the proof, we condition on a realization of $(\tilde{s}_m, \tilde{e}_m)_{m=1}^M$ as in \eqref{Eq:IntervalCoupling} such that $\Omega_0$ holds. 

Let $\mathcal{I}_0$, $\mathcal{I}_i$ and $\tilde{\mathcal{I}}_i$ be defined as in Theorem~\ref{Thm:Multiple}.  We first establish $\Omega_1$ and $\Omega_2$.  For any interval $(s,e]$ with $e-s \le p$, $\psi=0$ by definition. For every $(s,e] \in  \cup_{0 \le i\le \nu} \, \mathcal{I}_i$ whose length is longer than $p$, the fixed-ratio regime Condition~\ref{cond:regime} is true by the generating mechanism of the intervals in \eqref{Eq:IntervalCoupling}, and it is straightforward to verify that Theorem~\ref{thm:test} applies.  As a result, there exist $c, c', C$, which may depend on $(s, e]$, such that the conclusion of Theorem~\ref{thm:test} holds for each $(s,e]\in \cup_{0 \le i \le \nu} \,\mathcal{I}_i$ with $e - s >p$.  Inspecting the proof of Theorem~\ref{thm:test} shows that we can take the maximum of all such $c, c', C$'s so that the conclusion of Theorem~\ref{thm:test} holds for all intervals in $\cup_{0\le i \le \nu}\, \mathcal{I}_i$ with length longer than $p$.  As such, we have, for all but finitely many $p$'s, $\Omega_1$ and $\Omega_2$ holds.

Now we turn to $\Omega_3$.  Again, by reasoning similar to the above, we see the conditions of Theorem~\ref{Thm:LocalisationRate3} hold for each $(s,e] \in \tilde{\mathcal{I}}_i$, and for the above-mentioned specific choices of $c$, which may depend on $(s, e]$, the conclusion of Theorem~\ref{Thm:LocalisationRate3} holds for each $(s,e]\in \tilde{\mathcal{I}}_i$.  We can again take the maximum of all such $c$'s so that for all intervals in $\tilde{\mathcal{I}}_i$ for all $i\in [\nu]$, the conclusion of Theorem~\ref{Thm:LocalisationRate3} holds, i.e.,
\begin{equation*}
  \frac{\hat{z}(D_{(s,e]}) - (z_i -s) }{n} \le C'_{(s,e]} \frac{\lambda \sqrt{k} \log p}{\sqrt{n} \|\theta\|_2}.
\end{equation*}
Setting, e.g., $C' = \max_{i\in[\nu]} \max_{(s,e]\in \tilde{\mathcal{I}}_i} C'_{(s,e]}$, we have, for all but finitely many $p$'s, $\Omega_3$ holds.  Invoking Theorem~\ref{Thm:Multiple} completes the proof.
\end{proof}

\section{Ancillary results}
\label{Sec:Ancillary}
We collect here the ancillary results and their proofs.
% \teal{\sout{Recall the definitions of the operator norm and the $k$-operator norm of matrices.}}
% \fn{In the supplement: We adopt the same notations as in the main text.}
\begin{lemma}
	\label{Lem:kOpNorm}
	Fix $A\in\mathbb{R}^{p\times p}$ and $k\in[p]$. The following are true.
	\begin{enumerate}
		\item If $A$ is symmetric, then $\sup_{u,v\in B_0(k)} u^\top A v \leq \sup_{v \in B_0(2k)} v^\top A v$.
		\item $\sup_{v\in B_0(k)} \|Av\|_2 \leq \sqrt{p/k} \sup_{u,w \in B_0(k)} u^\top A w$.
	\end{enumerate}
\end{lemma}
\begin{proof}
	For the first claim, fix $u,v\in B_0(k)$ and let $S$ and $T$ be their respective support. Then by the symmetry of $A$, we have
	\[
		u^\top A v \leq \|A_{S,T}\|_{\mathrm{op}} \leq \|A_{S\cup T, S\cup T}\|_{\mathrm{op}} = \sup_{w\in\mathcal{S}^{2k-1}} w^\top A_{S\cup T, S\cup T} w \leq \sup_{v \in B_0(2k)} v^\top A v.
	\]
	The first claim then follows by taking supremum on the left-hand side.

	For the second claim, define $\psi:= \sup_{u,w\in B_0(k)} u^\top A w$. Write $\binom{[p]}{k} := \{S\subseteq B: |S| = k\}$. For any $v\in B_0(k)$, let $\tilde{v} := Av / \|Av\|_2$ and $T := \mathrm{supp} (v)$. Then by the Cauchy--Schwarz inequality, we have
	\begin{align*}
		\|Av\|_2 & = \tilde{v}^\top A v = \frac{1}{\binom{p-1}{k-1}} \sum_{S \in \binom{[p]}{k}} \tilde{v}_S^\top A_{S,T} v_T \leq \frac{1}{\binom{p-1}{k-1}} \sum_{S \in \binom{[p]}{k}}  \|\tilde{v}_S\|_2 \psi                                                                             \\
		         & \leq \biggl\{\frac{1}{\binom{p-1}{k-1}}\sum_{S \in \binom{[p]}{k}} \|\tilde{v}_S\|_2^2 \biggr\}^{1/2}\biggl\{\frac{1}{\binom{p-1}{k-1}}\sum_{S \in \binom{[p]}{k}} \psi^2 \biggr\}^{1/2} \leq \|\tilde{v}\|_2\cdot  \sqrt\frac{p}{k} \psi = \sqrt\frac{p}{k} \psi.
	\end{align*}
	Taking supremum over $v$ on the left-hand side, we arrive at the conclusion.
\end{proof}

Suppose that $X = (x_1, \dots, x_n)^\top$ is generated by independent $x_i \sim N_p(0, \Sigma)$ with $n \ge p$ and some positive definite $\Sigma$.  
Write $S = X^\top X$ and $S_1 = X_{(0,t]}^\top X_{(0,t]}$. 
By \cite{mitra1970density}, for any well-defined function $\phi: S \mapsto \phi(S)$ such that $\phi(S) S \phi(S)^\top = I_p$, $U = \phi(S) S_1 \phi(S)^\top$ is said to have a matrix-variate Beta distribution, i.e. $ U \sim \mathrm{Beta}_p(t/2,(n-t)/2)$.  
To the best of our knowledge, it is unclear what happens when $n < p$ in the literature. 
The following Lemma~\ref{lemma:generalised-beta} and Corollary~\ref{cor:generalise-betap} effectively generalize the existing matrix-variate Beta distribution to the rank-deficient case of $n <p$. 
% It is made possible by the insight that in the conventional case of $n > p$, if we conduct a QR decomposition on $X$ such that $X = Q R$ with $Q\in \mathbb{O}^{n\times r}$ and $R \in \mathbb{R}^{r \times p}$ such that $R_{i,i} \ge 0$ for all $i\in [r]$ and $R_{i,j} = 0$ for all $1\le j < i\le r$, then $B := Q_{(0,t]}^\top Q_{(0,t]} \sim \mathrm{Beta}_p(t/2, (n-t)/2)$.

\begin{lemma}
	Suppose $X\in \mathbb{R}^{ n \times p}$ has independent $N(0,1)$ entries and write $r := \min\{n,p\}$.
	There exists an almost surely unique way of writing $X = QR$ as its QR decomposition, where $Q \in \mathbb{O}^{n\times r}$ and $R \in \mathbb{R}^{r\times p}$ such that $R_{i,j} =0$ for all $i > j$ and $R_{i,i} \ge 0$ for all $i \in [n]$.
	We have that $Q$ and $R$ are independent and $Q \sim \mathrm{Unif}(\mathbb{O}^{n \times n})$.  Furthermore, $R_{i,i}^2 \sim \chi^2(n-i+1)$ and $R_{i,j} \sim N(0,1)$ for $i\in [n]$ and $j \in [p]$ with $i < j$.
	Then $B := Q_{(0,t]}^\top Q_{(0,t]}\sim \Beta_r(t/2, (n-t)/2)$ and is independent of $X^\top X$.
	\label{lemma:generalised-beta}
\end{lemma}
\begin{proof}
	First we consider the case of $n \ge p$. %, the conclusion holds by \cite{mitra1970density}. 
	Write the (almost surely) unique QR decomposition of $X$ by $X = Q R$ with $Q \in \mathbb{O}^{n\times p}$ and $R \in \mathbb{R}^{p\times p}$ being an upper triangular matrix with $R_{i,i} \ge 0 $ for all $i\in [p]$.
	For any fixed $H \in \mathbb{O}^{n\times n}$, $H X \stackrel{\mathrm{d}}{=} X$, whence $H Q R \stackrel{\mathrm{d}}{=} QR$.
	As such, the joint density of $Q$ and $R$ is constant for every possible value of $Q \in \mathbb{O}^{p\times n}$, whence $Q$ and $R$ are independent and $Q \sim \mathrm{Unif}(\mathbb{O}^{p \times n})$.
	By \cite[Theorem 3.2.14]{Muirhead1982}, we have $R_{i,i}^2 \sim \chi^2(n-i+1)$ and $R_{i,j}\sim N(0,1)$ and $R_{i,j}$ are independent for all $i\le j$.
	We define $S := X^\top X$ and $S_1 = X_{(0,t]}^\top X_{(0,t]}$ and $S_2 = X_{(t,n]}^\top X_{(t,n]}$. Define $S^{1/2}:= R^\top$, and by \cite{mitra1970density},
	\[
		B := Q_{(0,t]}^\top Q_{(0,t]} = S^{-1/2} S_1 (S^{-1/2})^\top \sim \Beta_p(t/2, (n-t)/2).
	\]
	We note that $B$ as a function of $Q$ is independent of $X^\top X = R^\top R$, by the independence of $Q$ and $R$.

	Now we consider the case $n < p$.
	Write $X_j$ as the $j$th column of $X$.
	Write $X = [X_{\mathrm{L}} \mid X_{\mathrm{R}}]$ where $X_{\mathrm{L}} = [ X_1 \mid \dots \mid X_n ]$ and $ X_{\mathrm{R}} := [ X_{n+1} \mid \dots \mid X_p]$. For $X_{\mathrm{L}}$ whose rank is almost surely $n$, there exists a unique QR decomposition such that $X_{\mathrm{L}} = Q R_{\mathrm{L}}$.
	Take $R_{\mathrm{R}} := Q^\top X_{\mathrm{R}}$ and $R := [ R_{\mathrm{L}} \mid R_{\mathrm{R}}]$, and we have $X = Q R$, where both $Q $ and $R$ are almost surely unique.
	By the same argument as the case of $n \ge p$, we have $Q$ and $R$ are independent and $Q \sim \mathrm{Unif}(\mathbb{O}^{n \times n})$.
	Applying the conclusion from the case of $n \ge p$ on $X_{\mathrm{L}} = Q R_{\mathrm{L}}$, we have $(R_L)_{i,i}\sim \chi^2(n-i+1)$ and $(R_L)_{i,j} \sim N(0,1)$.
	Furthermore, since $R_{\mathrm{R}} = Q^\top X_{\mathrm{R}}$ where both $Q$ and $X_{\mathrm{R}}$ is independent of $X_{\mathrm{L}}$ and $Q$ is independent of $X_{\mathrm{R}}$, all entries of $R_{\mathrm{R}}$ are standard normals independent of $R_{\mathrm{L}}$.

	Applying the case of $n \ge p$ on $X_{\mathrm{L}} = Q R_{\mathrm{L}}$, we have
	\(
	B := Q_{(0,t]}^\top Q_{(0,t]} \sim \Beta_n(t/2, (n-t)/2).
	\)
	By the same argument as before, $B$ is independent of $X^\top X$.
\end{proof}

% Note that for matrix $R \in \mathbb{R}^{n\times p}$ with structure as in Lemma~\ref{lemma:generalised-beta} ($n \le p$), we can find a matrix $R^{-1}$ such that $R R^{-1} = I_n$ by defining $R^{-1} = H T^{-\top} $, where $H$ and $T$ are from the QR decomposition for $R^\top = H T$ with $H \in \mathbb{O}^{p \times n}$ and $T$ being upper triangular.

\begin{cor}
	Let $X = (x_1, \dots, x_n)^\top$ where $x_i \sim N_p(0, \Sigma)$ where $\Sigma$ is a positive definite matrix.
	Write $S = X^\top X$ and $S_1 := X_{(0,t]}^\top X_{(0,t]}$.
	Let $r := \min\{ n,p\}$ and $\phi: \mathbb{R}^{p\times p} \to \mathbb{R}^{p\times r}$ be a function such that $\phi(S) S \phi(S)^\top = I_r$ for any positive semi-definite matrix $S\in \mathbb{R}^{p\times p}$ of rank $r$.
	Then $\phi(S) S_1 \phi(S)^\top \sim \mathrm{Beta}_r(t/2, (n-t)/2)$ and is independent of $S = X^\top X$.
	\label{cor:generalise-betap}
\end{cor}
\begin{proof}
	By the positive definiteness of $\Sigma$, we find $\Sigma^{1/2} \in \mathbb{R}^{p\times p}$ such that $\Sigma = \Sigma^{1/2} (\Sigma^{1/2})^\top$.
	Find $\Sigma^{-1/2}$ such that $\Sigma^{-1/2} \Sigma^{1/2} = I_p$.
	Define $y_i := \Sigma^{-1/2} x_i \sim N_p(0, I_p)$ and $Y = X (\Sigma^{-1/2})^\top$ has independent $N(0,1)$ entries.
	Let $\tilde{S} = Y^\top Y = \Sigma^{-1/2} S (\Sigma^{-1/2})^\top$ and
	$\tilde{S}_1 = Y_{(0,t]}^\top Y_{(0,t]} = \Sigma^{-1/2} S_1 (\Sigma^{-1/2})^\top$, whence by the definition of $\phi$ we have $\phi(S) S \phi(S)^\top = \phi(S) \Sigma^{1/2} \tilde{S} (\Sigma^{1/2})^\top \phi(S)^\top = I_r$.
	As such, we define $\tilde{\phi}(\tilde{S}) := \phi(\Sigma^{1/2} \tilde{S} (\Sigma^{-1/2})^\top) \Sigma^{1/2}$, which is a well-defined function, and have $\tilde{\phi}(\tilde{S}) \tilde{S} \tilde{\phi}({\tilde{S}})^\top = I_r$.
	Since $\tilde\phi(\tilde{S}) \tilde{S}_1 \tilde\phi(\tilde{S})^\top$ can be defined by $Y$ with no dependence on $\Sigma$, it suffices to work on the case $\Sigma = I_p$, which we assume for the rest of the proof.

	By Lemma~\ref{lemma:generalised-beta}, write the unique QR decomposition of $X$ by $X = QR$ with $Q\in \mathbb{O}^{n\times r}$ and $R \in \mathbb{R}^{r\times p}$, whence $S = R^\top R$. Write $B = Q_{(0,t]}^\top Q_{(0,t]}$.
	Since \( \phi(S)^\top R^\top R \phi(S) = I_r  \), $H : = R \phi(S)\in \mathbb{O}^{r\times r}$, whence
	\( \phi(S)^\top S_1 \phi(S) = H^\top Q_{(0,t]}^\top Q_{(0,t]} H = H^\top B H\).
	Since $B$ as a function of $Q$ is independent of $S$ and $R$, it is independent of $H$.  By noting $H^\top B H = \phi(S) S_1 \phi(S)^\top$ and $B \stackrel{\mathrm{d}}{=} H^\top B H$ because $QH \stackrel{\mathrm{d}}{=} Q$ for any $H \in \mathbb{O}^{r\times r}$.
	The independence between $\phi(S) S_1 \phi(S)$  and $S$ follows from the fact that the distribution of $\phi(S) S_1 \phi(S)$ is invariant conditionally on $S$. 
\end{proof}

% In case of $n \ge p$, $r = p$ 
% Suppose that $S_1 \sim W_p(n_1, I_p)$ and $S_2 \sim W_p(n_2, I_p)$ follow independent Wishart distributions. If $n_1 +n_2 >p$, then $S_1 + S_2$ is almost surely invertible and we find a symmetric square root of $S_1 + S_2$, denoted by $(S_1 + S_2)^{-1/2}$.
% We say $B:=  (S_1+S_2)^{-1/2} S_1 (S_1+S_2)^{-1/2} $ follows a matrix-variate Beta distribution with parameters $n_1/2$ and $n_2/2$, denoted by $B\sim \mathrm{Beta}_p(n_1/2,n_2/2)$.

% Recall also that the spectral distribution function of any $p\times p$ matrix $A$ is defined as $F^A(t) := n^{-1} \sum_{i=1}^p \mathbbb{1}_{\{ \lambda_i^A \le t\}}$, where $\lambda_i^A$s are eigenvalues (counting multiplicities) of the matrix $A$. Further, given a sequence $(A_n)_{n\in\mathbb{N}}$ of matrices, their limiting spectral distribution function $F$ is defined as the weak limit of the $F^{A_n}$, if it exists.

Recall that $J_{(a_1,a_2]} := (J_{a_1+1}, \cdots, J_{a_2})^\top $ is the submatrix of $J$ by taking only the $(a_1+1)$-th to $a_2$-th rows for any matrix $J$. 
For the rest of the paper, we define shorthand
\begin{equation*}
	% \label{Eq:Sab}
	S_{a_1,a_2} := \sum_{i=a_1+1}^{a_2} x_ix_i^\top = X^\top_{(a_1, a_2]} X_{(a_1, a_2]}
\end{equation*}
Define the scalar quantity $\eta(n,p) := (\mathbb{E}[ x_1 x_1^\top (n^{-1}S_{0,n})^{-1} x_n x_n^\top ])_{1,1}$.

\begin{lemma}
	\label{Lemma:WtWz}
	For all $ t \le z$,
	\( W_t^\top W_z = 4 S_{0,t}S_{0,n}^{-1} S_{z,n}\), whence for $z\in [n]$ and $t\in [z]$ we have
	\begin{gather*}
		\mathbb{E}[W_t^\top W_z] = 4t(n-z) n\eta(n,p) I_p. % \frac{t}{z} \mathbb{E} [W_z^\top W_z] =     \frac{n-z}{n-t} \mathbb{E} W_t^\top W_z.
	\end{gather*}
	Furthermore, under Condition~\ref{cond:regime},
	\( \frac{\eta(n,p)}{ (n-p)n^{-1}} \to 1 \), i.e., $\eta(n,p) \to \eta$ as $n,p \to \infty$.
\end{lemma}
\begin{proof}
	By the construction of $A$, we have $A A^\top = I_n - X (X^\top X)^{-1} X^\top$. We have
	\begin{equation*}
		\begin{aligned}
			W_t^\top W_z & =
			\begin{pmatrix} X_{(0, t]}^\top &  -X_{(t,n]}^\top\end{pmatrix}
			\begin{pmatrix}{A_{(0,t]}}\\{A_{(t,n]}}  \end{pmatrix}
			\begin{pmatrix} A_{(0, z]}^\top & A_{(z,n]}^\top \end{pmatrix}
			\begin{pmatrix}{X_{(0, z]}}\\{-X_{(z,n]}} \end{pmatrix}                                                                                                                    \\
			             & = \begin{pmatrix} X_{(0, t]}^\top &  -X_{(t,n]}^\top\end{pmatrix}  (I_n - X (X^\top X)^{-1} X^\top) \begin{pmatrix}{X_{(0, z]}}\\{-X_{(z,n]}} \end{pmatrix} \\
			             & = (S_{0,t} - S_{t,z} + S_{z,n}) - (S_{0,t} - S_{t,z} - S_{z,n}) S_{0,n}^{-1} (S_{0,t} + S_{t,z} - S_{z,n})                                                                                             \\
			             & = 2 (S_{0,t} - S_{t,z} - S_{z,n}) S_{0,n}^{-1} S_{z,n} + 2 S_{z,n}                                                                                                                                     \\
			             & = 4 S_{0,t} S_{0,n}^{-1} S_{z,n}.
		\end{aligned}
		% \label{Eq:WtWz}
	\end{equation*}
	In particular, we have for all $z\in[n]$ and $t\in [z]$
	\begin{align*}
		\mathbb{E}W_t^\top W_z & = 4 \sum_{i = 1}^t \sum_{j=z+1}^n \mathbb{E} [x_i x_i^\top  S_{0,n}^{-1} x_j x_j^\top]  = 4 t (n-z) \mathbb{E}[x_1 x_1^\top S_{0,n}^{-1} x_n x_n^\top], %=
		% (n-z) \mathbb{E} S_{0,t} S_{0,n}^{-1} x_z x_z^\top\\
		% & = 4\frac{n-z}{n-t} \mathbb{E} \biggl[ S_{0,t} S_{0,n}^{-1} \sum_{i=t+1}^n x_i x_i^\top\biggr] = \frac{n-z}{n-t} \mathbb{E} W_t^\top W_z,
	\end{align*}
	where we invoke the exchangeability of $x_i x_i^\top S_{0,n}^{-1} x_j x_j^\top$ for all $ 1\le i  < j \le n$ in the second equality.

	We first note $\mathbb{E} [W_z^\top W_z] =4 z(n-z) \mathbb{E}[x_1 x_1^\top S_{0,n}^{-1} x_n x_n^\top]$, and then show that $\mathbb{E}[W_z^\top W_z]$ must be a scale of $I_p$.
	Since for any $U\in \mathbb{O}^{p\times p}$, $x_i^\top U \stackrel{\mathrm{d}}{=} x_i^\top $ for $i \in [n]$, whence we have
	\(
	U^\top W_z^\top W_z U = 4 U^\top S_{0,t} U (U^\top S_{0,n} U)^{-1} U^\top S_{z,n}U \stackrel{\mathrm{d}}{=} W_z^\top W_z
	\).
	In particular, $W_z^\top W_z$ have identically distributed diagonals and identically distributed off-diagonals.  It suffices to verify that its off-diagonals have zero mean.

	Let $X = Q T$ be the almost surely unique QR decomposition of $X$, where we only take non-negative diagonal entries in $T$.
	By Equation (15) of \cite{gao2021twosample}, we have $W_z^\top W_z = 4 T^\top V \Lambda (I_p - \Lambda) V^\top T$, where $V \Lambda V^\top = Q_{(0,z]}^\top Q_{(0,z]}$ is the eigendecomposition of $B := Q_{(0,t]}^\top Q_{(0,t]}$.
	Note that $V \sim \mathrm{Unif}(\mathbb{O}^{p\times p})$, $\Lambda$ and $T$ are mutually independent and $T$ has independent entries with $T_{j,j} = t_j >0$ such that $t_j^2 \sim \chi^2_{n-j+1}$ and $T_{j,k} = z_{jk} \sim N(0,1)$ for $j \ne k$.
	For off-diagonals, it suffices to have
	\begin{equation*}
		\begin{aligned}
			(\mathbb{E}[W_z^\top W_z])_{1,2} & = 4\mathbb{E} \biggl[ \sum_{j=1}^p t_1 V_{j,1} \lambda_j (1-\lambda_j) ( t_2 V_{j,2} + z_{12} V_{j,1}) \biggr]                                                                     \\
			                                 & = 4 \sum_{j=1}^p \biggl[ \mathbb{E}[t_1]  \mathbb{E}[\lambda_j (1-\lambda_j)] (\mathbb{E}[t_2]\mathbb{E}[V_{j,1} V_{j,2}] + \mathbb{E}[z_{12}] \mathbb{E}[V_{j,1}^2] ) \biggr] =0,
		\end{aligned}
	\end{equation*}
	where $\lambda_1 \ge \lambda_2\ge \dots \ge \lambda_p$ are diagonal elements of $\Lambda$ and $\mathbb{E}[V_{j,1} V_{j,2}] = (1/p)\sum_{j=1}^p V_{j,1}V_{j,2} = 0$ since $V \in \mathbb{O}^{p \times p}$.

	Lastly, under Condition~\ref{cond:regime}, by Proposition 8 of \cite{gao2021twosample}, $(W_z^\top W_z)_{1,1} \xrightarrow{\mathrm{a.s.}} (4z(n-z)(n-p))/n^3 $.  Noting $(\mathbb{E} [ W_z^\top W_z])_{1,1} = 4 z(n-z) n\eta(n,p)$, we conclude the convergence of $\eta(n,p) \to \eta$.
\end{proof}

\begin{lemma} 
  \label{lem:S0t-S0z}
	Let $x_1,\ldots,x_n\stackrel{\mathrm{iid}}{\sim} N_p(0, I_p)$. Fix $z \in [n]$ and $t\in[z]$.
	For any nonrandom $u, w\in\mathcal{S}^{p-1}$, we have
	\[
		\mathbb{P}\biggl[  \frac{1}{t} \biggl| u^\top \biggl\{S_{0,t} - \frac{t}{z} S_{0,z} \biggr\} S_{0,n}^{-1} S_{z,n}  w \biggr| \geq x \biggm | S_{0,z}, S_{z,n} \biggr] \leq 52\exp\biggl\{-\frac{t^2x^2}{8z \|S_{0,n}^{-1}S_{z,n}\|_{\mathrm{op}}^2 \|S_{0,z}/z\|_{\mathrm{op}}^2}\biggr\}.
	\]
\end{lemma}
\begin{proof}
	For notational simplicity, we use $\mathbb{P}^z$ and $\mathbb{E}^z$ to denote the conditional probability and expectation with respect to the $\sigma$-algebra generated by $( S_{0,z}, S_{z,n} )$.
	Note that $S_{0,t} \mid (S_{0,z}, S_{z,n}) \stackrel{\mathrm{d}}{=} S_{0,t} \mid S_{0,z}$ and $\mathbb{E}^z[S_{0,t}] = (t/z) S_{0,z}$.

	Let $r:=\min\{z,p\}$. We note $S_{0,z}$ has (almost surely) rank $r$, and write $S_{0,z} = R^\top R$ for the (almost surely unique) Cholesky decomposition of $S_{0,z}$ such that $R\in\mathbb{R}^{r\times p}$ is an upper-triangular matrix with positive diagonal entries. Write $R^\dagger\in\mathbb{R}^{p\times r}$ for the (almost surely unique) Moore--Penrose pseudo-inverse of $R$ such that $R R^\dagger = I_p$.  By Corollary~\ref{cor:generalise-betap}, the matrix $B:= (R^\dagger)^\top S_{0,t} R^\dagger \sim \mathrm{Beta}_r(t/2, (z-t)/2)$ has a matrix-variate Beta distribution, and is independent of $S_{0,z}$ with $(t/z)I_p$ as its (conditional) mean. Observe that $R^\top B R = (R^\dagger R)^\top S_{0,t} (R^\dagger R) = S_{0,t}$, since $R^\dagger R$ is a (symmetric) orthogonal projection matrix onto the row space of $X_{(0,z]}$, which contains the row space of $X_{(0,t]}$.

	Define $v := S_{0,n}^{-1} S_{z,n} w / \| S_{0,n}^{-1} S_{z,n} w \|_2$.
	Writing $\tilde u:= R u / \sqrt{z}$ and $\tilde v:=R v / \sqrt{z}$, we have
	\begin{align*}
		\mathbb{P}^z\biggl[\frac{1}{t}| u^\top \{S_{0,t} - \mathbb{E}^z(S_{0,t})\} & S_{0,n}^{-1}  S_{z,n} w | \geq x\biggr]  \le  \mathbb{P}^z\biggl[\frac{1}{t} | u^\top \{S_{0,t} - \mathbb{E}^z(S_{0,t})\} v| \geq \frac{x}{\| S_{0,n}^{-1} S_{0,z}\|_{\mathrm{op}}}\biggr]                                                                \\
		                                                                           & = \mathbb{P}^z\biggl\{\biggl| \tilde u^\top \biggl(\frac{z}{t}B - I_p\biggr) \tilde v\biggr| \geq \frac{x}{\| S_{0,n}^{-1} S_{0,z}\|_{\mathrm{op}}} \biggr\}                                                                                              \\
		                                                                           & \leq \mathbb{P}^z\biggl\{\biggl| \biggl(\frac{\tilde u}{\|\tilde u\|_2}\biggr)^\top \biggl(\frac{z}{t}B - I_p\biggr) \frac{\tilde v}{\|\tilde v\|_2} \biggr| \geq \frac{x}{\| S_{0,n}^{-1} S_{0,z}\|_{\mathrm{op}} \|S_{0,z}/z\|_{\mathrm{op}}} \biggr\}.
	\end{align*}
	Write shorthand $\psi := \| S_{0,n}^{-1} S_{0,z}\|_{\mathrm{op}} \|S_{0,z}/z\|_{\mathrm{op}}$, which is measurable with respect to the $\sigma$-algebra generated by $(S_{0,z}, S_{z,n})$.
	There exists an orthogonal matrix $U$ such that $U\tilde u/\|\tilde u\|_2 = e_1$ and $U\tilde v/\|\tilde v\|_2 = \alpha e_1 + \beta e_2$ for real $\alpha$ and $\beta$ such that $\alpha^2+\beta^2=1$, where $e_j$ denotes the $j$th standard basis vector in $\mathbb{R}^p$. Using the fact that $B \mid (S_{0,z}, S_{z,n}) \stackrel{\mathrm{d}}{=} U^\top BU \mid (S_{0,z}, S_{z,n})$, we have for $J=\{1,2\}$ that
	\begin{align*}
		\mathbb{P}^z\biggl\{\biggl| \biggl(\frac{\tilde u}{\|\tilde u\|_2}\biggr)^\top \biggl(\frac{z}{t}B & - I_p\biggr) \frac{\tilde v}{\|\tilde v\|_2} \biggr| \geq \frac{x}{\psi} \biggr\}                                                                                                      \\
		                                                                                                   & \leq \mathbb{P}^z\biggl\{ \biggl| \frac{z}{t}B_{1,1} - 1\biggr| \geq \frac{x}{\sqrt{2} \psi }\biggr\} + \mathbb{P}^z\biggl\{\frac{z}{t}|B_{1,2}| \geq \frac{x}{\sqrt{2} \psi }\biggr\} \\
		                                                                                                   & \leq 2\mathbb{P}^z\biggl\{\biggl\|\biggl(B - \frac{t}{z}I_p\biggr)_{J,J}\biggr\|_{\mathrm{op}} \geq \frac{tx}{\sqrt{2} z \psi}\biggr\},
	\end{align*}
	where the first inequality holds by noting $|\alpha| + |\beta| \le (2\alpha^2 + 2\beta^2)^{1/2}= \sqrt{2}$.
	Note that $\{w\in\mathcal{S}^{p-1}:\mathrm{supp}(w)\subseteq J\}$ is isomorphic to $S^1$, which contains a $(1/4)$-net $\mathcal{N}$ of cardinality $\lceil \frac{2\pi}{4\arcsin(1/8)}\rceil= 13$. By \citet[Theorem~5.3.12]{gupta1999matrix}, for each $w\in\mathcal{N}$, we have $w^\top B w \sim \mathrm{Beta}(t/2, (z-t)/2)$. Hence, by \citet[Lemma~5.4]{Vershynin2012} and a union bound, we have
	\begin{align*}
		\mathbb{P}^z\biggl\{\biggl\|\biggl(B - \frac{t}{z}I_p\biggr)_{J,J}\biggr\|_{\mathrm{op}} \geq \frac{tx}{\sqrt{2} z \psi} \biggr\} & \leq \mathbb{P}^z\biggl\{\sup_{w\in\mathcal{N}} \bigl| w^\top B w - t/z\bigr| \geq \frac{tx}{2\sqrt{2} z \psi}\biggr\} \\
		                                                                                                                                  & \leq 13 \mathbb{P}^z \biggl\{ \bigl|B_{1,1} - t/z\bigr| \geq \frac{tx}{2\sqrt{2} z \psi}\biggr\}                       \\
		                                                                                                                                  & \leq 26\exp\biggl\{-\frac{t^2x^2}{8z\psi^2}\biggr\},
	\end{align*}
	where we have used \citet[Theorem~2.1]{marchal2017sub} in the final inequality.
\end{proof}

\begin{lemma}
	\label{Lem:Sphere}
	Let $X= (X_1,\ldots,X_p)^\top$ be uniformly distributed on the sphere $\mathcal{S}^{p-1}$. Then, for $\delta \geq e^{-p/16}$, we have

	\[
		\mathbb{P}\biggl( \|X\|_{\infty} > \sqrt{\frac{4 \log(2/\delta)}{p}} \biggr) \le p \delta.
		\]
\end{lemma}
\begin{proof}
	Let $Z_1,\ldots,Z_p$ be independent $N(0,1)$ random variables, then $X_1 \stackrel{\text{d}}{=} Z_1 / (Z_1^2+\cdots + Z_p^2)^{1/2}$. By a standard Gaussian tail bound, we have
	\[
		\mathbb{P}\bigl\{Z_1 > \sqrt{2\log(1/\delta)}\bigr\} \leq \delta.
	\]
	Moreover, $\sum_{j=1}^p Z_j^2 \sim \chi^2_p$. Since we have $\delta \geq e^{-p/16}$, by \citet[Lemma~1]{LaurentMassart2000},
	\[
		\mathbb{P}\biggl(\sum_{j=1}^p Z_j^2 < \frac{p}{2}\biggr)\leq \mathbb{P}\biggl(\sum_{j=1}^p Z_j^2 < p - 2\sqrt{p\log(1/\delta)}\biggr) \leq \delta.
	\]
	The result follows by combining the above two bounds and applying a union bound.
\end{proof}

Recall $B_0(k)\subseteq \mathbb{R}^p$ is the set of $k$-sparse unit vectors.
\begin{lemma}
	For any $A\in\mathbb{R}^{p\times p}$, and any $\epsilon\in(0,1)$, there exists an $\epsilon$-net $\mathcal{N}_\epsilon$ of $B_0(k)$ of cardinality at most $\{(1+2/\epsilon)ep/k\}^k$ such that
	\label{Lem:Net}
	\[
		\sup_{u\in B_0(k)} u^\top A v \leq (1-\epsilon)^{-1} \max_{u\in \mathcal{N}_\epsilon} u^\top A v.
	\]
\end{lemma}
\begin{proof}
	By \citet[Lemma~5.2]{Vershynin2012}, for each subset $S\subseteq [p]$ of cardinality $k$, there exists an $\epsilon$-net $\mathcal{N}_S$ of $\{v\in B_0(k): \mathrm{supp}(v)\subseteq S\}$ of cardinality at most $(1+2/\epsilon)^{k}$.  Define $\mathcal{N}_\epsilon := \cup_{S\subseteq [p]: |S| = k} \mathcal{N}_S$, then $|\mathcal{N}_\epsilon| \leq \binom{p}{k}(1+2\epsilon)^k \leq \{(1+2\epsilon)ep/k\}^k$.  For any fixed $x \in B_0(k)$, find $\tilde x \in \mathcal{N}_\epsilon$ such that $\|x-\tilde x\|_2 \leq \epsilon$ and $\|x-\tilde x\|_0\leq k$. Thus,
	\begin{align*}
		x^\top A v & = (x - \tilde{x})^\top A v +  \tilde{x}^\top A v                                                      \\
		           & \leq \|x - \tilde{x}\|_2 \sup_{u \in B_0(k)} u^\top A v + \sup_{u\in \mathcal{N}_\epsilon} u^\top A v \\
		           & \leq \epsilon \sup_{u \in B_0(k)} u^\top A v +  \sup_{u\in \mathcal{N}_\epsilon} u^\top Av.
	\end{align*}
	The desired result follows by taking supremum over $x \in B_0(k)$ above.
\end{proof}

\bibliographystyle{custom2author}
\bibliography{cpreg}

\end{document}